\documentclass{article}
\usepackage{amsmath,amsthm,amssymb,cases, bm}
\usepackage{indentfirst}
\usepackage{ascmac}
\usepackage{yhmath}
\usepackage{mathrsfs}
\usepackage{textcomp}
\usepackage{braket}
\usepackage{mathtools}
\usepackage{enumerate}
\usepackage{tikz-cd}
\tikzcdset{
  cells={font=\everymath\expandafter{\the\everymath\displaystyle}},
}
\usepackage[top=23truemm,bottom=25truemm,left=20truemm,right=20truemm]{geometry}
\usepackage{bookmark}

\newtheorem{definition}{Definition}[section]
\newtheorem{lemma}{Lemma}[section]
\newtheorem{theorem}[lemma]{Theorem}
\newtheorem{proposition}[lemma]{Proposition}
\newtheorem{corollary}[lemma]{Corollary}

\theoremstyle{definition}
\newtheorem{remark}{Remark}

\numberwithin{equation}{section}

\newcommand{\A}{\mathbb{A}}
\newcommand{\C}{\mathbb{C}}

\newcommand{\Q}{\mathbb{Q}}
\newcommand{\R}{\mathbb{R}}
\newcommand{\Z}{\mathbb{Z}}

\newcommand{\fin}{\mathit{fin}}

\newcommand{\GL}{\operatorname{GL}}
\newcommand{\SL}{\operatorname{SL}}

\newcommand{\SO}{\operatorname{SO}}

\newcommand{\GSp}{\operatorname{GSp}}
\newcommand{\PGSp}{\operatorname{PGSp}}

\newcommand{\Hom}{\operatorname{Hom}}

\newcommand{\Res}{\operatorname{Res}}
\newcommand{\Ind}{\operatorname{Ind}}
\newcommand{\cInd}{\operatorname{c-Ind}}

\newcommand{\RSold}{\mathrm{RSold}}
\newcommand{\RSnew}{\mathrm{RSnew}}
\newcommand{\KUold}{\mathrm{KUold}}
\newcommand{\KUnew}{\mathrm{KUnew}}
\newcommand{\Qold}{\mathcal{Q}\mathrm{old}}
\newcommand{\Qnew}{\mathcal{Q}\mathrm{new}}

\title{Kohnen--Ueda local newforms}
\author{Hiroshi Ishimoto\thanks{Osaka Metropolitan University}}
\date{[\today]}

\begin{document}

\maketitle

\begin{abstract}
    We give a local newform theory for the metaplectic group of rank 1 over a non-archimedean local field of characteristic 0 with odd residual characteristic.
    It was suggested by some results of Kohnen and Ueda on the classical theory of modular forms of half-integral weight.
\end{abstract}

\tableofcontents

\section{Introduction}
\subsection*{Background \& Main results}
In 1970, Atkin--Lehner \cite{al} introduced the notion of newforms for elliptic modular forms.
After that, Casselman \cite{cas} studied the Atkin--Lehner theory in the framework of local and global representations of $\GL_2$.
The theory of local newforms for $\GL_2$ given by Casselman in loc. cit. was generalized to the case for $\GL_N$ by Jacquet--Piatetski-Shapiro--Shalika \cite{jpss} and Atobe--Kondo--Yasuda \cite{aky}.
Moreover, the theory of local newforms has also been developed for several other groups.
For example, those for $\SL_2$ was studied by Lansky--Raghuram \cite{lr}, $\PGSp_4$ and $\widetilde{\SL_2}$ by Roberts--Schmidt \cite{rs-GSp4,rs-Mp2}, and $\SO_{2n+1}$ by Tsai \cite{tsai} and Cheng \cite{cheng}.

In \cite{rs-Mp2}, Roberts--Schmidt gave the number of local newforms for $\widetilde{\SL_2}$.
They defined the spaces of local newforms, and computed their dimensions without dealing with the conductors.
Let us review their theory briefly.
Let $F$ be a $p$-adic field whose residual characteristic is odd, $\mathcal{O}$ the ring of integers, and $\varpi$ a prime element.
The group $\widetilde{\SL_2}(F)$ is a unique nonlinear two-fold cover of $\SL_2(F)$ with the exact sequence
\begin{equation*}
    1 \longrightarrow \{\pm1\} \longrightarrow \widetilde{\SL_2}(F) \longrightarrow \SL_2(F) \longrightarrow 1.
\end{equation*}
Put $K_0=\SL_2(\mathcal{O})$, and let $K_m$ be the subgroup of $K_0$ consisting of matrices whose $(2,1)$-components belong to $\varpi^m\mathcal{O}$, for every integer $m\geq1$.
The exact sequence above is known to split uniquely over $K_0$, and hence $K_0$ can be regarded as a subgroup of $\widetilde{\SL_2}(F)$.
Let $(\pi,V)$ be an irreducible admissible (infinite-dimensional) genuine representation of $\widetilde{\SL_2}(F)$, where we say that a representation of $\widetilde{\SL_2}(F)$ is genuine if it does not factor through $\SL_2(F)$.
For any character $\eta$ of $\mathcal{O}^\times$, put
\begin{align*}
    \pi^{K_m}_\eta = \Set{v\in V | \pi\left( \left( \begin{array}{cc}a&b\\ c&d \end{array} \right) \right)v=\eta(d)v, \text{ for all } \left( \begin{array}{cc}a&b\\ c&d \end{array} \right)\in K_m}.
\end{align*}
Note that there is a natural inclusion from $\pi^{K_{m-1}}_\eta$ into $\pi^{K_m}_\eta$.
Roberts--Schmidt \cite{rs-Mp2} introduced a linear operator $\beta_2$ from $V$ to itself.
Note that it is denoted by $\alpha_2$ in \cite{rs-Mp2}.
It is an analogue of the Atkin--Lehner involution and sends $\pi^{K_{m-2}}_\eta$ to $\pi^{K_m}_\eta$.
They defined the subspace $\pi^{K_m, \RSold}_\eta$ of $\pi^{K_m}_\eta$ as the subspace spanned by these two images, and put
\begin{equation*}
    \pi^{K_m, \RSnew}_\eta = \pi^{K_m}_\eta / \pi^{K_m, \RSold}_\eta.
\end{equation*}
Then they showed that the sum
\begin{equation*}
    \sum_{m\geq 0} \dim_\C \pi^{K_m, \RSnew}_\eta
\end{equation*}
is finite and gave the value explicitly.

There are other studies after Atkin--Lehner theory.
Kohnen \cite{kohnew} introduced the notion of newforms of half-integral weight of odd squarefree level with some special characters, and proved that the space of those newforms corresponds to the space of newforms of integral weight defined by Atkin--Lehner via the Shimura correspondence \cite{shi, koh}.
Later, Ueda \cite{u1, u2, u3, u4} extended the theory of Kohnen to general odd level with some special characters.
However, the local newforms for $\widetilde{\SL_2}(F)$ which Roberts--Schmidt defined are not compatible with newforms of half-integral weight which Kohnen and Ueda defined, although the theory of Casselman is compatible with Atkin--Lehner theory.
Indeed, in the space $S_{k+\frac{1}{2}}(N,\chi)$ of cusp forms of weight $k+\frac{1}{2}$ of odd and squarefree level $N$ with a certain character, Kohnen defined the subspace of oldforms to be
\begin{equation*}
    \sum_{d|N, \ d<N} \left( S_{k+\frac{1}{2}}(d,\chi) + S_{k+\frac{1}{2}}(d,\chi)|U(N^2/d^2) \right),
\end{equation*}
and the space of newforms $S_{k+\frac{1}{2}}^{\mathrm{new}}(N,\chi)$ to be the orthogonal complement of the space of oldforms in $S_{k+\frac{1}{2}}(N,\chi)$.
Here, the nontrivial operator $U(N^2/d^2)$ sends a cusp form of level $d$ to that of level $N$.
In particular, in terms of representations of $p$-adic groups, it must send $\pi^{K_0}_\eta$ to $\pi^{K_1}_\eta$, since $N$ is squarefree.
However, the nontrivial operator $\beta_2$ introduced by Roberts--Schmidt sends $\pi^{K_0}_\eta$ to $\pi^{K_2}_\eta$.
The aim of this paper is to study the theory of Kohnen and Ueda in the framework of representation theory, and moreover to compare it with the result of Roberts--Schmidt.

In this paper, we define the space $\pi^{K_m, \KUnew}_\eta$ of local newforms for $\widetilde{\SL_2}(F)$ which are analogous to the theory of Kohnen and Ueda, and hence are different from those defined by Roberts--Schmidt.
Then we shall give an explicit formula for the conductor $c_\eta(\pi)$ (Theorem \ref{thm:conductor}), and show the following properties, which are the main theorems of this paper.
Here, $c_\eta(\pi)$ is the smallest integer $m$ such that $\pi^{K_m}_\eta$ is not zero.
\begin{proposition}[=Proposition {\ref{prop:compare-KU-RS}}]\label{prop:intro-cmpr}
    We have
    \begin{equation*}
        \pi^{K_m, \KUold}_\eta \supset \pi^{K_m, \RSold}_\eta
    \end{equation*}
    for every $m$.
\end{proposition}
\begin{theorem}[=Theorems {\ref{thm:main-nongeneric}} and {\ref{thm:main-generic}}]\label{thm:intro-main}
    Let $\pi$ be an irreducible genuine representation of $\widetilde{G}$, and $\eta$ a character of $\mathcal{O}^\times$.
    Fix a nontrivial additive character $\psi$ of $F$ of conductor 0.
    Assume that the central sign (Subsection \ref{subsec:conductors}) of $\pi$ with respect to $\psi$ is equal to $\eta(-1)$.
    \begin{enumerate}
        \item If $\pi$ is not $\psi_a$-generic (Subsection \ref{subsec:representations}) for any $a \in \mathcal{O}^\times$, then $\pi^{K_m, \KUnew}_\eta$ is zero for all $m$. Here, we write $\psi_a(x)=\psi(ax)$.
        \item Assume that $\pi$ is $\psi_a$-generic for some $a \in \mathcal{O}^\times$. Then $\pi^{K_m, \KUnew}_\eta$ is zero unless $m=c_\eta(\pi)$, in which case it has dimension 1 or 2. Moreover, in the case of dimension 2, we have
              \begin{equation*}
                  \pi^{K_{c_\eta(\pi)}, \KUnew}_\eta = \pi^{K_{c_\eta(\pi)}}_\eta
              \end{equation*}
              and it has a basis $\{v_+, v_-\}$ such that
              \begin{equation}\label{eq:intro-mult}
                  \mathcal{R}_\varpi v_\pm = \pm \gamma_F(\varpi, \psi) |\mathcal{O}/\varpi\mathcal{O}|^{-\frac{1}{2}} \operatorname{vol}(\mathcal{O}) v_\pm,
              \end{equation}
              where $\mathcal{R}_\varpi$ is a linear operator that will be defined in Subsection \ref{subsec:mainthm}, and $\gamma_F(\varpi, \psi)$ is the normalized Weil index.
    \end{enumerate}
\end{theorem}
In \cite{cas}, Casselman proved that any infinite-dimensional representation of $\GL_2(F)$ has one and only one newform up to a scalar multiple.
This property is sometimes referred to as multiplicity one for newforms, and is compatible with the Atkin--Lehner theory on newforms for elliptic modular forms.
Theorem \ref{thm:intro-main} says that in the case of metaplectic group, unlike the case of $\GL_2(F)$, the spaces of newforms $\pi^{K_m, \KUnew}_\eta$ is not one dimensional in general.
More precisely, its dimension is less than or equal to $2$.
Nevertheless, it is our substitute for multiplicity one for newforms that if it is 2-dimensional then there is a basis $\{v_+, v_-\}$ satisfying \eqref{eq:intro-mult}.

The theorem is proved by explicit calculation when the representation is a principal series representation or an even Weil representation, which has a good model.
For a Steinberg representation, which is a unique irreducible nontrivial subrepresentation of a principal series representation with the quotient isomorphic to an even Weil representation, combining the short exact sequence with the theorem for principal series representations and even Weil representations, we prove the theorem.
Then the remaining case is a supercuspidal representation.
The irreducible supercuspidal representations of $\widetilde{\SL_2}(F)$ can be constructed in parallel with those of $\SL_2(F)$.
Thanks to this fact, the theorem can be proved by transferring a similar theorem for $\SL_2(F)$.
The latter can be proved by realizing a supercuspidal representation of $\SL_2(F)$ as a subspace of the Kirillov model for a supercuspidal representation of $\GL_2(F)$.
While the arguments of Kohnen and Ueda were classical in the sense that they calculated some traces directly, our argument is purely representation theoretic.
Hence we can remove the technical condition on characters, which was imposed in their results.

\subsection*{Organization}
We shall first give some formulas on two variants of the Gauss sum and Jacobi sum in \S\ref{sec:Gausssum_Jacobisum}.
These variants appear in the proof of the main theorem for non-supercuspidal representations.
In \S\ref{subsec:groups} and \ref{subsec:representations}, we establish the notation for the metaplectic double covering of $\SL_2(F)$ and recall some known results on its representations.
Then we review the result of Roberts--Schmidt in \S\ref{subsec:conductors} and \ref{subsec:Roberts--Schmidt}.
In \S\ref{subsec:mainthm} we define the space of local newforms following Kohnen and Ueda, and state our main theorems.
In \S\ref{subsec:relation-KU} we will check that our level raising operators are compatible with those of Kohnen and Ueda.

We consider irreducible non-supercuspidal representations in \S\ref{sec:non-sc}, and supercuspidal ones in \S\ref{sec:sc}.
In the latter, we review the theories on constructions of supercuspidal representations before the proof of the main theorems for irreducible supercuspidal representations.

In \S\ref{sec:var-subgroup} we study local newforms associated to the other maximal compact subgroup.
Our $\{K_m\}$ is a sequence of subgroups of a maximal compact subgroup $K_0 \subset \SL_2(F)$, while $\SL_2(F)$ has two maximal compact subgroups up to conjugacy.
We will write $L_0$ for the other one, choose a sequence $\{L_m\}$ of subgroups, and study local newforms associated to it.

\subsection*{Convention \& Notation}
Let $F$ be a $p$-adic field, i.e., a non-archimedean local field of characteristic 0 with residual characteristic $p$.
In this paper, we assume that $p$ is an odd prime number.
Let $\mathcal{O}$ denote the ring of integers of $F$, and $\mathfrak{p}$ the maximal ideal of $\mathcal{O}$.
Fix a prime element $\varpi$ in $\mathcal{O}$.
Put $q=\#(\mathcal{O}/\mathfrak{p})$, and let $|-|=|-|_F$ be the valuation on $F$ normalized so that $|\varpi|=q^{-1}$.
Let $\operatorname{ord} \colon F\to \Z\cup\{\infty\}$ be a function defined by $|x|=q^{-\operatorname{ord}(x)}$ ($x\neq0$) and $\operatorname{ord}(0)=\infty$.
Notice that a subgroup $\mathcal{O}^{\times2}$ of $\mathcal{O}^\times$ has index 2, since $p$ is odd.
Fix an element $\xi$ in $\mathcal{O}^\times \setminus \mathcal{O}^{\times2}$.
We have $\mathcal{O}^\times=\mathcal{O}^{\times2}\sqcup \xi\mathcal{O}^{\times2}$.
Let $dx$ be any Haar measure on $F$ and fix it throughout this paper.

In this paper, a character means a continuous but not necessarily unitary topological group homomorphism to $\C^\times$.
For a nontrivial additive character $\psi \colon F\to \C^\times$, let $c(\psi)$ denote its conductor, i.e., the minimal integer $c$ such that $\psi$ is trivial on $\mathfrak{p}^c$.
Notice that any additive character of $F$ must be unitary.
For a nontrivial character $\eta \colon \mathcal{O}^\times \to \C^\times$, let $c(\eta)$ denote its conductor, i.e., the minimal positive integer $c$ such that $\eta$ is trivial on $1+\mathfrak{p}^c$.
For the trivial character of $\mathcal{O}^\times$, its conductor is defined to be $0$.
Notice also that any character of $\mathcal{O}^\times$ must be unitary.
For a character $\mu$ of $F^\times$, we write $c(\mu)=c(\mu|_{\mathcal{O}^\times})$ and call it the conductor of $\mu$.
Especially, $\mu$ is said to be unramified if $c(\mu)=0$.
We shall write $(-,-)_{F,2}$ for the quadratic Hilbert symbol over $F$, and for any $a\in F^\times$, let $\chi_a$ denote a quadratic or trivial character of $F^\times$ given by $x\mapsto (x,a)_{F,2}$.
It is known that the assignment $a\mapsto \chi_a$ gives a bijection between $F^\times/F^{\times2}$ and the set of quadratic or trivial characters of $F^\times$.

For a nontrivial additive character $\psi$ of $F$, let $\gamma_F(\psi)$ denote the unnormalized Weil index of $F\ni x \mapsto \psi(x^2)\in \C^\times$, which is an eighth root of unity in $\C^\times$.
For $a \in F^\times$, we shall define an additive character $\psi_a$ by $\psi_a(x)=\psi(ax)$, and put $\gamma_F(a,\psi)=\gamma_F(\psi_a)/\gamma_F(\psi)$, which we shall call the (normalized) Weil index.
Recall from \cite[Appendix]{rao} that
\begin{align*}
     & \gamma_F(ac^2,\psi)  =\gamma_F(a,\psi),                           &
     & \gamma_F(ab,\psi)   =\gamma_F(a,\psi)\gamma_F(b,\psi)(a,b)_{F,2}, & \\
     & \gamma_F(a,\psi_c)=(a,c)_{F,2} \gamma_F(a,\psi),                  &
     & \gamma_F(-1,\psi)  =\gamma_F(\psi)^{-2},                          &
\end{align*}
for any $a,b,c\in F^\times$.
They and \cite[Lemma 1.6]{szp} imply that
\begin{equation*}
    \gamma_F(ac, \psi) = \gamma_F(c,\psi),
\end{equation*}
for any $a\in\mathcal{O}^\times$ and $c\in F^\times$ such that the parity of $\operatorname{ord}(c)$ coincides with that of $c(\psi)$.

For an abstract group $G$ and its elements $g$ and $h$, we shall write $\prescript{g}{}{h}=ghg^{-1}$ and $h^g=g^{-1}hg$.
In addition, for a subset $X$ of $G$, we write $\prescript{g}{}{X}$ (resp. $X^g$) for the set of elements of the form $\prescript{g}{}{x}$ (resp. $x^g$), where $x\in X$.
Moreover, for a map $f$ from $X$ to any set, we write $\prescript{g}{}{f}$ (resp. $f^g$) for the map $\prescript{g}{}{X}\ni y\mapsto f(y^g)$ (resp. $X^g\ni y\mapsto f(\prescript{g}{}{y})$).

In this paper, a representation means a $\C$-representation, i.e., an action on a vector space over $\C$.
A representation of a topological group means a continuous representation.
Moreover, a representation of a locally profinite group means a smooth and admissible representation.
For a topological group $G$, its subgroup $H$, and a representation $(\sigma,W)$ of $H$, we write $\Ind^G_H(\sigma)$ for the induced representation, i.e., $\Ind^G_H(\sigma)=(\pi,V)$, where $V$ is the set of locally constant functions $f \colon G\to W$ such that $f(hg)=\sigma(h)f(g)$ for any $h\in H$ and $g\in G$, and $[\pi(g)f](x)=f(xg)$, for any $g,x\in G$.
We shall then write $\cInd^G_H(\sigma)$ for the subrepresentation of $\Ind^G_H(\sigma)$ consisting of the functions $f\in\Ind^G_H(\sigma)$ such that the support of $f$ is compact modulo $H$.
Conversely, for any representation $(\pi,V)$ of $G$, we write $\Res^G_H(\pi)$ or $\pi|_H$ for the restriction of $(\pi,V)$ to $H$.
For any group $G$, we shall write $\mathbf{1}_G$ for the trivial representation of $G$.
We simply write $\mathbf{1}$ if the group is clear from the context.

\subsection*{Acknowledgement}
The author would like to thank Shunsuke Yamana for many helpful comments.
In particular, he pointed out the gap between the formulation of Kohnen and that of Roberts--Schmidt, which motivated the author.
The author also would like to thank Hiraku Atobe for some helpful comments.
This work is supported by JSPS Research Fellowships for Young Scientists KAKENHI Grant Number 23KJ1824.

\section{Gauss sums and Jacobi sums}\label{sec:Gausssum_Jacobisum}
In this preliminary section, we calculate the Gauss sum, the Jacobi sum, and their quadratic variants.
These results play a role in the technical part of this paper.
\subsection{Gauss sum and quadratic variant}\label{subsec:Gausssum}
We first consider the Gauss sum and its variant.
For any nontrivial additive character $\psi$ of $F$ and any character $\chi$ of $F^\times$, we put
\begin{equation*}
    g(\chi,\psi)
    = \int_{\mathcal{O}^\times} \chi(x) \psi(x) dx,
\end{equation*}
and
\begin{equation*}
    h(\chi,\psi)
    = \int_{\mathcal{O}^\times} \chi(x) \psi(x^2) dx.
\end{equation*}
Recall from \cite[Section 2]{ish} the following properties on $g(\chi, \psi)$ and on $h(\chi, \psi)$.
Note that $h(\chi,\psi)=0$ if $\chi(-1)=-1$.
\begin{lemma}\label{lem:gausssum1}
    We have the followings.
    \begin{enumerate}
        \item If $c(\chi)=0$, then we have
              \begin{align*}
                  g(\chi,\psi)=\begin{cases*}
                                   \operatorname{vol}(\mathcal{O}^\times), & if $c(\psi)\leq0$, \\
                                   -\operatorname{vol}(\mathfrak{p}),      & if $c(\psi)=1$,    \\
                                   0,                                      & if $c(\psi)\geq2$.
                               \end{cases*}
              \end{align*}
        \item Assume that $c(\chi)\geq1$. Then
              \begin{align*}
                  |g(\chi,\psi)|=\begin{cases*}
                                     q^{-\frac{c(\psi)}{2}}\operatorname{vol}(\mathcal{O}), & if $c(\psi)=c(\chi)$,     \\
                                     0,                                                     & if $c(\psi)\neq c(\chi)$.
                                 \end{cases*}
              \end{align*}
              In particular, $g(\chi,\psi)\neq0$ if and only if $c(\chi)=c(\psi)$.
    \end{enumerate}
\end{lemma}
\begin{lemma}\label{lem:gausssum2}
    Assume that $\chi(-1)=+1$.
    \begin{enumerate}
        \item If $c(\psi)\leq 0$, then we have
              \begin{align*}
                  h(\chi,\psi)=\begin{cases*}
                                   \operatorname{vol}(\mathcal{O}^\times), & if $c(\chi)=0$,    \\
                                   0,                                      & if $c(\chi)\geq1$.
                               \end{cases*}
              \end{align*}
        \item Assume that $c(\psi) \geq 1$. Then we have
              \begin{align*}
                  |h(\chi,\psi)|^2 + |h(\chi,\psi_\xi)|^2
                  = \begin{cases*}
                        4q^{-c(\psi)}\operatorname{vol}(\mathcal{O})^2,     & if $c(\chi)=c(\psi)$,           \\
                        (2q^{-1}+2q^{-2})\operatorname{vol}(\mathcal{O})^2, & if $c(\chi)=0$ and $c(\psi)=1$, \\
                        0,                                                  & otherwise.
                    \end{cases*}
              \end{align*}
              In particular, at least one of $h(\chi,\psi)$ or $h(\chi,\psi_\xi)$ is not zero if and only if $c(\chi)=c(\psi)$ or $(c(\chi),c(\psi))=(0,1)$.
        \item Assume that $c(\psi)\geq 2$. Then we have
              \begin{align*}
                  (|h(\chi,\psi)|, |h(\chi,\psi_\xi)|)
                  = \begin{cases*}
                        (2q^{-\frac{c(\psi)}{2}}\operatorname{vol}(\mathcal{O}), 0) \text{ or } (0, 2q^{-\frac{c(\psi)}{2}}\operatorname{vol}(\mathcal{O})), & if $c(\chi)=c(\psi)$, \\
                        (0,0),                                                                                                                               & otherwise.
                    \end{cases*}
              \end{align*}
    \end{enumerate}
\end{lemma}

\subsection{Jacobi sum and quadratic variant}\label{subsec:Jacobisum}
Next we consider the Jacobi sum and its variants.
Let $\chi_1$ and $\chi_2$ be characters of $\mathcal{O}^\times$ with conductor $\leq1$ (i.e., characters of $\mathcal{O}^\times/1+\mathfrak{p}$), and we put
\begin{equation*}
    J(\chi_1, \chi_2) = \int_{\mathcal{O}^\times \cap 1+\mathcal{O}^\times} \chi_1(x) \chi_2(1-x) dx,
\end{equation*}
and
\begin{align*}
    I(\chi_1, \chi_2)  & = \int_{\mathcal{O}^\times \setminus \pm1+\mathfrak{p}} \chi_1(x) \chi_2(1-x^2) dx, \\
    I'(\chi_1, \chi_2) & = \int_{\mathcal{O}^\times} \chi_1(x) \chi_2(1-\xi x^2) dx.
\end{align*}
We first consider $J(\chi_1, \chi_2)$.
\begin{lemma}
    \begin{enumerate}
        \item We have $J(\chi_1, \chi_2)=J(\chi_2, \chi_1)$ and $J(1, 1)=(q-2)\operatorname{vol}(\mathfrak{p})$.
        \item If $\chi$ is nontrivial, then we have $J(\chi, 1)=-\operatorname{vol}(\mathfrak{p})$.
        \item If $\chi_1=\chi_2^{-1}=\chi$ is nontrivial, then we have $J(\chi, \chi^{-1})= -\chi(-1)\operatorname{vol}(\mathfrak{p})$.
        \item If $\chi_1\chi_2$ is nontrivial, then we have $J(\chi_1, \chi_2)=g(\chi_1)g(\chi_2)/g(\chi_1\chi_2)$, where $g(\chi)=g(\chi, \psi)$ and $\psi$ is any additive character of $F$ with $c(\psi)=1$.
    \end{enumerate}
\end{lemma}
\begin{proof}
    Note that
    \begin{align*}
        \int_{\mathcal{O}^\times} \chi(x) dx = \begin{dcases*}
                                                   0,                                      & if $\chi$ is nontrivial, \\
                                                   \operatorname{vol}(\mathcal{O}^\times), & if $\chi$ is trivial.
                                               \end{dcases*}
    \end{align*}
    (a) and (b) are easy.
    For (c), we have
    \begin{equation*}
        J(\chi, \chi^{-1}) = \int_{\mathcal{O}^\times \cap 1+\mathcal{O}^\times} \chi\left( \frac{x}{1-x} \right) dx.
    \end{equation*}
    Put $y=x/(1-x)$ to obtain
    \begin{align*}
        J(\chi, \chi^{-1}) = \int_{\mathcal{O}^\times \cap -1+\mathcal{O}^\times} \chi(y) dy = -\chi(-1)\operatorname{vol}(\mathfrak{p}).
    \end{align*}
    For (d), we have
    \begin{equation*}
        g(\chi_1) g(\chi_2) = \int_{\mathcal{O}^\times} \int_{\mathcal{O}^\times} \chi_1(x) \chi_2(y) \psi(x+y) dx dy.
    \end{equation*}
    Since
    \begin{equation*}
        \int_{\mathcal{O}^\times} \int_{-y+\mathfrak{p}} \chi_1(x) \chi_2(y) \psi(x+y) dx dy = \int_{\mathcal{O}^\times} \int_{-y+\mathfrak{p}} \chi_1(-y) \chi_2(y) dx dy = \chi_1(-1) \int_{-y+\mathfrak{p}} dx \int_{\mathcal{O}^\times} \chi_1\chi_2(y) dy = 0,
    \end{equation*}
    we obtain
    \begin{equation*}
        g(\chi_1) g(\chi_2) = \int_{\mathcal{O}^\times} \int_{\mathcal{O}^\times \setminus -y+\mathfrak{p}} \chi_1(x) \chi_2(y) \psi(x+y) dx dy.
    \end{equation*}
    Put $u=x+y$ and $t=x/u$. Then $y=u(1-t)$ and we have
    \begin{equation*}
        g(\chi_1) g(\chi_2) = \int_{\mathcal{O}^\times} \int_{\mathcal{O}^\times \cap 1+\mathcal{O}^\times} \chi_1(ut) \chi_2(u(1-t)) \psi(u) dt du = g(\chi_1\chi_2) J(\chi_1, \chi_2).
    \end{equation*}
\end{proof}
As a corollary, we can show the following equations.
\begin{lemma}\label{lem:square-1_Hilbertsymbol}
    For any $a \in \mathcal{O}^\times$, we have
    \begin{equation*}
        \int_{\mathcal{O}^\times \setminus \pm1+\mathfrak{p}} \left( x^2-1, \varpi \right)_{F,2} dx = - \operatorname{vol}(\mathfrak{p}) \left( 1 + (-1,\varpi)_{F,2} \right),
    \end{equation*}
    and
    \begin{equation*}
        \int_{\mathcal{O}^\times} \left( \xi x^2-1, \varpi \right)_{F,2} dx = \operatorname{vol}(\mathfrak{p}) \left( 1 - (-1,\varpi)_{F,2} \right).
    \end{equation*}
\end{lemma}
\begin{proof}
    Let $\chi_\varpi$ denote the nontrivial quadratic character $(-, \varpi)_{F,2}$ of $\mathcal{O}^\times$.
    For the first equation, we change the variables from $x$ to $u=(x+1)/2$.
    Then we have
    \begin{align*}
        \int_{\mathcal{O}^\times \setminus \pm1+\mathfrak{p}} \left( x^2-1, \varpi \right)_{F,2} dx
         & = \int_{(\mathcal{O}^\times \cap 1+\mathcal{O}^\times) \setminus \frac{1}{2}+\mathfrak{p}} (4u(u-1), \varpi)_{F,2} du       \\
         & = (-1, \varpi)_{F,2} \left( J(\chi_\varpi, \chi_\varpi) - \int_{\frac{1}{2}+\mathfrak{p}} (u(1-u), \varpi)_{F,2} du \right) \\
         & = - \operatorname{vol}(\mathfrak{p}) \left( 1+ (-1,\varpi)_{F,2} \right).
    \end{align*}
    On the other hand, we can show the second equation as follows.
    Since
    \begin{align*}
          & \int_{\mathcal{O}^\times} \left( \xi x^2-1, \varpi \right)_{F,2} dx + \int_{\mathcal{O}^\times \setminus \pm1+\mathfrak{p}} \left( x^2-1, \varpi \right)_{F,2} dx \\
        = & 2 \int_{\mathcal{O}^\times \setminus 1+\mathfrak{p}} (u-1, \varpi)_{F,2} du                                                                                       \\
        = & 2 \int_{\mathcal{O}^\times \setminus -1+\mathfrak{p}} (t, \varpi)_{F,2} dt                                                                                        \\
        = & -2\operatorname{vol}(\mathfrak{p}) (-1,\varpi)_{F,2},
    \end{align*}
    we have
    \begin{equation*}
        \int_{\mathcal{O}^\times} \left( \xi x^2-1, \varpi \right)_{F,2} dx = \left( -2\operatorname{vol}(\mathfrak{p}) (-1,\varpi)_{F,2} \right) - \left( - \operatorname{vol}(\mathfrak{p}) \left( 1+ (-1,\varpi)_{F,2} \right) \right) = \operatorname{vol}(\mathfrak{p}) \left( 1 - (-1,\varpi)_{F,2} \right).
    \end{equation*}
\end{proof}

Next we shall consider $I(\chi_1, \chi_2)$ and $I'(\chi_1, \chi_2)$.
\begin{lemma}\label{lem:jacobisum2}
    Assume that $\chi_1\chi_2^2 \neq 1$.
    Let $\psi$ be an additive character of $F$ with $c(\psi)=1$.
    Then we have
    \begin{equation*}
        2 h(\chi_1, \psi) g(\chi_2, \psi) = h(\chi_1\chi_2^2, \psi) I(\chi_1, \chi_2) + \chi_1\chi_2(\xi) h(\chi_1\chi_2^2, \psi_\xi) I'(\chi_1, \chi_2).
    \end{equation*}
    In particular, if moreover $\chi_1(-1)=+1$, then at least one of $I(\chi_1, \chi_2)$ or $I'(\chi_1, \chi_2)$ is nonzero.
\end{lemma}
\begin{proof}
    We have
    \begin{equation*}
        h(\chi_1, \psi) g(\chi_2, \psi) = \int_{\mathcal{O}^\times} \int_{\mathcal{O}^\times} \chi_1(x) \chi_2(y) \psi(x^2+y) dy dx.
    \end{equation*}
    Changing the variables $y \mapsto u=x^2+y$, we have
    \begin{equation*}
        h(\chi_1, \psi) g(\chi_2, \psi) = \int_{\mathcal{O}^\times} \int_{\mathcal{O} \setminus x^2+\mathfrak{p}} \chi_1(x) \chi_2(u-x^2) \psi(u) du dx.
    \end{equation*}
    The inner integral is equal to
    \begin{align*}
         & \int_{\mathfrak{p}} \chi_1(x) \chi_2(u-x^2) \psi(u) du + \int_{\mathcal{O}^{\times 2} \setminus x^2+\mathfrak{p}} \chi_1(x) \chi_2(u-x^2) \psi(u) du + \int_{\xi\mathcal{O}^{\times 2}} \chi_1(x) \chi_2(u-x^2) \psi(u) du                                          \\
         & = \operatorname{vol}(\mathfrak{p}) \chi_2(-1) \chi_1\chi_2^2(x) + \frac{1}{2} \int_{\mathcal{O}^\times \setminus \pm x+\mathfrak{p}} \chi_1(x) \chi_2(t^2-x^2) \psi(t^2) dt + \frac{1}{2} \int_{\mathcal{O}^\times} \chi_1(x) \chi_2(\xi t^2-x^2) \psi_\xi(t^2) dt.
    \end{align*}
    Thus $h(\chi_1, \psi) g(\chi_2, \psi)$ is equal to
    \begin{align*}
         & \frac{1}{2} \int_{\mathcal{O}^\times} \int_{\mathcal{O}^\times \setminus \pm x+\mathfrak{p}} \chi_1(x) \chi_2(t^2-x^2) \psi(t^2) dt dx + \frac{1}{2} \int_{\mathcal{O}^\times} \int_{\mathcal{O}^\times} \chi_1(x) \chi_2(\xi t^2-x^2) \psi_\xi(t^2) dt dx \\
         & = \frac{1}{2} \int_{\mathcal{O}^\times} \left( \int_{\mathcal{O}^\times \setminus \pm t+\mathfrak{p}} \chi_1(t^{-1} x) \chi_2(1-t^{-2} x^2) dx \right) \chi_1\chi_2^2(t) \psi(t^2) dt                                                                      \\
         & \qquad + \frac{1}{2} \int_{\mathcal{O}^\times} \left( \int_{\mathcal{O}^\times} \chi_1(t^{-1} x) \chi_2(\xi - t^{-2} x^2) dx \right) \chi_1\chi_2^2(t) \psi_\xi(t^2) dt                                                                                    \\
         & = \frac{1}{2} h(\chi_1\chi_2^2, \psi) I(\chi_1, \chi_2) + \frac{1}{2} \chi_1\chi_2(\xi) h(\chi_1\chi_2^2, \psi_\xi) I'(\chi_1, \chi_2).
    \end{align*}
    This proves the first assertion.
    The last assertion follows from this and Lemmas \ref{lem:gausssum1} and \ref{lem:gausssum2}.
\end{proof}

\section{Main theorem}\label{sec:mainthm}
In this section, we shall establish some notations for the metaplectic group and its subgroups, and review several notions and theories on its representations.
Then, in Subsection \ref{subsec:mainthm}, we shall state the main theorems of this paper.
\subsection{The metaplectic group}\label{subsec:groups}
Put $G=\SL_2(F)$, and let $\widetilde{G}=\widetilde{\SL_2}(F)$ be the nonlinear double cover of $G$.
As a set, we shall write $\widetilde{G}=G\times\{\pm1\}$ with the group law $(g_1,\epsilon_1)\cdot(g_2,\epsilon_2)=(g_1g_2,\epsilon_1\epsilon_2\bm{c}(g_1,g_2))$, where
\begin{align*}
    \bm{c}(g_1,g_2)                      =\left( \frac{\bm{x}(g_1)}{\bm{x}(g_1g_2)}, \frac{\bm{x}(g_2)}{\bm{x}(g_1g_2)} \right)_{F,2}, \\
    \bm{x}\left(\left( \begin{array}{cc}
                                   a & b \\c&d
                               \end{array} \right)\right) = \begin{cases*}
                                                        c, & if $c\neq0$, \\
                                                        d, & if $c=0$.
                                                    \end{cases*}
\end{align*}
The 2-cocycle $\bm{c}$ is called the Kubota cocycle.
For any subset $A\subset G$, we shall write $\widetilde{A}$ for the preimage of $A$ in $\widetilde{G}$.
A function $f\colon \widetilde{A} \to \C$ is said to be genuine if $f((a,-1))=-f((a,1))$ for any $a\in A$.
Analogously, when $A$ is a subgroup of $G$, a representation of $\widetilde{A}$ is said to be genuine if $(1_2, -1)$ acts as multiplication by $-1$.

Put
\begin{align*}
     & t(a)=\left( \begin{array}{cc}a & \\ &a^{-1} \end{array} \right), &  & n(b)=\left( \begin{array}{cc}1&b\\ &1\end{array} \right), &  & n^{\mathrm{op}}(b)=\left( \begin{array}{cc}1& \\ b&1\end{array} \right), &  & \beta=\left( \begin{array}{cc}1&\\ &\varpi\end{array} \right), &  & w=\left( \begin{array}{cc}&1\\ -1&\end{array} \right), &
\end{align*}
for $a\in F^\times$ and $b\in F$.
Put
\begin{align*}
     & T=\Set{t(a)|a\in F^\times}, &  & N=\Set{n(b)|b\in F}, &  & B=TN, &
\end{align*}
\begin{align*}
    K_0 & =\SL_2(\mathcal{O}),                                                                         \\
    K_m & =\Set{\left( \begin{array}{cc}a&b \\ c&d \end{array} \right) \in K_0 | c\in \mathfrak{p}^m},
\end{align*}
for $m \in \Z_{>0}$.
We shall regard $N$ as a subgroup of $\widetilde{G}$ via a splitting $N\ni n\mapsto(n,1)\ni \widetilde{G}$.
It is well-known that $K_0$ is a maximal compact subgroup of $G$, and splits the covering map $\widetilde{G}\to G$ uniquely.
We shall write $\gamma \mapsto(\gamma,\bm{s}(\gamma))$ the splitting.
The splitting is given explicitly by
\begin{align*}
    \bm{s}(n(b))               & =1, & b & \in\mathcal{O},        \\
    \bm{s}(t(a))               & =1, & a & \in\mathcal{O}^\times, \\
    \bm{s}(n^{\mathrm{op}}(c)) & =1, & c & \in\mathcal{O},        \\
    \bm{s}(w)                  & =1. &   &
\end{align*}
See \cite[Lemma 1.1]{hi} for the proof.
Via the splitting, we shall regard $K_0$ as a subgroup of $\widetilde{G}$, i.e., by abuse of notation we may simply write $\gamma$ for $(\gamma, \bm{s}(\gamma))$.
Here we have two lemmas.
\begin{lemma}\label{lem:spl_lv1}
    For
    \begin{align*}
        \gamma=\left( \begin{array}{cc}a&b\\ c&d \end{array} \right) \in K_0
    \end{align*}
    such that $d\in \mathcal{O}^\times$, we have
    \begin{align*}
        \bm{s}(\gamma) = \begin{cases*}
                             (c,d)_{F,2}, & if $c\neq 0$, \\
                             1,           & if $c=0$.
                         \end{cases*}
    \end{align*}
\end{lemma}
\begin{proof}
    Since
    \begin{equation*}
        \gamma = n(d^{-1} b) t(d^{-1}) n^{\mathrm{op}}(d^{-1} c),
    \end{equation*}
    we have
    \begin{align*}
        (\gamma, \bm{s}(\gamma))
         & = (n(d^{-1} b), 1) (t(d^{-1}), 1) (n^{\mathrm{op}}(d^{-1} c), 1) \\
         & = (\gamma, s),
    \end{align*}
    where $s$ is $(c,d)_{F,2}$ if $c\neq 0$ and $1$ otherwise.
\end{proof}

\begin{lemma}\label{lem:coset_ps}
    Let $m\geq0$ be an integer.
    We have a natural bijection
    \begin{equation*}
        K_m\backslash K_0/ B\cap K_0 \to K_m\backslash G/ B,\quad K_m g(B\cap K_0) \mapsto K_m g B.
    \end{equation*}
    Moreover, the set
    \begin{align*}
        \begin{cases*}
            \{1_2\},                                                                                            & when $m=0$,    \\
            \{1_2, w\},                                                                                         & when $m=1$,    \\
            \set{1_2, w} \cup \set{n^{\mathrm{op}}(\varpi^i), n^{\mathrm{op}}(\xi \varpi^i) | i=1,\ldots, m-1}, & when $m\geq2$,
        \end{cases*}
    \end{align*}
    is a complete representative system of the double coset spaces.
\end{lemma}
\begin{proof}
    The natural bijection of double coset spaces is implied by the Iwasawa decomposition $G=K_0B$.

    Now we consider the assertion on representatives.
    If $m=0$, the assertion is trivial.
    Suppose that $m\geq1$.
    Let $g=(g_{ij})\in K_0$.
    If $g_{21}\in\mathcal{O}^\times$, we have
    \begin{equation*}
        t(-g_{21})n(-g_{21}^{-1}g_{11}) \cdot g \cdot n(-g_{21}^{-1}g_{22})=w.
    \end{equation*}
    Since $t(-g_{21})n(-g_{21}^{-1}g_{11})\in K_m$ and $n(-g_{21}^{-1}g_{22})\in B\cap K_0$, this means that
    \begin{equation*}
        K_m g (B\cap K_0) = K_m w (B\cap K_0).
    \end{equation*}
    Suppose that $g_{21}\in \mathfrak{p}$.
    If $g_{21}\in \mathfrak{p}^m$, then $g\in K_m$ and we have
    \begin{equation*}
        K_m g (B\cap K_0) = K_m 1_2 (B\cap K_0).
    \end{equation*}
    If $g_{21}\in \mathfrak{p}\setminus \mathfrak{p}^m$, then $g_{11} \in\mathcal{O}^\times$, and there exist $\varepsilon\in\{0,1\}$, $a\in\mathcal{O}^\times$, and $1\leq i \leq m-1$ such that $g_{21}g_{11}^{-1}=\xi^\varepsilon a^2\varpi^i$.
    Then we have
    \begin{equation*}
        t(a) \cdot g \cdot n(-g_{11}^{-1}g_{12}) t(g_{11}^{-1}) t(a^{-1}) = n^{\mathrm{op}}(\xi^\varepsilon \varpi^i),
    \end{equation*}
    which means that
    \begin{equation*}
        K_m g (B\cap K_0) = K_m n^{\mathrm{op}}(\xi^\varepsilon \varpi^i) (B\cap K_0).
    \end{equation*}
    Thus we have
    \begin{align*}
        K_0= K_m w (B\cap K_0) \cup K_m 1_2 (B\cap K_0) \cup \bigcup_{i=1}^{m-1} \left( K_m n^{\mathrm{op}}(\varpi^i) (B\cap K_0) \cup K_m n^{\mathrm{op}}(\xi\varpi^i) (B\cap K_0) \right).
    \end{align*}
    If $m=1$, it is of course understood to be a union of the former two cosets.
    Note that if the (2,1)-component of $g\in K_0$ is an element of $\varpi^i \mathcal{O}^\times$ and $0\leq i\leq m-1$, then so is that of every element in $K_m g (B\cap K_0)$.
    Hence we have
    \begin{align*}
        K_0= K_m w (B\cap K_0) \sqcup K_m 1_2 (B\cap K_0) \sqcup \bigsqcup_{i=1}^{m-1} \left( K_m n^{\mathrm{op}}(\varpi^i) (B\cap K_0) \cup K_m n^{\mathrm{op}}(\xi\varpi^i) (B\cap K_0) \right).
    \end{align*}
    Now, it remains to show that if $1\leq i \leq m-1$, then $n^{\mathrm{op}}(\varpi^i)$ and $n^{\mathrm{op}}(\xi\varpi^i)$ are not contained in a same coset.
    Suppose that they are.
    Then there exist
    \begin{equation*}
        \left( \begin{array}{cc}a&b\\ c&d \end{array} \right)\in K_m \quad \text{and} \quad \left( \begin{array}{cc}x&y\\ &x^{-1} \end{array} \right)\in B\cap K_0,
    \end{equation*}
    such that
    \begin{align*}
        \left( \begin{array}{cc}a&b\\ c&d \end{array} \right) n^{\mathrm{op}}(\varpi^i) = n^{\mathrm{op}}(\xi\varpi^i) \left( \begin{array}{cc}x&y\\ &x^{-1} \end{array} \right).
    \end{align*}
    In other words, there exist $a, d, x\in\mathcal{O}^\times$, $b, y\in \mathcal{O}$, and $c\in\mathfrak{p}^m$, such that
    \begin{align*}
        \left\{
        \begin{aligned}
             & ad-bc=1,                  \\
             & a+\varpi^ib=x,            \\
             & b=y,                      \\
             & c+\varpi^id=\xi\varpi^ix, \\
             & d=\xi\varpi^iy+x^{-1}.
        \end{aligned}
        \right.
    \end{align*}
    Combining the assumption $1\leq i \leq m-1$, this implies that $\xi\in d^2+\mathfrak{p}$.
    Since the residual characteristic $p$ is odd, we have $d^2 + \mathfrak{p} \subset \mathcal{O}^{\times 2}$.
    This is a contradiction.
\end{proof}

\subsection{Representations of the metaplectic group}\label{subsec:representations}
Next, we shall recall some notions and properties on genuine representations of $\widetilde{G}$.

Let $(\pi,V)$ be an irreducible admissible genuine representation of $\widetilde{G}$ and $\psi$ any nontrivial additive character of $F$.
For $a \in F^\times$, the representation $(\pi,V)$ is said to be $\psi_a$-generic if
\begin{equation*}
    \Hom_N(\pi,\psi_a) \neq\{0\},
\end{equation*}
i.e., there exists a nonzero linear map $\ell_a \colon V\to \C$ such that $\ell_a(\pi(n(u))v)=\psi(au)\ell_a(v)$, for any $u\in F$ and $v\in V$.
Note that $\pi$ is $\psi_a$-generic if and only if it is $\psi_{ab}$-generic for all $b \in F^{\times 2}$.
Any nonzero element $\ell_a\in \Hom_N(\pi,\psi_a)$ is called a $\psi_a$-Whittaker functional, and it is proven in \cite{wal1, wal2} that the space $\Hom_N(\pi,\psi_a)$ of $\psi_a$-Whittaker functionals is at most one dimensional.
It is also known that any irreducible genuine representation of $\widetilde{G}$ is $\psi_a$-generic for some $a \in F^\times$.

The representation $(\pi,V)$ is said to be supercuspidal if for any $v\in V$ there exists a compact open subgroup $N(v)$ of $N$ such that
\begin{equation*}
    \int_{N(v)} \pi(n)v dn=0,
\end{equation*}
where $dn$ is a Haar measure on $N(v)$.
If $(\pi,V)$ is supercuspidal and $\psi_a$-generic for a nonzero element $a \in F^\times$, then exactly one of the followings holds.
\begin{itemize}
    \item It is $\psi_b$-generic if and only if $ab^{-1} \in F^{\times 2}$.
    \item It is $\psi_b$-generic if and only if $ab^{-1} \in \varpi^{2\Z}\mathcal{O}^\times$.
    \item It is $\psi_b$-generic if and only if $ab^{-1} \in \varpi^\Z \mathcal{O}^{\times 2}$.
\end{itemize}
For more details, see Lemma \ref{lem:sc-generic} and \S\ref{subsec:sc-constr}, where we shall review a classification of irreducible supercuspidal representations of $\widetilde{G}$.

Now we shall review the Weil representations of $\widetilde{G}$.
For every nontrivial additive character $\psi$ of $F$, let $(\omega_\psi, \mathcal{S}(F))$ be the Schr\"{o}dinger model of the Weil representation of $\widetilde{G}$ associated to $\psi$, i.e., $\mathcal{S}(F)$ is the space of locally constant and compactly supported functions on $F$ and the action $\omega_\psi$ is determined by
\begin{align*}
    [\omega_\psi((t(a),\epsilon))\cdot\varphi](y) & =\epsilon|a|^{\frac{1}{2}}\gamma_F(a,\psi)^{-1} \varphi(ay), \\
    [\omega_\psi(n(b)) \cdot \varphi](y)          & =\psi(by^2)\varphi(y),                                       \\
    [\omega_\psi((w,1))\cdot\varphi](y)           & =\gamma_F(\psi)\int_F \varphi(x)\psi(2xy)d_{\psi_2} x,
\end{align*}
where $d_{\psi_2}x$ denotes the self-dual measure on $F$ with respect to $\psi_2$.
Let $\mathcal{S}^+(F)$ (resp. $\mathcal{S}^-(F)$) be the subspace of $\mathcal{S}(F)$ consisting of all even (resp. odd) functions, i.e., the functions $\varphi\in\mathcal{S}(F)$ such that $\varphi(-x)=\varphi(x)$ (resp. $\varphi(-x)=-\varphi(x)$).
Then a decomposition $\mathcal{S}(F)=\mathcal{S}^+(F)\oplus\mathcal{S}^-(F)$ gives the decomposition of the Weil representation, which we shall write $\omega_\psi=\omega_\psi^+ \oplus \omega_\psi^-$.
The representation $\omega_\psi^+$ (resp. $\omega_\psi^-$) is called the even (resp. odd) Weil representation.
The even and odd Weil representations are known to be irreducible.
Note that $\omega_{\psi_a}\cong\omega_{\psi_b}$ if and only if $ab^{-1} \in F^{\times 2}$.
Thus, for any quadratic or trivial character $\chi=\chi_a$, we shall write $\omega_{\psi,\chi}=\omega_{\psi_a}$, $\omega_{\psi,\chi}^+=\omega_{\psi_a}^+$, and $\omega_{\psi,\chi}^-=\omega_{\psi_a}^-$, which are uniquely determined up to equivalence.
We note:
\begin{itemize}
    \item $\omega_{\psi_a}^\pm$ is $\psi_b$-generic if and only if $ab^{-1}\in F^{\times 2}$;
    \item $\omega_{\psi_a}^+$ is not supercuspidal, and $\omega_{\psi_a}^-$ is supercuspidal;
    \item $\omega_{\psi_a}^\pm$ is characterized by these properties.
\end{itemize}

Next we review the principal series representations of $\widetilde{G}$.
We have a genuine character $\chi_\psi$ of $\widetilde{T}$ defined by
\begin{align*}
     & \chi_\psi((t(a),\epsilon))=\epsilon\gamma_F(a,\psi)^{-1}, &  & a\in F^\times,\ \epsilon\in\{\pm1\}. &
\end{align*}
Given a character $\mu$ of $F^\times$, we shall write $\pi_\psi(\mu)$ for the representation parabolically induced from $\widetilde{B}$ by $\mu \cdot \chi_\psi$, and realize it on the space of locally constant functions $f\colon \widetilde{G}\to \C$ such that
\begin{equation*}
    f((t(a),\epsilon)(n(b),1)g)=|a|\mu(a)\chi_\psi((t(a),\epsilon))f(g),
\end{equation*}
for any $a\in F^\times$, $\epsilon\in\{\pm1\}$, $b\in F$, and $g\in\widetilde{G}$.
The representation is defined by $[g\cdot f](h)=f(hg)$.
The reducibility property of these representations are summarized as follows.
\begin{proposition}
    \begin{enumerate}
        \item The representation $\pi_\psi(\mu)$ is irreducible if and only if $\mu^2 \neq |-|^{\pm1}$. In this case we have $\pi_\psi(\mu)\cong\pi_\psi(\mu^{-1})$.
        \item If $\mu=\chi \cdot |-|^{\frac{1}{2}}$, where $\chi$ is a quadratic or trivial character, then $\pi_\psi(\mu)$ is reducible and there exists an irreducible representation $\mathit{St}_{\psi,\chi}$ such that we have a short exact sequence
              \begin{equation}\label{eq:SESstd}
                  0 \to \mathit{St}_{\psi,\chi} \to \pi_\psi(\mu) \to \omega_{\psi,\chi}^+ \to 0.
              \end{equation}
              We shall call the representation $\mathit{St}_{\psi,\chi}$ the Steinberg representation associated to $\psi$ and $\chi$.
        \item If $\mu= \chi \cdot |-|^{-\frac{1}{2}}$, where $\chi$ is a quadratic or trivial character, then $\pi_\psi(\mu)$ is reducible and there is a short exact sequence
              \begin{equation}\label{eq:SESdual}
                  0 \to \omega_{\psi,\chi}^+ \to \pi_\psi(\mu) \to \mathit{St}_{\psi,\chi} \to 0.
              \end{equation}
    \end{enumerate}
\end{proposition}
It is known that the proposition gives all the irreducible non-supercuspidal genuine representations of $\widetilde{G}$.

\subsection{Definition of conductors}\label{subsec:conductors}
Let $(\pi,V)$ be an admissible representation of $\widetilde{G}$.
Consider an element $-1_\psi=(-1_2, \gamma_F(-1,\psi))$ in $\widetilde{G}$ for any nontrivial additive character $\psi$ of $F$.
It is central in $\widetilde{G}$ and has order 2.
Hence it acts on $V$ by $\pm1$.
We define the central sign $z_\psi(\pi)\in\{\pm1\}$ with respect to $\psi$ by $\pi(-1_\psi)=z_\psi(\pi)1_V$.

In the rest of this section, we let $\psi$ be a nontrivial additive character of $F$ of conductor 0.
Let $\eta$ be any character of $\mathcal{O}^\times$.
When $m\geq c(\eta)$, $\eta$ gives a character on $K_m$ defined by
\begin{align*}
    \left( \begin{array}{cc}a&b \\ c&d \end{array} \right) \mapsto \eta(d),
\end{align*}
for which we shall again write $\eta$.
Let $(\pi,V)$ be an admissible representation of $\widetilde{G}$.
Then we put
\begin{align*}
    \pi^{K_m}_\eta
    =\begin{cases*}
         \Set{v\in V | \pi(\gamma)v=\eta(\gamma)v,\ \text{for all $\gamma\in K_m$}}, & if $m \geq c(\eta)$, \\
         \{0\},                                                                      & if $m<c(\eta)$.
     \end{cases*}
\end{align*}
Here, notice that $\pi^{K_m}_\eta$ is defined to be the zero space if $m$ is negative although $K_m$ is not defined then.
Let $\eta$ be a character of $\mathcal{O}^\times$ such that $\eta(-1)=z_\psi(\pi)$.
We put
\begin{equation*}
    c_\eta(\pi)
    =\min(m\geq0 \mid \pi^{K_m}_\eta \neq \{0\}).
\end{equation*}
In this paper, we study $\pi^{K_m}_\eta$ and $c_\eta(\pi)$ for irreducible genuine representations $\pi$ of $\widetilde{G}$.
Before we finish this subsection, set
\begin{equation*}
    \pi^{K_\infty}_\eta= \bigcup_{m \in \Z} \pi^{K_m}_\eta.
\end{equation*}
Notice that we do not define any set denoted by $K_\infty$.

\subsection{Review on Roberts--Schmidt's local newform}\label{subsec:Roberts--Schmidt}
Here we shall recall the result of Roberts--Schmidt \cite{rs-Mp2}.
We continue to take $\psi$ to be a nontrivial additive character of $F$ of conductor $0$.
For any admissible representation $(\pi, V)$ of $\widetilde{G}$, Roberts--Schmidt defined a linear automorphism $\beta_2$ on $V$ by
\begin{equation*}
    v \mapsto \beta_2 v = \pi(\left( \begin{array}{cc}\varpi^{-1}&\\ &\varpi \end{array} \right),1) v.
\end{equation*}
Let $\eta$ be a character of $\mathcal{O}^\times$ such that $\eta(-1)=z_\psi(\pi)$.
It follows immediately that $\beta_2(\pi^{K_m}_\eta) \subset \pi^{K_{m+2}}_\eta$.
In addition, $\pi^{K_m}_\eta$ is a subspace of $\pi^{K_{m+1}}_\eta$.
Following \cite{rs-Mp2}, we define the subspace $\pi^{K_m, \RSold}_\eta$ of $\pi^{K_m}_\eta$ as the linear subspace generated by $\pi^{K_{m-1}}_\eta$ and $\beta_2(\pi^{K_{m-2}}_\eta)$.
Then we shall define
\begin{equation*}
    \pi^{K_m, \RSnew}_\eta = \pi^{K_m}_\eta / \pi^{K_m, \RSold}_\eta.
\end{equation*}
Let $F_\psi(\pi)$ be the set of $a$ in $F^\times$ such that $\pi$ is $\psi_a$-generic.
The group $F^{\times 2}$ acts on $F_\psi(\pi)$, and the number $\# F_\psi(\pi) / F^{\times 2}$ is given by Waldspurger as follows.
\begin{proposition}\label{prop:num-generic}
    Let $\pi$ be an irreducible genuine representation of $\widetilde{G}$.
    Then we have
    \begin{align*}
        \# F_\psi(\pi) / F^{\times 2}
        =\begin{cases*}
             1, & if $\pi$ is an even or odd Weil representation,               \\
             2, & if $\pi$ is supercuspidal but not an odd Weil representation, \\
             3, & if $\pi$ is a Steinberg representation,                       \\
             4, & if $\pi$ is an irreducible principal series representation.
         \end{cases*}
    \end{align*}
\end{proposition}
Then, Roberts--Schmidt \cite{rs-Mp2} showed the following theorem.
\begin{theorem}\label{thm:rs-Mp2}
    Let $\pi$ be an irreducible genuine representation of $\widetilde{G}$, and $\eta$ a character of $\mathcal{O}^\times$ such that $\eta(-1)=z_\psi(\pi)$.
    Then $\pi^{K_m, \RSnew}_\eta$ is finite dimensional for all $m\geq0$, and zero-dimensional for all but finitely many $m\geq 0$, and we have
    \begin{equation*}
        \sum_{m=0}^\infty \dim_\C \pi^{K_m, \RSnew}_\eta = \# F_\psi(\pi) / F^{\times 2}.
    \end{equation*}
\end{theorem}

\subsection{Main theorem}\label{subsec:mainthm}
Now we define another space of local newforms and then state the main theorems of this paper.
We continue to take $\psi$ to be a nontrivial additive character of $F$ of conductor $0$.
For an irreducible genuine representation $(\pi,V)$ of $\widetilde{G}$, we define two linear operators $\mathcal{U}_{\varpi^2}$ and $\mathcal{R}_\varpi$ on $\pi^{K_\infty}_\eta$ by
\begin{align*}
    \mathcal{U}_{\varpi^2} v & = \int_{\mathcal{O}} \pi\left( \left( \left( \begin{array}{cc}\varpi&\varpi^{-1} u\\ 0&\varpi^{-1} \end{array} \right), 1 \right) \right) v du,          \\
    \mathcal{R}_\varpi v     & = \int_{\mathcal{O}^\times} (u,\varpi)_{F,2} \pi\left( \left( \left( \begin{array}{cc}1&\varpi^{-1} u\\ 0&1 \end{array} \right), 1 \right) \right) v du.
\end{align*}
Before introducing the space of local newforms, we first check the most important and basic properties of them.
\begin{proposition}\label{prop:U-operator}
    The operator $\mathcal{U}_{\varpi^2}$ satisfies the following properties:
    \begin{enumerate}
        \item $\mathcal{U}_{\varpi^2}( \pi^{K_0}_\eta ) \subset \pi^{K_1}_\eta$;
        \item $\mathcal{U}_{\varpi^2}( \pi^{K_m}_\eta ) \subset \pi^{K_m}_\eta$ for $m \geq 1$.
    \end{enumerate}
\end{proposition}
\begin{proof}
    Since $\pi^{K_0}_\eta \subset \pi^{K_1}_\eta$, the first property follows from the second one.
    Now we shall show (b).
    Let $m\geq1$.
    If $\pi^{K_m}_\eta=\{0\}$, the assertion is trivial.
    We therefore assume that $m \geq c(\eta)$.
    Let $v \in \pi^{K_m}_\eta$ and $\gamma=\left( \begin{array}{cc}a&b\\ c&d \end{array} \right) \in K_m$.
    A calculation shows that
    \begin{align*}
        (\gamma, \bm{s}(\gamma)) \left( \left( \begin{array}{cc}\varpi&\varpi^{-1} u\\ 0&\varpi^{-1} \end{array} \right), 1 \right)
         & = \left( \left( \begin{array}{cc} \varpi & \varpi^{-1} \dfrac{au+b}{cu+d} \\\\ & \varpi^{-1} \end{array} \right), 1 \right) \left( \gamma', \bm{s}(\gamma) \right),
    \end{align*}
    for any $u \in \mathcal{O}$, where
    \begin{align*}
        \gamma' = \left( \begin{array}{cc} \dfrac{1}{cu+d}&\\ \\ \varpi^2 c& cu+d \end{array} \right) \in K_m.
    \end{align*}
    Since $\bm{s}(\gamma)=\bm{s}(\gamma')$ by Lemma \ref{lem:spl_lv1}, we then have
    \begin{align*}
        \pi(\gamma) \mathcal{U}_{\varpi^2} v
         & = \int_\mathcal{O} \pi\left( \left( \left( \begin{array}{cc} \varpi & \varpi^{-1} \dfrac{au+b}{cu+d} \\\\ & \varpi^{-1} \end{array} \right), 1 \right) \right) \pi(\gamma') v du \\
         & = \int_\mathcal{O} \eta(cu+d) \pi\left( \left( \left( \begin{array}{cc} \varpi & \varpi^{-1} \dfrac{au+b}{cu+d} \\\\ & \varpi^{-1} \end{array} \right), 1 \right) \right) v du.
    \end{align*}
    Since $c$ is an element of $\mathfrak{p}^m$ and $\eta$ is trivial on $1+\mathfrak{p}^m$, we have $\eta(cu+d)=\eta(d)$.
    Changing variables from $u$ to $t=\frac{au+b}{cu+d}$, the last integral is then equal to
    \begin{align*}
        \int_\mathcal{O} \eta(d) \pi\left( \left( \left( \begin{array}{cc} \varpi & \varpi^{-1} t \\ & \varpi^{-1} \end{array} \right), 1 \right) \right) v dt = \eta(d) \mathcal{U}_{\varpi^2} v.
    \end{align*}
    This completes the proof.
\end{proof}
Let $\mathcal{R}_1$ be a linear operator on $\pi^{K_\infty}_\eta$ defined by
\begin{equation*}
    \mathcal{R}_1 v = \operatorname{vol}(\mathcal{O}^\times) v - \int_{\mathcal{O}^\times} \pi\left( \left( \left( \begin{array}{cc}1&\varpi^{-1}u\\ &1 \end{array} \right), 1 \right) \right) v du.
\end{equation*}
\begin{proposition}\label{prop:R-operator}
    The operator $\mathcal{R}_\varpi$ satisfies the following properties:
    \begin{enumerate}
        \item $\beta_2( \pi^{K_m}_\eta ) \subset \operatorname{Ker}(\mathcal{R}_\varpi)$, for any integer $m$;
        \item $\mathcal{R}_\varpi(\pi^{K_m}_\eta) \subset \pi^{K_{m'}}_\eta$, for any integer $m$, where $m'=\max(2, c(\eta)+1, m)$;
        \item $\mathcal{R}_\varpi^3 = (-1,\varpi)_{F,2} q^{-1} \operatorname{vol}(\mathcal{O})^2 \mathcal{R}_\varpi$, as operators on $\pi^{K_\infty}_\eta$.
    \end{enumerate}
\end{proposition}
\begin{proof}
    Suppose $\pi^{K_m}_\eta \neq \{0\}$, since the assertion is trivial otherwise. For any $v \in \pi^{K_m}_\eta$, we have
    \begin{align*}
        \mathcal{R}_\varpi \circ \beta_2 (v)
         & =\int_{\mathcal{O}^\times} (u,\varpi)_{F,2} \pi\left( \left( \left( \begin{array}{cc}1&\varpi^{-1} u\\ &1 \end{array} \right), 1 \right) \right) \circ \pi\left( \left( \left( \begin{array}{cc}\varpi^{-1}&\\ &\varpi \end{array} \right), 1 \right) \right) (v) du \\
         & =\int_{\mathcal{O}^\times} (u,\varpi)_{F,2} \pi\left( \left( \left( \begin{array}{cc}\varpi^{-1}&\\ &\varpi \end{array} \right), 1 \right) \right) \circ \pi\left( \left( \left( \begin{array}{cc}1&\varpi u\\ &1 \end{array} \right), 1 \right) \right) (v) du      \\
         & =\int_{\mathcal{O}^\times} (u,\varpi)_{F,2} du \cdot \beta_2 v                                                                                                                                                                                                       \\
         & =0.
    \end{align*}
    Hence (a) holds.
    Moreover, let $\gamma=\left( \begin{array}{cc}a&b\\ c&d \end{array} \right)$ be an element in $K_{m'}$.
    Note that $m' \geq 2$.
    Then for any $v \in \pi^{K_m}_\eta$,
    \begin{align*}
        \pi(\gamma) \circ \mathcal{R}_\varpi (v)
         & = \int_{\mathcal{O}^\times} (u,\varpi)_{F,2} \pi((\gamma, \bm{s}(\gamma))) \circ \pi\left( \left( \left( \begin{array}{cc}1&\varpi^{-1} u\\ &1 \end{array} \right), 1 \right) \right) (v) du       \\
         & =\int_{\mathcal{O}^\times} (u,\varpi)_{F,2} \pi\left( \left( \left( \begin{array}{cc}1&\varpi^{-1} ad^{-1}u\\ &1 \end{array} \right) \right) \right) \circ \pi((\gamma', \bm{s}(\gamma'))) (v) du,
    \end{align*}
    where
    \begin{align*}
        \gamma'= \left( \begin{array}{cc}a-\varpi^{-1}ad^{-1}c&b-\varpi^{-2}ad^{-1}cu^2\\ c&\varpi^{-1}cu+d \end{array} \right).
    \end{align*}
    Since $\gamma' \in K_{m'} \subset K_m$ and $\varpi^{-1}cu \in \mathfrak{p}^{m'-1} \subset \mathfrak{p}^{c(\eta)}$ for any $u \in \mathcal{O}^\times$, we have $\pi(\gamma')v=\eta(d)v$.
    Changing the variables $u \mapsto t=ad^{-1}u$, we see that the last integral is equal to
    \begin{equation*}
        \eta(d) \mathcal{R}_\varpi v,
    \end{equation*}
    and hence (b) holds.
    Here note that $( a^{-1}d, \varpi )_{F,2}=1$ since $ad \in 1+\mathfrak{p}$.

    We now consider the assertion (c).
    For $v \in \pi^{K_\infty}_\eta$, we have
    \begin{equation*}
        \mathcal{R}_\varpi^2 v = \int_{\mathcal{O}^\times} \int_{\mathcal{O}^\times} ( ut, \varpi )_{F,2} \pi\left( \left( \left( \begin{array}{cc}1&\varpi^{-1}(u+t)\\ &1 \end{array} \right), 1 \right) \right) v dt du.
    \end{equation*}
    Changing the variables $t \mapsto x=1+u^{-1}t$ first, and then $u \mapsto s=xu$ next, we see that this equals
    \begin{align*}
         & \int_{x \in 1+\mathcal{O}^\times} \int_{s\in x\mathcal{O}^\times} ( x-1, \varpi )_{F,2} \pi\left( \left( \left( \begin{array}{cc}1&\varpi^{-1}s\\ &1 \end{array} \right), 1 \right) \right) v \frac{ds}{|x|} dx                                                                                  \\
         & =( -1, \varpi )_{F,2} J(\mathbf{1}, \chi_\varpi) \int_{\mathcal{O}^\times} \pi\left( \left( \left( \begin{array}{cc}1&\varpi^{-1}s\\ &1 \end{array} \right), 1 \right) \right) v ds + \int_{x \in \mathfrak{p}} ( x-1, \varpi )_{F,2} \int_{s \in x\mathcal{O}^\times} \frac{ds}{|x|} dx \cdot v \\
         & =-( -1, \varpi )_{F,2} \operatorname{vol}(\mathfrak{p}) \int_{\mathcal{O}^\times} \pi\left( \left( \left( \begin{array}{cc}1&\varpi^{-1}s\\ &1 \end{array} \right), 1 \right) \right) v ds + ( -1, \varpi )_{F,2} \operatorname{vol}(\mathfrak{p}) \operatorname{vol}(\mathcal{O}^\times) v      \\
         & = ( -1, \varpi )_{F,2} \operatorname{vol}(\mathfrak{p}) \mathcal{R}_1 v.
    \end{align*}
    We also have
    \begin{equation*}
        \mathcal{R}_\varpi \circ \mathcal{R}_1 (v)
        =\operatorname{vol}(\mathcal{O}^\times) \mathcal{R}_\varpi v - \int_{\mathcal{O}^\times} \int_{\mathcal{O}^\times} ( u, \varpi )_{F,2} \pi\left( \left( \left( \begin{array}{cc}1&\varpi^{-1} (u+x)\\ &1 \end{array} \right), 1 \right) \right) v dx du .
    \end{equation*}
    If $x \in -u+\mathfrak{p}$, then $\varpi^{-1} (u+x)$ is in $\mathcal{O}$  and hence
    \begin{align*}
         & \int_{u \in \mathcal{O}^\times} \int_{x \in -u+\mathfrak{p}} ( u, \varpi )_{F,2} \pi\left( \left( \left( \begin{array}{cc}1&\varpi^{-1} (u+x)\\ &1 \end{array} \right), 1 \right) \right) v dx du \\
         & =\int_{u \in \mathcal{O}^\times} \int_{x \in -u+\mathfrak{p}} ( u, \varpi )_{F,2} v dx du                                                                                                         \\
         & =0.
    \end{align*}
    If $x \in \mathcal{O}^\times \setminus -u+\mathfrak{p}$, changing the variables $x \mapsto y=u+x$, we obtain
    \begin{align*}
         & \int_{u \in \mathcal{O}^\times} \int_{x \in \mathcal{O}^\times \setminus -u+\mathfrak{p}} ( u, \varpi )_{F,2} \pi\left( \left( \left( \begin{array}{cc}1&\varpi^{-1} (u+x)\\ &1 \end{array} \right), 1 \right) \right) v dx du \\
         & =\int_{u \in \mathcal{O}^\times} ( u, \varpi )_{F,2} \int_{y \in \mathcal{O}^\times \setminus u+\mathfrak{p}} \pi\left( \left( \left( \begin{array}{cc}1&\varpi^{-1} y\\ &1 \end{array} \right), 1 \right) \right) v dy du     \\
         & =-\int_{u \in \mathcal{O}^\times} ( u, \varpi )_{F,2} \int_{y \in u+\mathfrak{p}} \pi\left( \left( \left( \begin{array}{cc}1&\varpi^{-1} y\\ &1 \end{array} \right), 1 \right) \right) v dy du                                 \\
         & =-\operatorname{vol}(\mathfrak{p}) \int_{u \in \mathcal{O}^\times} ( u, \varpi )_{F,2} \pi\left( \left( \left( \begin{array}{cc}1&\varpi^{-1} u\\ &1 \end{array} \right), 1 \right) \right) v du                               \\
         & =-\operatorname{vol}(\mathfrak{p}) \mathcal{R}_\varpi v.
    \end{align*}
    Thus we have
    \begin{equation*}
        \mathcal{R}_\varpi \circ \mathcal{R}_1 (v)
        =\left( \operatorname{vol}(\mathcal{O}^\times) + \operatorname{vol}(\mathfrak{p}) \right) \mathcal{R}_\varpi v
        =\operatorname{vol}(\mathcal{O}) \mathcal{R}_\varpi v.
    \end{equation*}
    Therefore, we conclude that
    \begin{equation*}
        \mathcal{R}_\varpi^3
        = ( -1, \varpi )_{F,2} \operatorname{vol}(\mathfrak{p}) \mathcal{R}_\varpi \circ \mathcal{R}_1
        = ( -1, \varpi )_{F,2} \operatorname{vol}(\mathfrak{p}) \operatorname{vol}(\mathcal{O}) \mathcal{R}_\varpi.
    \end{equation*}
    This completes the proof.
\end{proof}

We can now define the spaces of oldforms and newforms.
\begin{definition}\label{def:local-KU}
    Let $\pi$ be an irreducible genuine representation of $\widetilde{G}$, and $\eta$ a character of $\mathcal{O}^\times$.
    For $m \geq 0$, define a subspace $\pi^{K_m, \KUold}_\eta$ of $\pi^{K_m}_\eta$ by
    \begin{align*}
        \pi^{K_m, \KUold}_\eta =\begin{dcases*}
                                    \operatorname{Ker}(\mathcal{R}_\varpi|_{\pi^{K_0}_\eta}),                                                                            & if $m=0$,      \\
                                    \pi^{K_0}_\eta + \mathcal{U}_{\varpi^2}\left( \pi^{K_0}_\eta \right) + \operatorname{Ker}(\mathcal{R}_\varpi|_{\pi^{K_1}_\eta}),     & if $m=1$,      \\
                                    \pi^{K_{m-1}}_\eta + \mathcal{R}_\varpi\left( \pi^{K_{m-1}}_\eta \right) + \operatorname{Ker}(\mathcal{R}_\varpi|_{\pi^{K_m}_\eta}), & if $m \geq 2$.
                                \end{dcases*}
    \end{align*}
    Then we define
    \begin{equation*}
        \pi^{K_m, \KUnew}_\eta = \pi^{K_m}_\eta / \pi^{K_m, \KUold}_\eta.
    \end{equation*}
\end{definition}
The subspace $\pi^{K_m, \KUold}_\eta$ can be defined uniformly.
See Remark \ref{rem:uni-def}.

\begin{proposition}\label{prop:compare-KU-RS}
    We have
    \begin{equation*}
        \pi^{K_m, \KUold}_\eta \supset \pi^{K_m, \RSold}_\eta
    \end{equation*}
    for every $m$.
    In particular, we have
    \begin{equation*}
        \dim_\C \pi^{K_m, \KUnew}_\eta \leq \dim_\C \pi^{K_m, \RSnew}_\eta.
    \end{equation*}
\end{proposition}
\begin{proof}
    The proposition follows from Proposition \ref{prop:R-operator} (a).
\end{proof}

Our first theorem is an explicit formula of the conductor $c_\eta(\pi)$.
This was shown in \cite{ish} partially, but in this paper we will get it in all cases.
\begin{theorem}\label{thm:conductor}
    Let $\pi$ be an irreducible genuine representation of $\widetilde{G}$, and $\eta$ a character of $\mathcal{O}^\times$.
    If $\eta(-1) \neq z_\psi(\pi)$, then the subspace $\pi^{K_m}_\eta$ is zero for every $m$.

    Assume that $\eta(-1) = z_\psi(\pi)$.
    If $\pi$ is not supercuspidal, then the conductor $c_\eta(\pi)$ is equal to
    \begin{equation*}
        \begin{cases*}
            c(\mu) + \min(c(\eta\mu), c(\eta\mu^{-1})), & if $\pi = \pi_\psi(\mu)$ and $c(\eta) \leq c(\mu)$,                          \\
            c(\eta\mu) + c(\eta\mu^{-1}),               & if $\pi = \pi_\psi(\mu)$ and $c(\eta) > c(\mu)$,                             \\
            2c(\eta\chi) + c(\chi),                     & if $\pi=\omega_{\psi, \chi}^+$,                                              \\
            1,                                          & if $\pi=\mathit{St}_{\psi,\chi}$ and $\eta = \chi|_{\mathcal{O}^\times}$,    \\
            2c(\eta\chi),                               & if $\pi=\mathit{St}_{\psi,\chi}$ and $\eta \neq \chi|_{\mathcal{O}^\times}$,
        \end{cases*}
    \end{equation*}
    where $\mu$ is a character of $F^\times$, and $\chi$ a quadratic or trivial character of $F^\times$.
    If $\pi$ is supercuspidal, then as will be explained in Subsection \ref{subsec:sc-constr}, it can be written as $\cInd^{\widetilde{G}}_{\widetilde{J}}(\widetilde{\lambda})$, and the conductor $c_\eta(\pi)$ is equal to
    \begin{equation*}
        \begin{cases*}
            2 \max(c(\lambda), c(\eta)),     & if $d(\lambda)=0$, and $\pi$ is $\psi_a$-generic for some $a \in \mathcal{O}^\times$,    \\
            1 + 2 \max(c(\lambda), c(\eta)), & if $d(\lambda)=0$, and $\pi$ is not $\psi_a$-generic for any $a \in \mathcal{O}^\times$, \\
            \max(2c(\lambda)-1, 2c(\eta)),   & if $d(\lambda)=1$,
        \end{cases*}
    \end{equation*}
    where $J = J_0$ or $J_1$, and $\lambda$ is an irreducible strongly cuspidal representation of $J$, with notation and terminology of Section \ref{sec:sc}.
\end{theorem}
\begin{proof}
    If $\eta(-1) \neq z_\psi(\pi)$, it is immediate from definition that $\pi^{K_m}_\eta$ is zero for every $m$.
    Consider the case that $\eta(-1) = z_\psi(\pi)$.
    If $\pi$ is not supercuspidal, the formula follows from \cite[Theorems 4.1, 4.4, and 4.7]{ish}.
    If $\pi$ is supercuspidal, it follows from Propositions \ref{prop:sc-unram2}, \ref{prop:sc-unram4}, and \ref{prop:sc-ram}.
\end{proof}

In this paper, moreover, we shall calculate the dimension of the space of our local newforms, and show the uniqueness property.
\begin{theorem}\label{thm:main-nongeneric}
    Let $\pi$ be an irreducible genuine representation of $\widetilde{G}$, and $\eta$ a character of $\mathcal{O}^\times$ such that $\eta(-1) = z_\psi(\pi)$.
    Assume that $\pi$ is neither $\psi$-generic nor $\psi_\xi$-generic.
    Then we have
    \begin{equation*}
        \pi^{K_m, \KUnew}_\eta = \{0\}
    \end{equation*}
    for all $m$.
\end{theorem}
\begin{proof}
    If $\pi$ is not supercuspidal, then by Proposition \ref{prop:num-generic}, $\pi$ must be an even Weil representation.
    Hence the assertion follows from Proposition \ref{prop:evenWeil}.
    If $\pi$ is supercuspidal, then by the results of Kutzko--Sally (which will be reviewed in Subsection \ref{subsec:sc-constr}) and Lemma \ref{lem:sc-generic}, $\pi$ must come from a strongly cuspidal representation of defect 1.
    Hence the assertion follows from Propositions \ref{prop:sc-unram2} and \ref{prop:sc-unram4}.
\end{proof}
\begin{theorem}\label{thm:main-generic}
    Let $\pi$ be an irreducible genuine representation of $\widetilde{G}$, and $\eta$ a character of $\mathcal{O}^\times$ such that $\eta(-1) = z_\psi(\pi)$.
    Assume that $\pi$ is $\psi_a$-generic for some $a \in \mathcal{O}^\times$.
    Then $\pi^{K_m, \KUnew}_\eta$ is zero unless $m = c_\eta(\pi)$.
    Moreover, the following holds.
    If $\pi$ is supercuspidal, we use the notation and terminology introduced in Section \ref{sec:sc}.
    \begin{enumerate}[(1)]
        \item In one of the following cases, the space $\pi^{{K_m}, \KUnew}_\eta$ is 1-dimensional for $m = c_\eta(\pi)$.\begin{itemize}
                  \item $\pi = \pi_\psi(\mu)$ and $\eta = \mu^{\pm1}|_{\mathcal{O}^\times}$.
                  \item $\pi$ is an even or odd Weil representation.
                  \item $\pi = \mathit{St}_{\psi, \chi}$ and $\eta = \chi|_{\mathcal{O}^\times}$.
                  \item $\pi = \cInd^{\widetilde{G}}_{\widetilde{J}}(\widetilde{\lambda})$ with $d(\lambda)=1$.
              \end{itemize}
        \item Consider the other cases, namely\begin{itemize}
                  \item $\pi = \pi_\psi(\mu)$ and $\eta \neq \mu^{\pm1}|_{\mathcal{O}^\times}$;
                  \item $\pi = \mathit{St}_{\psi, \chi}$ and $\eta \neq \chi|_{\mathcal{O}^\times}$;
                  \item $\pi = \cInd^{\widetilde{G}}_{\widetilde{J}}(\widetilde{\lambda})$ with $d(\lambda)=0$ but not isomorphic to even Weil representations.
              \end{itemize}
              Then for $m = c_\eta(\pi)$, the space $\pi^{{K_m}, \KUold}_\eta$ is zero, and $\pi^{{K_m}, \KUnew}_\eta = \pi^{K_m}_\eta$ is 2-dimensional with a basis $\{v_+, v_-\}$ satisfying
              \begin{equation*}
                  \mathcal{R}_\varpi v_\pm = \pm \gamma_F(\varpi, \psi) q^{-\frac{1}{2}} \operatorname{vol}(\mathcal{O}) v_\pm.
              \end{equation*}
    \end{enumerate}
\end{theorem}
\begin{proof}
    The theorem follows from Propositions \ref{prop:unram_case}, \ref{prop:quad-ps_case}, \ref{prop:general-min-conductor-ps_case}, \ref{prop:general-non-min-ps_case}, \ref{prop:evenWeil}, \ref{prop:Steinberg}, \ref{prop:sc-unram2}, \ref{prop:sc-unram4}, and \ref{prop:sc-ram}.
\end{proof}
In particular, the metaplectic group $\widetilde{\SL_2}(F)$ satisfies a principle that an irreducible representation has local newforms if and only if the representation is generic, and that the local newforms are unique in some sense.

\subsection{Relation with the classical theory}\label{subsec:relation-KU}
Before proving the theorems, we shall give a rough explanation for the relation with the theory of Kohnen and Ueda.

Let $k$ be a non-negative integer, $N$ a positive odd integer, and $\chi$ an quadratic or trivial Dirichlet character modulo $4N$ such that $\chi(-1)=+1$.
For such $k$, $N$, and $\chi$, let $S_{k+\frac{1}{2}}^+(4N, \chi)$ denote the plus space of cusp forms of weight $k+\frac{1}{2}$, of level $\Gamma_0(4N)$, and with character $\chi$.
For a positive integer $m$ and an odd prime number $p$, the operators $\delta_m$, $U(m)$, and $R_p$ are defined as follows:
\begin{align*}
     & f|\delta_m (z) = f(mz),                                                     \\
     & f|U(m) (z) = \sum_{n \geq 1} a_{mn} e^{2 \pi i n z},                        \\
     & f|R_p (z) = \sum_{n \geq 1} a_n \left( \frac{n}{p} \right) e^{2 \pi i n z},
\end{align*}
for $f(z) = \sum_{n\geq 1} a_n e^{2 \pi i n z} \in S_{k+\frac{1}{2}}^+(4N, \chi)$.
Here $\left( \dfrac{}{p} \right)$ denotes the Legendre symbol.
We write $\mathfrak{H}$ for the upper half plane in this subsection.

From a cusp form $f \in S_{k+\frac{1}{2}}^+(4N, \chi)$, we can construct a cuspidal automorphic form $\varphi_f$ on $\SL_2(\Q) \backslash \widetilde{\SL_2}(\A_\Q)$ in the following way.
We write $\A_\Q$ for the ring of adeles of $\Q$.
Recall from \cite[Section 8]{hi} the definition of $\widetilde{\SL_2}(\A_\Q)$.
For a finite set $\mathfrak{S}$ of places of $\Q$ that contains $2$ and $\infty$, put
\begin{equation*}
    \SL_2(\A_\Q)_\mathfrak{S} = \prod_{v \in \mathfrak{S}} \SL_2(\Q_v) \times \prod_{v \notin \mathfrak{S}} \SL_2(\Z_v).
\end{equation*}
Let $\widetilde{\SL_2}(\A_\Q)_\mathfrak{S}$ be the double cover of $\SL_2(\A_\Q)_\mathfrak{S}$ defined by the 2-cocycle $\prod_{v \in \mathfrak{S}} \bm{c}_v$, where $\bm{c}_v$ denotes the Kubota cocycle for $\SL_2(\Q_v)$.
For two such finite sets $\mathfrak{S}_1 \subset \mathfrak{S}_2$, there is an embedding $\widetilde{\SL_2}(\A_\Q)_{\mathfrak{S}_1} \hookrightarrow \widetilde{\SL_2}(\A_\Q)_{\mathfrak{S}_2}$ given by
\begin{equation*}
    ((g_v)_v, \epsilon) \mapsto \left( (g_v)_v, \epsilon \prod_{v \in \mathfrak{S}_2 \setminus \mathfrak{S}_1} \bm{s}_v(g_v) \right),
\end{equation*}
where we write $\bm{s}_v$ for the splitting $\bm{s}$ over $\SL_2(\Z_v)$ given in \S\ref{subsec:groups} for every odd prime number $v$.
The global metaplectic group $\widetilde{\SL_2}(\A_\Q)$ is defined to be the direct limit of $\{\widetilde{\SL_2}(\A_\Q)_\mathfrak{S}\}_\mathfrak{S}$.
We have a short exact sequence
\begin{equation}\label{eq:cov-global}
    1 \to \{\pm1\} \to \widetilde{\SL_2}(\A_\Q) \to \SL_2(\A_\Q) \to 1.
\end{equation}
The covering is known to split canonically over $\SL_2(\Q)$, and thus $\SL_2(\Q)$ is regarded as a subgroup of $\widetilde{\SL_2}(\A_\Q)$.
The embedding is given by $x \mapsto (x,1)$ for sufficiently large $\mathfrak{S}$.
For any prime number $l$, put
\begin{equation*}
    \Gamma_0(4N)_l = \Set{ \left( \begin{array}{cc} a&b\\c&d \end{array} \right) \in \SL_2(\Z_l) | c \in 4N\Z_l},
\end{equation*}
and let $\tilde{\Gamma}_0(4N)_l$ denote the preimage of $\Gamma_0(4N)_l$ under the covering map $\widetilde{\SL_2}(\Q_l) \to \SL_2(\Q_l)$.
Let $\Gamma_0(4N)_\fin$ be the direct product of $\Gamma_0(4N)_l$ over all prime numbers $l$, and $\tilde{\Gamma}_0(4N)_\fin$ denote the preimage of $\Gamma_0(4N)_\fin$ under the covering map \eqref{eq:cov-global}.
We have a decomposition
\begin{equation*}
    \tilde{\Gamma}_0(4N)_\fin = \tilde{\Gamma}_0(4N)_2 \times \prod_{l\neq2} \Gamma_0(4N)_l.
\end{equation*}
Let $\psi_\Q = \otimes_v \psi_{\Q_v}$ be the additive character of $\Q \backslash \A_\Q$ such that $\psi_{\Q_\infty}$ sends $x \in \R$ to $e^{2 \pi i x} \in \C$.
We shall write $\varepsilon_2$ for the genuine character of $\tilde{\Gamma}_0(4N)_2$ given in \cite[Lemma 1.1]{hi} with respect to $\psi_{\Q_2}$.
Also let $j_\infty$ denote the factor of automorphy on $\widetilde{\SL_2}(\R) \times \mathfrak{H}$ given in \cite[Section 7]{hi}.
The strong approximation theorem tells us that
\begin{equation*}
    \widetilde{\SL_2}(\A_\Q) = \SL_2(\Q) \widetilde{\SL_2}(\R) \tilde{\Gamma}_0(4N)_\fin.
\end{equation*}
Then $\varphi_f$ is given by
\begin{equation*}
    \varphi_f(g) = \chi(\kappa) \left( j_\infty(g_\infty, i) \varepsilon_2(\kappa_2) \right)^{-2k-1} f(g_\infty i),
\end{equation*}
for $g=g_\Q g_\infty \kappa \in \widetilde{\SL_2}(\A_\Q)$, where $g_\Q \in \SL_2(\Q)$, $g_\infty \in \widetilde{\SL_2}(\R)$, and $$\kappa=(\kappa_l)_l \in \tilde{\Gamma}_0(4N)_\fin = \tilde{\Gamma}_0(4N)_2 \times \prod_{l\neq2} \Gamma_0(4N)_l.$$
Here, $\chi(\kappa)$ denotes the image of $\kappa$ under
\begin{align*}
    \tilde{\Gamma}_0(4N)_\fin \to \Gamma_0(4N)_\fin \overset{\mod 4N}{\longrightarrow} \prod_l B(\Z_l/4N\Z_l) \overset{(2,2)-component}{\longrightarrow} \prod_l (\Z_l/4N\Z_l)^\times \cong (\Z/4N\Z)^\times \overset{\chi}{\to} \C^\times,
\end{align*}
where $B$ denotes the subgroup of $\SL_2$ consisting of upper triangular matrices.
Let us write $\pi_f=\otimes_v \pi_{f,v}$ for the automorphic representation of $\widetilde{\SL_2}(\A_\Q)$ generated by $\varphi_f$.

For any odd prime number $p$, we have $\pi_f(\kappa_p)\varphi_f = \chi(\kappa_p) \varphi_f$ for any $\kappa_p \in \Gamma_0(4N)_p$.
In particular, we can consider $\mathcal{R}_p \varphi_f$ and $\mathcal{U}_{p^2} \varphi_f$.
\begin{proposition}\label{prop:relation-KU}
    Let $p$ be an odd prime number.
    Then for any $f(z) = \sum_{n \geq 1} a_n e^{2 \pi i n z} \in S_{k+\frac{1}{2}}^+(4N, \chi)$, we have
    \begin{align*}
        \mathcal{R}_p \varphi_f     & = \operatorname{vol}(p\Z_p) \bm{G}_p \varphi_{f|R_p}, \\
        \mathcal{U}_{p^2} \varphi_f & = \operatorname{vol}(p^2\Z_p) C_p \varphi_{f|U(p^2)},
    \end{align*}
    where
    \begin{align*}
        \bm{G}_p & = \sum_{y \in \Z/p\Z} \left( \frac{y}{p} \right)  e^{-2 \pi i p^{-1}y},                   \\
        C_p      & = \chi^{\{p\}}(p) p^{-k+\frac{3}{2}} (p,p)_{\Q_2,2} \gamma_{\Q_2}(p, \psi_{\Q_2})^{2k+3},
    \end{align*}
    and $\chi^{\{p\}}$ is the composition
    \begin{equation*}
        \prod_{l \neq p} (\Z_l/4N\Z_l)^\times \overset{1 \times \mathrm{id}}{\longrightarrow} (\Z_p/4N\Z_p) \times \prod_{l \neq p} (\Z_l/4N\Z_l)^\times \cong (\Z/4N\Z) \overset{\chi}{\to} \C^\times.
    \end{equation*}
\end{proposition}
\begin{proof}
    Fix $g=g_\Q g_\infty \kappa \in \widetilde{\SL_2}(\A_\Q)$ such that $g_\Q \in \SL_2(\Q)$, $g_\infty \in \widetilde{\SL_2}(\R)$, and $\kappa \in \tilde{\Gamma}_0(4N)_\fin$.
    Moreover, put $z = g_\infty i \in \mathfrak{H}$, and write $\kappa=(\kappa_l)_l$, where $\kappa_2 \in \tilde{\Gamma}_0(4N)_2$ and $\kappa_l \in \Gamma_0(4N)_l$ ($l \neq 2$).
    Furthermore, write
    \begin{equation*}
        \kappa_p = \left( \begin{array}{cc}a&b\\ c&d \end{array} \right).
    \end{equation*}
    Since $S_{k+\frac{1}{2}}^+(4N, \chi) \subset S_{k+\frac{1}{2}}^+(4Np, \chi)$, we may assume $N$ is divisible by $p^2N'$, where $N'$ is the conductor of $\chi$.
    Fix a sufficiently large finite set $\mathfrak{S}$ of places containing $p$, $2$, and $\infty$ so that $g_\Q$ is $(g_\Q, 1)$ in $\widetilde{\SL_2}(\A_\Q)_\mathfrak{S}$.
    Then we can see that every element of $\widetilde{\SL_2}(\A_\Q)$ which appears in the following arguments is in $\widetilde{\SL_2}(\A_\Q)_\mathfrak{S}$, and that every element $x$ of $\SL_2(\Q)$ which appears in the following arguments is of the form $(x,1)$ in $\widetilde{\SL_2}(\A_\Q)_\mathfrak{S}$.

    By definition, $\mathcal{R}_p \varphi_f (g)$ is given by
    \begin{equation*}
        \int_{\Z_p^\times} (u,p)_{\Q_p, 2} \varphi_f\left( g \left( \left( \begin{array}{cc}1&p^{-1}u\\ &1 \end{array} \right), 1 \right)_p \right) du,
    \end{equation*}
    where the subscript $p$ indicates that the element is a member of the $p$-component $\widetilde{\SL_2}(\Q_p)$ in $\widetilde{\SL_2}(\A_\Q)$.
    Since the integrand is determined by $u + p\Z_p$, this equals
    \begin{align*}
          & \operatorname{vol}(p\Z_p) \sum_{u \in (\Z_p/p\Z_p)^\times} (u,p)_{\Q_p, 2} \varphi_f\left( g \left( \left( \begin{array}{cc}1&p^{-1}u\\ &1 \end{array} \right), 1 \right)_p \right) \\
        = & \operatorname{vol}(p\Z_p) \sum_{u \in \Z/p\Z} \left( \frac{u}{p} \right) \varphi_f\left( g \left( \left( \begin{array}{cc}1&p^{-1}u\\ &1 \end{array} \right), 1 \right)_p \right).  \\
    \end{align*}
    For each $u \in \Z/p\Z$, let $x$ be an element in $\Z$ such that $ua+xd \in p\Z$, and put
    \begin{align*}
        \kappa_p' & =(n(p^{-1}x), 1) \kappa_p (n(p^{-1}u), 1), &  &                                                        \\
        g_\infty' & = (n(p^{-1}x), 1) g_\infty,                &  &                                                        \\
        \kappa_l' & = (n(p^{-1}x), 1) \kappa_l,                &  & \text{for $l \in \mathfrak{S}\setminus\{p, \infty\}$}, \\
        \kappa_l' & = n(p^{-1}x) \kappa_l,                     &  & \text{for $l \notin \mathfrak{S}$},                    \\
        \kappa'   & = (\kappa_l')_{l\text{: all primes}},      &  &                                                        \\
        g_\Q'     & = g_\Q n(-p^{-1}x)                         &  & .
    \end{align*}
    Note that $\kappa'_l \in \Gamma_0(4N)_l$ for $l \neq 2$.
    Then some calculations show that $g (n(p^{-1}u), 1)_p$ equals $g_\Q' g_\infty' \kappa'$.
    We can therefore write
    \begin{equation*}
        \varphi_f\left( g \left( \left( \begin{array}{cc}1&p^{-1}u\\ &1 \end{array} \right), 1 \right)_p \right)
        = \varphi_f \left( g_\infty' \kappa' \right),
    \end{equation*}
    and that this is equal to
    \begin{equation*}
        \chi(\kappa) \left( j_\infty(g_\infty, i) \varepsilon_2(\kappa_2) \right)^{-2k-1} f(z+p^{-1}x).
    \end{equation*}
    Hence we have
    \begin{equation*}
        \mathcal{R}_p \varphi_f (g)
        = \operatorname{vol}(p\Z_p) \chi(\kappa) \left( j_\infty(g_\infty, i) \varepsilon_2(\kappa_2) \right)^{-2k-1} \sum_{u \in \Z/p\Z} \left( \frac{u}{p} \right) f(z+p^{-1}x).
    \end{equation*}
    Since $ad \in 1+ p\Z_p \subset \Z_p^{\times2}$, the last sum over $u$ equals
    \begin{equation*}
        \sum_{x \in \Z/p\Z} \left( \frac{-x}{p} \right) f(z+p^{-1}x)
        = \sum_{n \geq 1} a_n e^{2 \pi i n z} \sum_{x \in \Z/p\Z} \left( \frac{-x}{p} \right)  e^{2 \pi i n p^{-1}x}.
    \end{equation*}
    If $n$ is a multiple of $p$, then the inner sum over $x$ vanishes.
    Otherwise, by changing the variables $x \mapsto y=-nx$, one can see that it is
    \begin{equation*}
        \left( \frac{n}{p} \right)\sum_{y \in \Z/p\Z} \left( \frac{y}{p} \right)  e^{-2 \pi i p^{-1}y}.
    \end{equation*}
    Consequently, we have
    \begin{equation*}
        \mathcal{R}_p \varphi_f (g)
        = \operatorname{vol}(p\Z_p) \bm{G}_p \chi(\kappa) \left( j_\infty(g_\infty, i) \varepsilon_2(\kappa_2) \right)^{-2k-1} f|R_p(z)
        = \operatorname{vol}(p\Z_p) \bm{G}_p \varphi_{f|R_p}(g).
    \end{equation*}

    Next we calculate $\mathcal{U}_{p^2} \varphi_f (g)$.
    It is given by
    \begin{equation*}
        \int_{\Z_p} \varphi_f\left( g \left( \left( \begin{array}{cc}p&p^{-1}u\\ &p^{-1} \end{array} \right), 1 \right)_p \right) du.
    \end{equation*}
    Since the integrand is determined by $u + p^2\Z_p$, this equals
    \begin{equation*}
        \operatorname{vol}(p^2\Z_p) \sum_{u \in \Z_p/p^2\Z_p} \varphi_f\left( g \left( \left( \begin{array}{cc}p&p^{-1}u\\ &p^{-1} \end{array} \right), 1 \right)_p \right).
    \end{equation*}
    For each $u \in \Z_p/p^2\Z_p$, let $x$ be an element in $\Z$ such that $ua+b+xd \in p^2\Z_p$, and put
    \begin{align*}
        \kappa_p' & = \left( \left( \begin{array}{cc}p^{-1}&p^{-1}x\\ &p \end{array} \right), (p,p)_{\Q_p,2} \right) \kappa_p \left( \left( \begin{array}{cc}p&p^{-1}u\\ &p^{-1} \end{array} \right), 1 \right), &                                                        & \\
        g_\infty' & = \left( \left( \begin{array}{cc}p^{-1}&p^{-1}x\\ &p \end{array} \right), (p,p)_{\R,2} \right) g_\infty,                                                                                     &                                                        & \\
        \kappa_l' & = \left( \left( \begin{array}{cc}p^{-1}&p^{-1}x\\ &p \end{array} \right), (p,p)_{\Q_l,2} \right) \kappa_l,                                                                                   & \text{for $l \in \mathfrak{S}\setminus \{p,\infty\}$}, & \\
        \kappa_l' & = \left( \begin{array}{cc}p^{-1}&p^{-1}x\\ &p \end{array} \right) \kappa_l,                                                                                                                  & \text{for $l \notin \mathfrak{S}$},                    & \\
        \kappa'   & = (\kappa_l')_{\text{$l$: all primes}},                                                                                                                                                      &                                                        & \\
        g_\Q'     & =g_\Q\left( \begin{array}{cc}p^{-1}&p^{-1}x\\ &p \end{array} \right)^{-1}.                                                                                                                   &                                                        &
    \end{align*}
    By the product formula for the Hilbert symbol, we have
    \begin{equation*}
        g \left( \left( \begin{array}{cc}p&p^{-1}u\\ &p^{-1} \end{array} \right), 1 \right)_p
        = g_\Q' g_\infty' \kappa'.
    \end{equation*}
    Notice that $\kappa_l' \in \Gamma_0(4N)_l$ for $l \neq 2$.
    Now we have
    \begin{align*}
          & \varphi_f\left( g \left( \left( \begin{array}{cc}p&p^{-1}u\\ &p^{-1} \end{array} \right), 1 \right)_p \right)                                                                                   \\
        = & \chi(\kappa') \left( j_\infty(g_\infty', i) \varepsilon_2(\kappa_2') \right)^{-2k-1} f(g_\infty' i)                                                                                             \\
        = & \chi^{\{p\}}(p) p^{-k-\frac{1}{2}} (p,p)_{\Q_2,2} \gamma_{\Q_2}(p, \psi_{\Q_2})^{2k+3} \chi(\kappa) \left( j_\infty(g_\infty, i) \varepsilon_2(\kappa_2) \right)^{-2k-1} f(p^{-2} z + p^{-2}x).
    \end{align*}
    We arrive at
    \begin{align*}
        \mathcal{U}_{p^2} \varphi_f (g)
        = & \operatorname{vol}(p^2\Z_p) \chi^{\{p\}}(p) p^{-k-\frac{1}{2}} (p,p)_{\Q_2,2} \gamma_{\Q_2}(p, \psi_{\Q_2})^{2k+3}                         \\
          & \times \chi(\kappa) \left( j_\infty(g_\infty, i) \varepsilon_2(\kappa_2) \right)^{-2k-1}  \sum_{u \in \Z_p/p^2\Z_p} f(p^{-2} z + p^{-2}x).
    \end{align*}
    Since $a, d \in \Z_p^\times$, taking the sum over all $u \in \Z_p/p^2\Z_p$ is equivalent to taking the sum over all $x \in \Z_/p^2\Z_p$.
    Moreover, we have
    \begin{align*}
        \sum_{x \in \Z_p/p^2\Z_p} f(p^{-2} z + p^{-2}x)
         & = \sum_{x \in \Z_p/p^2\Z_p} \sum_{n \geq 1} a_n e^{2 \pi i n p^{-2}z} e^{2 \pi i n p^{-2}x} \\
         & = \sum_{n \geq 1} a_n e^{2 \pi i n p^{-2}z} \sum_{x \in \Z_p/p^2\Z_p} e^{2 \pi i n p^{-2}x} \\
         & = p^2 \sum_{\substack{n \geq 1                                                              \\ p^2|n}} a_n e^{2 \pi i n p^{-2}z}\\
         & =p^2 f|U(p^2)(z).
    \end{align*}
    Consequently, we have
    \begin{equation*}
        \mathcal{U}_{p^2} \varphi_f (g)
        =\operatorname{vol}(p^2\Z_p) \chi^{\{p\}}(p) p^{-k+\frac{3}{2}} (p,p)_{\Q_2,2} \gamma_{\Q_2}(p, \psi_{\Q_2})^{2k+3} \varphi_{f|U(p^2)}(g).
    \end{equation*}
\end{proof}
In \cite{kohnew, u4}, Kohnen and Ueda defined the space of oldforms to be the subspace spanned by the images of the plus spaces of lower level under the operators $\delta_m$, $U(m)$, and $R_p$.
Since the image of $\delta_p$ is characterized by the kernel of $R_p$ (\cite[Proposition 1.10]{u3}), the proposition above suggests that our formulation is compatible with the theory of Kohnen and Ueda.

\section{Non-supercuspidal representations}\label{sec:non-sc}
In this section, we shall prove the main theorems for irreducible representations that are not supercuspidal.
Since they can be realized as constituents of principal series representations, we are able to calculate directly.
Throughout this section, let $\eta$ be an arbitrary character of $\mathcal{O}^\times$.
Moreover, as before, we fix a nontrivial additive character $\psi \colon F \to \C^\times$ such that $c(\psi)=0$.
Note that we already know Theorem \ref{thm:conductor} for irreducible genuine representations that are not supercuspidal.

\subsection{Principal series representations}\label{subsec:ps}
Let $\mu$ be a character of $F^\times$.
Since $z_\psi(\pi_\psi(\mu))=\mu(-1)$, we assume that $\eta(-1)=\mu(-1)$.
First, we state one lemma on the support of the elements in $\pi_\psi(\mu)^{K_m}_\eta$.
\begin{lemma}\label{lem:structure-of-subspaces}
    The subspace $\pi_\psi(\mu)^{K_m}_\eta$ consists of elements in $\pi_\psi(\mu)$ that are supported on
    \begin{equation*}
        \bigsqcup_{i \in I_\psi(\mu,\eta, m)} \bigsqcup_{x \in R_i} \widetilde{B} (x, 1) K_m,
    \end{equation*}
    where
    \begin{equation*}
        I_\psi(\mu,\eta, m) = \Set{0 \leq i \leq \frac{m}{2} | c(\eta\mu^{-1}) \leq i \leq m-c(\mu)} \cup \Set{\frac{m}{2} \leq i \leq m | c(\mu) \leq i \leq m-c(\eta\mu)}.
    \end{equation*}
    and
    \begin{equation*}
        R_i = \begin{cases*}
            \{w\},                                                          & if $i=0$,               \\
            \{ n^{\mathrm{op}}(\varpi^i), n^{\mathrm{op}}(\varpi^i \xi) \}, & if $1 \leq i \leq m-1$, \\
            \{1_2\},                                                        & if $i=m$.
        \end{cases*}
    \end{equation*}
\end{lemma}
\begin{proof}
    The lemma essentially follows from the proof of \cite[Theorem 4.1]{ish}.

    By the proof of \cite[Corollary 2.7 (1)]{yam}, we have
    \begin{equation*}
        \pi_\psi(\mu)^{K_m}_\eta \cong \bigoplus_{g \in K_m \backslash G / B} \Hom_{K_m \cap \prescript{g}{}{B}} \left( \eta, \prescript{g}{}{\mu} \right),
    \end{equation*}
    which is given by
    \begin{equation*}
        f \mapsto \left( f(g^{-1}) \right)_{g \in K_m \backslash G / B},
    \end{equation*}
    where $\prescript{g}{}{\mu}$ denotes the character of $\prescript{g}{}{B}$ which sends $g t(a)n(b) g^{-1}$ to $\mu(a)$.
    In the proof of \cite[Theorem 4.1]{ish}, it was showed that the space $\Hom_{K_m \cap \prescript{g}{}{B}} ( \eta, \prescript{g}{}{\mu} )$ vanishes if and only if $g$ is in the set
    \begin{equation*}
        \bigsqcup_{i \in I_\psi(\mu,\eta, m)} \bigsqcup_{x \in R_i} \widetilde{B} (x, 1) K_m.
    \end{equation*}
    Note that the union is disjoint thanks to Lemma \ref{lem:coset_ps}.
    Since the set is stable under taking inverse, this completes the proof.
\end{proof}

Now, we shall give a proof of the main theorem for irreducible principal series representations, after two lemmas on the actions of $\mathcal{U}_{\varpi^2}$ and $\mathcal{R}_\varpi$ on principal series representations.
\begin{lemma}\label{lem:U-op_on_unram-ps}
    Suppose that $\mu$ is an unramified character of $F^\times$.
    Then
    \begin{equation*}
        \mathcal{U}_{\varpi^2} \left( \pi_\psi(\mu)^{K_0}_1 \right) \subset \pi_\psi(\mu)^{K_0}_1
    \end{equation*}
    if and only if there exists a quadratic or trivial character $\chi$ of $F^\times$ such that
    \begin{equation*}
        \mu = \chi \cdot |-|^{-\frac{1}{2}}.
    \end{equation*}
\end{lemma}
\begin{proof}
    Let $f_0$ be a nonzero element of $\pi_\psi(\mu)^{K_0}_1$.
    Since $\pi_\psi(\mu)^{K_0}_1$ is 1-dimensional,
    \begin{equation*}
        \mathcal{U}_{\varpi^2} \left( \pi_\psi(\mu)^{K_0}_1 \right) \subset \pi_\psi(\mu)^{K_0}_1
    \end{equation*}
    is equivalent to
    \begin{equation*}
        \mathcal{U}_{\varpi^2} f_0 \in \pi_\psi(\mu)^{K_0}_1.
    \end{equation*}
    Moreover, since $\mathcal{U}_{\varpi^2}f_0$ lies in $\pi_\psi(\mu)^{K_1}_1$ by Proposition \ref{prop:U-operator}, it follows from Lemma \ref{lem:structure-of-subspaces} that this is equivalent to
    \begin{equation*}
        \mathcal{U}_{\varpi^2}f_0(1) = \mathcal{U}_{\varpi^2}f_0 ((w,1)).
    \end{equation*}
    We have by definition
    \begin{align*}
        \mathcal{U}_{\varpi^2} f_0 (1)
         & = \int_\mathcal{O} f_0 \left( \left( \left( \begin{array}{cc}\varpi&\varpi^{-1} u\\ 0&\varpi^{-1} \end{array} \right), 1 \right) \right) du \\
         & = \int_\mathcal{O} |\varpi| \mu(\varpi) \chi_\psi((t(\varpi), 1)) f_0(1) du                                                                 \\
         & = q^{-1} \operatorname{vol}(\mathcal{O}) \mu(\varpi) \gamma_F(\varpi, \psi)^{-1} f_0(1).
    \end{align*}
    On the other hand, we have by definition
    \begin{align*}
        \mathcal{U}_{\varpi^2}f_0 ((w,1))
         & = \int_\mathcal{O} f_0\left( \left( \left( \begin{array}{cc}&\varpi^{-1}\\ -\varpi&-\varpi^{-1}u \end{array} \right), (-1,\varpi)_{F,2} \right) \right) du.
    \end{align*}
    When $u \in \mathcal{O}^\times$, we have
    \begin{align*}
        \left( \left( \begin{array}{cc}&\varpi^{-1}\\ -\varpi&-\varpi^{-1}u \end{array} \right), (-1,\varpi)_{F,2} \right)
        =\left( t(\varpi), 1 \right) \left( n(-\varpi^{-2}u^{-1}), 1 \right) \left( n^{\mathrm{op}}(\varpi^2 u), 1 \right) \left( t(-u^{-1}), 1 \right),
    \end{align*}
    and hence
    \begin{align*}
         & \int_{\mathcal{O}^\times} f_0\left( \left( \left( \begin{array}{cc}&\varpi^{-1}\\ -\varpi&-\varpi^{-1}u \end{array} \right), (-1,\varpi)_{F,2} \right) \right) du \\
         & = \int_{\mathcal{O}^\times} |\varpi| \mu(\varpi) \chi_\psi((t(\varpi), 1)) \bm{s}(t(-u^{-1}))^{-1} \bm{s}(n^{\mathrm{op}}(\varpi^2 u))^{-1} f_0(1) du             \\
         & = q^{-1} \operatorname{vol}(\mathcal{O}^\times) \mu(\varpi) \gamma_F(\varpi, \psi)^{-1} f_0(1).
    \end{align*}
    When $u \in \varpi \mathcal{O}^\times$, we have
    \begin{align*}
        \left( \left( \begin{array}{cc}&\varpi^{-1}\\ -\varpi&-\varpi^{-1}u \end{array} \right), (-1,\varpi)_{F,2} \right)
        = \left( n(-\varpi^{-1}), (-u,\varpi)_{F,2} \right) \left( n^{\mathrm{op}}(u), 1 \right) \left( t(-\varpi u^{-1}), 1 \right),
    \end{align*}
    and hence
    \begin{align*}
         & \int_{\varpi \mathcal{O}^\times} (-u,\varpi)_{F,2} \bm{s}(t(-\varpi u^{-1}))^{-1} \bm{s}(n^{\mathrm{op}}(u))^{-1} f_0(1) du \\
         & = f_0(1) \int_{\varpi \mathcal{O}^\times} (-u,\varpi)_{F,2} du                                                              \\
         & = 0.
    \end{align*}
    When $u \in \mathfrak{p}^2$, we have
    \begin{align*}
        \left( \left( \begin{array}{cc}&\varpi^{-1}\\ -\varpi&-\varpi^{-1}u \end{array} \right), (-1,\varpi)_{F,2} \right)
        = \left( t(\varpi^{-1}), 1 \right) \left( w , 1 \right) \left( n(\varpi^{-2} u), 1 \right),
    \end{align*}
    and hence
    \begin{align*}
         & \int_{\mathfrak{p}^2} |\varpi^{-1}| \mu(\varpi^{-1}) \chi_\psi((t(\varpi^{-1}), 1)) \bm{s}(n(\varpi^{-2} u))^{-1} \bm{s}(w)^{-1} f_0((w,1)) du \\
         & = \int_{\mathfrak{p}^2} |\varpi|^{-1} \mu(\varpi)^{-1} \gamma_F(\varpi^{-1}, \psi)^{-1} f_0(1) du                                              \\
         & = q \operatorname{vol}(\mathfrak{p}^2) \mu(\varpi)^{-1} \gamma_F(\varpi, \psi)^{-1} f_0(1).
    \end{align*}
    Therefore, we have
    \begin{align*}
        \mathcal{U}_{\varpi^2}f_0 ((w,1))
         & = q^{-1} \operatorname{vol}(\mathcal{O}^\times) \mu(\varpi) \gamma_F(\varpi, \psi)^{-1} f_0(1)
        + q \operatorname{vol}(\mathfrak{p}^2) \mu(\varpi)^{-1} \gamma_F(\varpi, \psi)^{-1} f_0(1).
    \end{align*}
    Thus we have
    \begin{align*}
         & \mathcal{U}_{\varpi^2}f_0 (1) - \mathcal{U}_{\varpi^2}f_0 ((w,1))                                         \\
         & =(\mu(\varpi) - q \mu(\varpi)^{-1}) \operatorname{vol}(\mathfrak{p}^2) \gamma_F(\varpi,\psi)^{-1} f_0(1),
    \end{align*}
    which is zero if and only if
    \begin{equation*}
        \mu(\varpi)^2 = q.
    \end{equation*}
    Since $\mu$ is unramified, the assertion follows.
\end{proof}

\begin{lemma}\label{lem:R-op_on_ps}
    Let $z$ be one of $1$ or $\xi$, and $z^\vee$ the other one.
    For any $m \geq 0$ and $f \in \pi_\psi(\mu)^{K_m}_\eta$, we have
    \begin{align*}
        \mathcal{R}_\varpi f (1)                                                         & = 0,                                                                                                                                                                                                     \\
        \mathcal{R}_\varpi f ((w,1))                                                     & = \frac{(-1,\varpi)_{F,2}}{2} q^{-1} \mu(\varpi) \gamma_F(\varpi,\psi)^{-1} \int_{\mathcal{O}^\times} \eta\mu^{-1}(x) dx                                                                                 \\
                                                                                         & \quad \times \left( \eta(-1) f\left( \left( n^{\mathrm{op}}(\varpi), 1 \right) \right) + \eta(-\xi) f\left( \left( n^{\mathrm{op}}(\xi\varpi), 1 \right) \right) \right),                                \\
        \mathcal{R}_\varpi f \left( \left( n^{\mathrm{op}}(z\varpi), 1 \right) \right)   & = \frac{(z,\varpi)}{2} \int_{\mathcal{O}^\times \setminus \pm1+\mathfrak{p}} (x^2-1, \varpi)_{F,2} \eta\mu^{-1}(x) dx \times f\left( \left( n^{\mathrm{op}}(z\varpi), 1 \right) \right)                  \\
                                                                                         & \quad -\frac{(z,\varpi)}{2} \eta(z^{-1} z^\vee) \int_{\mathcal{O}^\times} (z^{-1} z^\vee x^2-1, \varpi)_{F,2} \eta\mu^{-1}(x) dx \times f\left( \left( n^{\mathrm{op}}(z^\vee \varpi), 1 \right) \right) \\
                                                                                         & \qquad + (-1,\varpi)_{F,2} \operatorname{vol}(\mathcal{O}) \gamma_F(\varpi, \psi)^{-1} \mu(-z \varpi)^{-1} \times f\left( \left( w, 1 \right) \right),                                                   \\
        \mathcal{R}_\varpi f \left( \left( m^{\mathrm{op}}(z\varpi^i), 1 \right) \right) & = \int_{\mathcal{O}^\times} (u,\varpi) \eta\mu^{-1}(u') du \times f\left( \left(  m^{\mathrm{op}}(z\varpi^i), 1 \right) \right),
    \end{align*}
    for any $i \geq2$, where $u'$ is an element in $1+\varpi^{i-1}\mathcal{O}^\times$ such that $(u')^2=\varpi^{i-1}zu+1$.
\end{lemma}
\begin{proof}
    By definition, we have
    \begin{align*}
        \mathcal{R}_\varpi f (g) = \int_{\mathcal{O}^\times} (u,\varpi)_{F,2} f\left( g \left( \left( \begin{array}{cc}\varpi&\varpi^{-1} u\\ &\varpi^{-1} \end{array} \right) 1 \right) \right) du,
    \end{align*}
    for any $f \in \pi_\psi(\mu)$ and $g \in \widetilde{G}$.
    In particular, we have
    \begin{align*}
        \mathcal{R}_\varpi f (1) = 0,
    \end{align*}
    for any $f$.
    Next, we compute $\mathcal{R}_\varpi f ((w,1))$.
    For any $u\in \mathcal{O}^\times$, there exist exactly one $z\in \{1, \xi\}$ and two $x\in \mathcal{O}^\times$ such that $u=z x^2$, and in this situation we have
    \begin{align*}
        (w,1) (n(\varpi^{-1} u), 1) = (t(\varpi), 1) \left( n(-\varpi^{-1} u^{-1}), 1 \right) \left( t(x^{-1}), (-z,\varpi)_{F,2} \right) \left( n^{\mathrm{op}}(z \varpi), 1 \right) \left( t(-z^{-1} x^{-1}), 1 \right).
    \end{align*}
    This implies that
    \begin{align*}
        \mathcal{R}_\varpi f ((w,1))
         & = \frac{1}{2} \sum_{z=1,\xi} \int_{\mathcal{O}^\times} (zx^2,\varpi)_{F,2} |\varpi| \mu(\varpi) \gamma_F(\varpi, \psi)^{-1} |x^{-1}| \mu(x^{-1}) \gamma_F(x^{-1},\psi)^{-1} (-z, \varpi)_{F,2}                 \\
         & \quad\times f\left( \left( n^{\mathrm{op}}(z \varpi), 1 \right) \right) \bm{s}(t(-z^{-1} x^{-1}))^{-1} \eta(-zx) dx                                                                                            \\
         & = \frac{(-1,\varpi)_{F,2}}{2} q^{-1} \mu(\varpi) \gamma_F(\varpi, \psi)^{-1} \int_{\mathcal{O}^\times} \eta\mu^{-1}(x) dx \sum_{z=1,\xi} \eta(-z) f\left( \left( n^{\mathrm{op}}(z \varpi), 1 \right) \right).
    \end{align*}
    Next, let $z=1$ or $\xi$.
    We shall compute $\mathcal{R}_\varpi f (( n^{\mathrm{op}}(z\varpi), 1))$.
    For any $u \in \mathcal{O}^\times \setminus -z^{-1} + \mathfrak{p}$, there exist exactly one $y \in \{1, \xi\}$ and two $x\in \mathcal{O}^\times$ such that $zu+1=x^2 zy^{-1}$.
    Here, note that $x$ cannot be in $\pm1+\mathfrak{p}$ if $y=z$.
    In this situation we have
    \begin{align*}
        (n^{\mathrm{op}}(z \varpi), 1) (n(\varpi^{-1} u), 1) = \left( n(\varpi^{-1}u (zu+1)^{-1}), 1 \right) \left( t(x^{-1} z^{-1} y), (zy, \varpi)_{F,2} \right) \left( n^{\mathrm{op}}(\varpi y), 1 \right) \left( t(x^{-1}), 1 \right).
    \end{align*}
    For any $u \in -z^{-1} + \mathfrak{p}$, on the other hand, we have
    \begin{align*}
        (n^{\mathrm{op}}(z \varpi), 1) (n(\varpi^{-1} u), 1) = \left( t(-z^{-1} \varpi^{-1}), (-1, \varpi)_{F,2} \right) \left( n(z \varpi), 1 \right) \left( w, 1 \right) \left( n(z^{-1} \varpi^{-1} (zu+1)), 1 \right).
    \end{align*}
    Therefore, we have
    \begin{align*}
         & \mathcal{R}_\varpi f (( n^{\mathrm{op}}(z\varpi), 1))                                                                                                                                                                                                                                                      \\
         & = \frac{1}{2} \int_{\mathcal{O}^\times \setminus \pm1+\mathfrak{p}} \left( \frac{x^2-1}{z},\varpi \right)_{F,2} |x|^{-1} \mu(x^{-1}) \gamma_F(x^{-1}, \psi)^{-1} f \left(\left( n^{\mathrm{op}}(z\varpi), 1 \right)\right) \bm{s}(t(x^{-1}))^{-1} \eta(x) dx                                               \\
         & \quad - \frac{1}{2}\int_{\mathcal{O}^\times} \left( \frac{x^2 z(z^\vee)^{-1}-1}{z},\varpi \right)_{F,2} \left| \frac{z^\vee}{xz} \right| \mu(x^{-1}z^{-1}z^\vee) \gamma_F(x^{-1}z^{-1}z^\vee, \psi)^{-1} f \left(\left( n^{\mathrm{op}}(z^\vee \varpi), 1 \right)\right) \bm{s}(t(x^{-1}))^{-1} \eta(x) dx \\
         & \qquad + \int_{-z^{-1}+\mathfrak{p}} (u,\varpi)_{F,2} |-z^{-1} \varpi^{-1}| \mu(-z^{-1} \varpi^{-1}) \gamma_F(-z^{-1} \varpi^{-1}, \psi)^{-1} (-1, \varpi)_{F,2} f\left( \left( w,1 \right) \right) du                                                                                                     \\
         & =\frac{(z,\varpi)_{F,2}}{2} \int_{\mathcal{O}^\times \setminus \pm1+\mathfrak{p}} (x^2-1,\varpi)_{F,2} \eta\mu^{-1} (x) f \left(\left( n^{\mathrm{op}}(z\varpi), 1 \right)\right) dx                                                                                                                       \\
         & \quad - \frac{(z,\varpi)_{F,2}}{2} \eta(z^{-1}z^\vee) \int_{\mathcal{O}^\times} (x^2 z^{-1}z^\vee-1,\varpi)_{F,2} \eta\mu^{-1}(x) f \left(\left( n^{\mathrm{op}}(z^\vee \varpi), 1 \right)\right) dx                                                                                                       \\
         & \qquad + q \operatorname{vol}(\mathfrak{p}) (-1,\varpi)_{F,2}\gamma_F(\varpi, \psi)^{-1}\mu(-z\varpi)^{-1} f\left( \left( w,1 \right) \right).
    \end{align*}
    We finally compute $\mathcal{R}_\varpi f (( n^{\mathrm{op}}(z\varpi^i), 1))$, for $z \in \{1, \xi\}$ and $i \geq 2$.
    For any $u\in \mathcal{O}^\times$, there exist exactly one $x\in 1+\varpi^{i-1}\mathcal{O}^\times$ such that $x^2 = zu\varpi^{i-1}+1$, and we have
    \begin{align*}
        (n^{\mathrm{op}}(z \varpi^i), 1) (n(\varpi^{-1} u), 1) = \left(n(\varpi^{-1} u x^{-2}), 1 \right) \left( t(x^{-1}), 1 \right) \left( n^{\mathrm{op}}(z \varpi^i), 1 \right) \left( t(x^{-1}), 1 \right).
    \end{align*}
    Thus we have
    \begin{align*}
        \mathcal{R}_\varpi f (( n^{\mathrm{op}}(z\varpi^i), 1))
         & = \int_{\mathcal{O}^\times} (u, \varpi)_{F,2} |x^{-1}| \mu(x^{-1}) \gamma_F(x^{-1}, \psi)^{-1} f\left( \left( n^{\mathrm{op}}(z \varpi^i), 1 \right) \right) \eta(x) du \\
         & = \int_{\mathcal{O}^\times} (u, \varpi)_{F,2} \eta\mu^{-1}(x) du f\left( \left( n^{\mathrm{op}}(z \varpi^i), 1 \right) \right).
    \end{align*}
    Now the assertions follow.
\end{proof}

\begin{proposition}\label{prop:unram_case}
    Suppose that $\mu$ is unramified and $\eta$ is trivial.
    Assume that $\mu^2 \neq |-|^{-1}$.
    Then Theorem \ref{thm:main-generic} holds for $\pi=\pi_\psi(\mu)$ and $\eta=1$.
\end{proposition}
\begin{proof}
    Let $f_0$ be a nonzero element in $\pi_\psi(\mu)^{K_0}_1$, which is unique up to a scalar multiple.
    By Lemma \ref{lem:R-op_on_ps}, we have
    \begin{equation*}
        \mathcal{R}_\varpi f_0 ((w,1))
        = (-1,\varpi)_{F,2} q^{-1} \mu(\varpi) \gamma_F(\varpi,\psi)^{-1} \operatorname{vol}(\mathcal{O}^\times) f_0(1)
        \neq 0.
    \end{equation*}
    Thus we have $\pi_\psi(\mu)^{K_0,\KUold}_1=\{0\}$ and
    \begin{equation*}
        \pi_\psi(\mu)^{K_0, \KUnew}_1 = \pi_\psi(\mu)^{K_0}_1 = \C f_0.
    \end{equation*}

    Recall that $\pi_\psi(\mu)^{K_1}_1$ is 2-dimensional.
    Combining with the assumption that $\mu^2 \neq |-|^{-1}$, Lemma \ref{lem:U-op_on_unram-ps} implies that
    \begin{equation*}
        \pi_\psi(\mu)^{K_1}_1 = \pi_\psi(\mu)^{K_0}_1 + \mathcal{U}_{\varpi^2} \left( \pi_\psi(\mu)^{K_0}_1 \right).
    \end{equation*}
    Thus we have $\pi_\psi(\mu)^{K_1, \KUold}_1 = \pi_\psi(\mu)^{K_1}_1$ and
    \begin{equation*}
        \pi_\psi(\mu)^{K_1, \KUnew}_1 = \{0\}.
    \end{equation*}

    Next we consider $\pi_\psi(\mu)^{K_3, \KUnew}_1$.
    Recall that it is 6-dimensional and any $f \in \pi_\psi(\mu)^{K_3}_1$ is determined by the values at $1$, $(w,1)$, $(n^{\mathrm{op}}(\varpi), 1)$, $(n^{\mathrm{op}}(\xi\varpi), 1)$, $(n^{\mathrm{op}}(\varpi^2), 1)$, and $(n^{\mathrm{op}}(\xi\varpi^2), 1)$.
    Recall also that $\mathcal{R}_\varpi(\pi_\psi(\mu)^{K_3}_1)$ is contained in $\pi_\psi(\mu)^{K_3}_1$.
    Then by Lemmas \ref{lem:R-op_on_ps} and \ref{lem:square-1_Hilbertsymbol}, one can see that an element $f \in \pi_\psi(\mu)^{K_3}_1$ is in $\operatorname{Ker}(\mathcal{R}_\varpi)$ if and only if
    \begin{equation*}
        f((n^{\mathrm{op}}(\varpi), 1))) = - f((n^{\mathrm{op}}(\xi\varpi), 1)) = q \gamma_F(\varpi, \psi)^{-1} \mu(\varpi)^{-1} f((w,1)).
    \end{equation*}
    From this we can deduce that $\operatorname{Ker}(\mathcal{R}_\varpi|_{\pi_\psi(\mu)^{K_3}_1})$ is 4-dimensional and $\operatorname{Ker}(\mathcal{R}_\varpi|_{\pi_\psi(\mu)^{K_1}_1})$ is zero.
    Since $\pi_\psi(\mu)^{K_1}_1$ is 2-dimensional, we now have
    \begin{equation*}
        \pi_\psi(\mu)^{K_3}_1 = \pi_\psi(\mu)^{K_1}_1 \oplus \operatorname{Ker}(\mathcal{R}_\varpi|_{\pi_\psi(\mu)^{K_3}_1}) \subset \pi_\psi(\mu)^{K_3, \KUold}_1.
    \end{equation*}
    Thus we have $\pi_\psi(\mu)^{K_3, \KUold}_1 = \pi_\psi(\mu)^{K_3}_1$ and
    \begin{equation*}
        \pi_\psi(\mu)^{K_3, \KUnew}_1 = \{0\}.
    \end{equation*}

    Finally, we have
    \begin{align*}
        \pi_\psi(\mu)^{K_2}_1
         & = \pi_\psi(\mu)^{K_3}_1 \cap \pi_\psi(\mu)^{K_2}_1                                                                                       \\
         & = \left( \pi_\psi(\mu)^{K_1}_1 \oplus \operatorname{Ker}(\mathcal{R}_\varpi|_{\pi_\psi(\mu)^{K_3}_1}) \right) \cap \pi_\psi(\mu)^{K_2}_1 \\
         & = \pi_\psi(\mu)^{K_1}_1 \oplus \operatorname{Ker}(\mathcal{R}_\varpi|_{\pi_\psi(\mu)^{K_2}_1})                                           \\
         & \subset \pi_\psi(\mu)^{K_2, \KUold}_1,
    \end{align*}
    which implies that $\pi_\psi(\mu)^{K_2, \KUold}_1 = \pi_\psi(\mu)^{K_2}_1$ and
    \begin{equation*}
        \pi_\psi(\mu)^{K_2, \KUnew}_1 = \{0\}.
    \end{equation*}
    Since $\pi_\psi(\mu)^{K_m, \RSnew}_\eta$ is zero for $m\geq4$ (\cite[Corollary 4.2]{ish}), this completes the proof.
\end{proof}
\begin{remark}\label{rem:uni-def}
    In the setting and the notation of the proof of Proposition \ref{prop:unram_case} but without the assumption that $\mu^2 \neq |-|^{-1}$, one can see that $\mathcal{R}_\varpi f_0 \in \pi_\psi(\mu)^{K_2}_1$ does not lie in $\pi_\psi(\mu)^{K_1}_1$.
    Moreover, we have just seen that $\operatorname{Ker}(\mathcal{R}_\varpi|_{\pi_\psi(\mu)^{K_0}_\eta})$ is zero.
    Thus we may uniformly define
    \begin{equation*}
        \pi^{K_m, \KUold}_\eta = \pi^{K_{m-1}}_\eta + \mathcal{U}_{\varpi^2}\left( \pi^{K_{m-1}}_\eta \right) + \left( \mathcal{R}_\varpi\left( \pi^{K_{m-1}}_\eta \right) \cap \pi^{K_m}_\eta  \right) + \operatorname{Ker}(\mathcal{R}_\varpi|_{\pi^{K_m}_\eta}),
    \end{equation*}
    for all $m \geq 0$.
\end{remark}

\begin{proposition}\label{prop:quad-ps_case}
    Suppose that $\mu|_{\mathcal{O}^\times}$ is nontrivial quadratic, and $\eta=\mu|_{\mathcal{O}^\times}$.
    Then Theorem \ref{thm:main-generic} holds for $\pi=\pi_\psi(\mu)$ and $\eta$.
\end{proposition}
\begin{proof}
    Recall from Lemma \ref{lem:structure-of-subspaces} that any element in $\pi_\psi(\mu)^{K_1}_\eta$ (resp. $\pi_\psi(\mu)^{K_2}_\eta$) is determined by the values at $1$ and $(w,1)$, (resp. $1$, $(w,1)$, $(n^{\mathrm{op}}(\varpi), 1)$, and $(n^{\mathrm{op}}(\xi\varpi), 1)$).
    For an element $f_1 \in \pi_\psi(\mu)^{K_1}_\eta$ such that $f_1(1)=0$ and $f_1((w,1)) \neq 0$, it follows from Lemma \ref{lem:R-op_on_ps} that $\mathcal{R}_\varpi f_1 (1)$ is zero but $\mathcal{R}_\varpi f_1 ((n^{\mathrm{op}}(\varpi), 1))$ is not.
    This means that $\mathcal{R}_\varpi f_1$ is an element of $\pi_\psi(\mu)^{K_2}_\eta$ by Proposition \ref{prop:R-operator}, but not of $\pi_\psi(\mu)^{K_1}_\eta$.
    In particular, $f_1$ and $\mathcal{R}_\varpi f_1$ are linearly independent and neither of them lies in $\operatorname{Ker}(\mathcal{R}_\varpi)$.
    As in the proof of Proposition \ref{prop:unram_case}, one can see that an element $f \in \pi_\psi(\mu)^{K_2}_\eta$ is in $\operatorname{Ker}(\mathcal{R}_\varpi)$ if and only if
    \begin{equation*}
        f\left( \left( n^{\mathrm{op}}(\varpi), 1 \right) \right) = f\left( \left( n^{\mathrm{op}}(\xi\varpi), 1 \right) \right) = q (-1,\varpi)_{F,2} \gamma_F(\varpi, \psi)^{-1} \mu(\varpi) f\left( \left( w, 1 \right) \right).
    \end{equation*}
    From this we can deduce that the dimension of $\operatorname{Ker}(\mathcal{R}_\varpi|_{\pi_\psi(\mu)^{K_2}_\eta})$ is $4-2=2$, and that of $\operatorname{Ker}(\mathcal{R}_\varpi|_{\pi_\psi(\mu)^{K_1}_\eta})$ is $2-1=1$.
    Hence the dimension of $\pi_\psi(\mu)^{K_1, \KUnew}_\eta$ is $2-1=1$.
    We can also deduce
    \begin{equation*}
        \C f_1 \oplus \C \mathcal{R}_\varpi f_1 \oplus \operatorname{Ker}(\mathcal{R}_\varpi|_{\pi_\psi(\mu)^{K_2}_\eta}) \subset \pi_\psi(\mu)^{K_2, \KUold}_\eta \subset \pi_\psi(\mu)^{K_2}_\eta.
    \end{equation*}
    Comparing the dimensions, we see that $\pi_\psi(\mu)^{K_2, \KUold}_\eta$ is equal to $\pi_\psi(\mu)^{K_2}_\eta$, and hence $\pi_\psi(\mu)^{K_2, \KUnew}_\eta$ is zero.

    Since $\pi_\psi(\mu)^{K_m, \RSnew}_\eta$ is zero unless $m=1$ or $2$ (\cite[Corollary 4.2]{ish}), this completes the proof.
\end{proof}

\begin{proposition}\label{prop:general-min-conductor-ps_case}
    Suppose that $\mu^2$ is nontrivial on $\mathcal{O}^\times$, and that $\eta=\mu|_{\mathcal{O}^\times}$ or $\mu^{-1}|_{\mathcal{O}^\times}$.
    Then Theorem \ref{thm:main-generic} holds for $\pi=\pi_\psi(\mu)$ and $\eta$.
\end{proposition}
\begin{proof}
    Put $M=c_\eta(\pi_\psi(\mu))=c(\mu)$.
    Let $f_1$, $f_2$, $f_3$, $f_4$, and $f_5$ be the nonzero elements in $\pi_\psi(\mu)^{K_{M+2}}_\eta$ such that
    \begin{align*}
        f_i(g_j)=\begin{dcases*}
                     1, & if $i=j$,     \\
                     0, & if $i\neq j$,
                 \end{dcases*}
    \end{align*}
    where
    \begin{equation*}
        g_1=(w,1), \quad g_2=(n^{\mathrm{op}}(\varpi), 1), \quad g_3=(n^{\mathrm{op}}(\xi\varpi), 1), \quad g_4=(n^{\mathrm{op}}(\varpi^2), 1), \quad g_5=(n^{\mathrm{op}}(\xi\varpi^2), 1),
    \end{equation*}
    if $\eta=\mu|_{\mathcal{O}^\times}$, and
    \begin{equation*}
        g_1=1, \quad g_2=(n^{\mathrm{op}}(\varpi^M), 1), \quad g_3=(n^{\mathrm{op}}(\xi\varpi^M), 1), \quad g_4=(n^{\mathrm{op}}(\varpi^{M+1}), 1), \quad g_5=(n^{\mathrm{op}}(\xi\varpi^{M+1}), 1),
    \end{equation*}
    if $\eta=\mu^{-1}|_{\mathcal{O}^\times}$.
    We shall divide the proof into the following three cases:
    \begin{enumerate}[(1)]
        \item $\eta=\mu|_{\mathcal{O}^\times}$;
        \item $\eta=\mu^{-1}|_{\mathcal{O}^\times}$ and $c(\mu)=1$;
        \item $\eta=\mu^{-1}|_{\mathcal{O}^\times}$ and $c(\mu)\geq 2$.
    \end{enumerate}

    Suppose first that $\eta=\mu|_{\mathcal{O}^\times}$.
    It follows from Lemma \ref{lem:structure-of-subspaces} that $\{f_1\}$, $\{f_1, f_2, f_3\}$, and $\{f_1, f_2, f_3, f_4, f_5\}$ are bases of $\pi_\psi(\mu)^{K_M}_\eta$, $\pi_\psi(\mu)^{K_{M+1}}_\eta$, and $\pi_\psi(\mu)^{K_{M+2}}_\eta$, respectively.
    By Lemma \ref{lem:R-op_on_ps}, we have
    \begin{align*}
        \mathcal{R}_\varpi f_1 & = a f_2 + \mu(\xi)^{-1} a f_3,               \\
        \mathcal{R}_\varpi f_2 & = b f_1 + c f_2 + d f_3,                     \\
        \mathcal{R}_\varpi f_3 & = \mu(\xi) b f_1 - \mu(\xi)^2 d f_2 - c f_3, \\
        \mathcal{R}_\varpi f_4 & = 0,                                         \\
        \mathcal{R}_\varpi f_5 & = 0,
    \end{align*}
    where
    \begin{align*}
        a & = \operatorname{vol}(\mathcal{O}) (-1, \varpi)_{F,2} \gamma_F(\varpi, \psi)^{-1} \mu(-\varpi)^{-1},              \\
        b & = \frac{\operatorname{vol}(\mathcal{O}^\times)}{2q} (-1, \varpi)_{F,2} \gamma_F(\varpi, \psi)^{-1} \mu(-\varpi), \\
        c & = - \frac{\operatorname{vol}(\mathfrak{p})}{2} \left( 1 + (-1, \varpi)_{F,2} \right),                            \\
        d & = \mu(\xi)^{-1} \frac{\operatorname{vol}(\mathfrak{p})}{2} \left( 1 - (-1, \varpi)_{F,2} \right).
    \end{align*}
    Note that $a$ and $b$ are nonzero, and that exactly one of $c$ or $d$ is zero and the other is nonzero.
    It follows that $\operatorname{Ker}(\mathcal{R}_\varpi|_{\pi_\psi(\mu)^{K_M}_\eta})$ is zero, and hence
    \begin{equation*}
        \pi_\psi(\mu)^{K_M, \KUnew}_\eta = \pi_\psi(\mu)^{K_M}_\eta = \C f_1.
    \end{equation*}
    It also follows that
    \begin{align*}
        \operatorname{Ker}(\mathcal{R}_\varpi|_{\pi_\psi(\mu)^{K_{M+1}}_\eta}) & = \C \left( (c+\mu(\xi)d)f_1 - a(f_2 - \mu(\xi)^{-1} f_3) \right),                                    \\
        \mathcal{R}_\varpi \left( \pi_\psi(\mu)^{K_M}_\eta \right)             & = \C \left( f_2 + \mu(\xi)^{-1} f_3 \right),                                                          \\
        \operatorname{Ker}(\mathcal{R}_\varpi|_{\pi_\psi(\mu)^{K_{M+2}}_\eta}) & = \operatorname{Ker}(\mathcal{R}_\varpi|_{\pi_\psi(\mu)^{K_{M+1}}_\eta}) \oplus \C f_4 \oplus \C f_5,
    \end{align*}
    and hence
    \begin{equation*}
        \pi_\psi(\mu)^{K_{M+1}, \KUnew}_\eta = \pi_\psi(\mu)^{K_{M+2}, \KUnew}_\eta = \{0\}.
    \end{equation*}

    Suppose next that $\eta=\mu^{-1}|_{\mathcal{O}^\times}$ and $c(\mu)=1$.
    Then $M=1$ and it follows from Lemma \ref{lem:structure-of-subspaces} that $\{f_1+ f_2 + f_3 + f_4 +f_5\}$, $\{f_1+ f_4 + f_5, f_2, f_3\}$, and $\{f_1, f_2, f_3, f_4, f_5\}$ are bases of $\pi_\psi(\mu)^{K_1}$, $\pi_\psi(\mu)^{K_2}$, and $\pi_\psi(\mu)^{K_3}$, respectively.
    By Lemma \ref{lem:R-op_on_ps}, we have
    \begin{align*}
        \mathcal{R}_\varpi f_1 & = 0,             \\
        \mathcal{R}_\varpi f_2 & = a f_2 + b f_3, \\
        \mathcal{R}_\varpi f_3 & = -b f_2 -a f_3, \\
        \mathcal{R}_\varpi f_4 & = c f_4,         \\
        \mathcal{R}_\varpi f_5 & = -c f_4,
    \end{align*}
    where
    \begin{align*}
        a & = \int_{\mathcal{O}^{\times 2} \setminus 1+\mathfrak{p}} (x-1, \varpi)_{F,2} \mu(x)^{-1} dx, \\
        b & = \int_{\xi\mathcal{O}^{\times 2}} (x-1, \varpi)_{F,2} \mu(x)^{-1} dx,                       \\
        c & = \int_{\mathcal{O}^\times} (u,\varpi)_{F,2} \mu(\varpi u + 1)^{-1} du.
    \end{align*}
    Moreover, since $c(\mu)=1$, we have
    \begin{equation*}
        c = \int_{\mathcal{O}^\times} (u,\varpi)_{F,2} du = 0.
    \end{equation*}
    Since $(x, \varpi)_{F,2}=+1$ (resp. $-1$) if $x\in \mathcal{O}^{\times 2}$ (resp. $\xi\mathcal{O}^{\times 2}$), changing the variables $x \mapsto t = x^{-1}$, we have
    \begin{equation*}
        a - b = \int_{\mathcal{O}^\times \setminus 1+\mathfrak{p}} (1-t, \varpi)_{F,2} \mu(t) dt = J(\chi_\varpi, \mu) \neq 0.
    \end{equation*}
    Thus we have
    \begin{equation*}
        \mathcal{R}_\varpi (f_1+ f_2 + f_3 + f_4 +f_5) = (a-b) (f_2-f_3) \neq 0,
    \end{equation*}
    and hence
    \begin{equation*}
        \pi_\psi(\mu)^{K_1, \KUnew}_\eta = \pi_\psi(\mu)^{K_1}_\eta = \C (f_1+ f_2 + f_3 + f_4 +f_5).
    \end{equation*}
    This also implies that $\pi_\psi(\mu)^{K_2, \KUold}_\eta$ contains $f_2-f_3$.
    Since $f_1+ f_2 + f_3 + f_4 +f_5 \in \pi_\psi(\mu)^{K_1}$ and $f_1+ f_4 + f_5 \in \operatorname{Ker}(\mathcal{R}_\varpi)$ also lie in $\pi_\psi(\mu)^{K_2, \KUold}_\eta$, we have $\pi_\psi(\mu)^{K_2, \KUold}_\eta = \pi_\psi(\mu)^{K_2}_\eta$ and
    \begin{equation*}
        \pi_\psi(\mu)^{K_2, \KUnew}_\eta = \{0\}.
    \end{equation*}
    Since $f_4$ and $f_5$ belong to $\operatorname{Ker}(\mathcal{R}_\varpi)$, we have
    \begin{equation*}
        \pi_\psi(\mu)^{K_3, \KUnew}_\eta = \{0\}.
    \end{equation*}

    Suppose finally that $\eta=\mu^{-1}|_{\mathcal{O}^\times}$ and $c(\mu)\geq 2$.
    As in the last case, one can see that $\{f_1+ f_2 + f_3 + f_4 +f_5\}$, $\{f_1+ f_4 + f_5, f_2, f_3\}$, and $\{f_1, f_2, f_3, f_4, f_5\}$ are bases of $\pi_\psi(\mu)^{K_M}$, $\pi_\psi(\mu)^{K_{M+1}}$, and $\pi_\psi(\mu)^{K_{M+2}}$, respectively.
    Again by Lemma \ref{lem:R-op_on_ps}, we have
    \begin{align*}
        \mathcal{R}_\varpi f_1 & = 0,      \\
        \mathcal{R}_\varpi f_2 & = a f_2,  \\
        \mathcal{R}_\varpi f_3 & = -a f_3, \\
        \mathcal{R}_\varpi f_4 & = b f_4,  \\
        \mathcal{R}_\varpi f_5 & = -b f_5,
    \end{align*}
    where
    \begin{align*}
        a & = \int_{\mathcal{O}^\times} (u, \varpi)_{F,2} \mu(\varpi^{M-1} u + 1)^{-1} du, \\
        b & = \int_{\mathcal{O}^\times} (u, \varpi)_{F,2} \mu(\varpi^M u + 1)^{-1} du.
    \end{align*}
    Since $M=c(\mu)\geq 2$, the mapping
    \begin{equation*}
        \mathcal{O} \ni u \mapsto \mu(\varpi^{M-1} u + 1)^{-1} \in \C^\times
    \end{equation*}
    is a nontrivial additive character of $\mathcal{O}$ of conductor 1, for which we shall write $\psi_\mu$.
    Then we have
    \begin{equation*}
        a = g(\chi_\varpi, \psi_\mu) \neq 0.
    \end{equation*}
    On the other hand, since  $M=c(\mu)$, we have
    \begin{equation*}
        b = \int_{\mathcal{O}^\times} (u, \varpi)_{F,2} du = 0.
    \end{equation*}
    As above, we can see that
    \begin{equation*}
        \pi_\psi(\mu)^{K_M, \KUnew}_\eta = \pi_\psi(\mu)^{K_M}_\eta = \C (f_1+ f_2 + f_3 + f_4 +f_5)
    \end{equation*}
    and
    \begin{equation*}
        \pi_\psi(\mu)^{K_{M+1}, \KUnew}_\eta = \pi_\psi(\mu)^{K_{M+2}, \KUnew}_\eta = \{0\}.
    \end{equation*}

    Since $\pi_\psi(\mu)^{K_m, \RSnew}_\eta$ is zero unless $M \leq m \leq M+2$ (\cite[Corollary 4.2]{ish}), this completes the proof.
\end{proof}

\begin{proposition}\label{prop:general-non-min-ps_case}
    Suppose that $c(\eta\mu) > 0$ and $c(\eta\mu^{-1}) > 0$.
    Then Theorem \ref{thm:main-generic} holds for $\pi=\pi_\psi(\mu)$ and $\eta$.
\end{proposition}
\begin{proof}
    Put $M=c_\eta(\pi_\psi(\mu)) \geq 2$ and $C=c(\eta\mu^{-1})$.
    Let $f_1$, $f_2$, $f_3$, and $f_4$ be the nonzero elements in $\pi_\psi(\mu)^{K_{M+1}}_\eta$ such that
    \begin{align*}
        f_i(g_j)=\begin{dcases*}
                     1, & if $i=j$,     \\
                     0, & if $i\neq j$,
                 \end{dcases*}
    \end{align*}
    where
    \begin{equation*}
        g_1 = (n^{\mathrm{op}}(\varpi^C), 1), \quad g_2=(n^{\mathrm{op}}(\xi\varpi^C), 1), \quad g_3=(n^{\mathrm{op}}(\varpi^{C+1}), 1), \quad g_4=(n^{\mathrm{op}}(\xi\varpi^{C+1}), 1).
    \end{equation*}
    These elements exist uniquely by Lemma \ref{lem:structure-of-subspaces}.
    More precisely, by a straightforward calculation we have $I_\psi(\mu, \eta, M) = \{C\}$, and the lemma tells us that $\{f_1, f_2\}$ and $\{f_1, f_2, f_3, f_4\}$ are bases of $\pi_\psi(\mu)^{K_M}_\eta$ and $\pi_\psi(\mu)^{K_{M+1}}_\eta$, respectively.

    By Lemma \ref{lem:R-op_on_ps}, we have
    \begin{align*}
        \mathcal{R}_\varpi f_1 & = a f_1 + b f_2,  \\
        \mathcal{R}_\varpi f_2 & = -b f_1 -a f_2 , \\
        \mathcal{R}_\varpi f_3 & = c f_3,          \\
        \mathcal{R}_\varpi f_4 & = -c f_4,
    \end{align*}
    where
    \begin{align*}
        a & = \begin{dcases*}
                  \frac{1}{2} \int_{\mathcal{O}^\times \setminus \pm1+\mathfrak{p}} (x^2-1, \varpi)_{F,2} \eta\mu^{-1}(x) dx, & if $C=1$,    \\
                  \int_{\mathcal{O}^\times} (u, \varpi)_{F,2} \chi(\varpi^{C-1} u + 1) du,                                    & if $C\geq2$,
              \end{dcases*} \\
        b & = \begin{dcases*}
                  \frac{1}{2} \eta(\xi) \int_{\mathcal{O}^\times} (\xi x^2-1, \varpi)_{F,2} \eta\mu^{-1}(x) dx , & if $C=1$,    \\
                  0,                                                                                             & if $C\geq2$,
              \end{dcases*}                         \\
        c & = \int_{\mathcal{O}^\times} (u, \varpi)_{F,2} \chi(\varpi^C u + 1) du,
    \end{align*}
    where $\chi$ is a character of $1+\mathfrak{p}$ such that $\chi^2=\eta\mu^{-1}|_{1+\mathfrak{p}}$.
    If $C=1$, by Lemma \ref{lem:jacobisum2}, at least one of $a$ or $b$ is nonzero.
    If $C\geq2$, since $c(\chi)=c(\eta\mu^{-1})=C$, the mapping
    \begin{equation*}
        \mathcal{O} \ni u \mapsto \chi(\varpi^{C-1} u + 1)^{-1} \in \C^\times
    \end{equation*}
    is a nontrivial additive character of $\mathcal{O}$ of conductor 1, for which we shall write $\psi_\chi$.
    Then we have $a=g(\chi_\varpi, \psi_\chi)\neq 0$.
    In any case, $\mathcal{R}_\varpi|_{\pi_\psi(\mu)^{K_M}_\eta}$ is a nonzero linear map from $\pi_\psi(\mu)^{K_M}_\eta$ to itself with trace $0$.
    Combining this with Proposition \ref{prop:R-operator}, there exist $f_+$ and $f_-$ such that
    \begin{align*}
        \pi_\psi(\mu)^{K_M, \KUnew}_\eta = \pi_\psi(\mu)^{K_M}_\eta = \C f_+ \oplus \C f_-, \\
        \mathcal{R}_\varpi f_\pm = \pm \gamma_F(\varpi, \psi) q^{-\frac{1}{2}} \operatorname{vol}(\mathcal{O}) f_\pm.
    \end{align*}
    Since $c(\chi)=C$, we have $c=0$.
    This implies that
    \begin{equation*}
        \pi_\psi(\mu)^{K_{M+1}, \KUnew}_\eta = \{0\}.
    \end{equation*}
    Since $\pi_\psi(\mu)^{K_m, \RSnew}_\eta$ is zero unless $m=M$ or $M+1$ (\cite[Corollary 4.2]{ish}), this completes the proof.
\end{proof}

\subsection{Even Weil and Steinberg representations}\label{subsec:evenWeil/Steinberg}
Next we shall consider the nontrivial irreducible subquotients of principal series representations, i.e., the even Weil representations and the Steinberg representations.
\begin{proposition}\label{prop:evenWeil}
    Let $\chi$ be a trivial or quadratic character of $F^\times$ such that $\chi(-1)=\eta(-1)$.
    Then Theorems \ref{thm:main-nongeneric} and \ref{thm:main-generic} hold for $\pi=\omega_{\psi,\chi}^+$ and $\eta$.
\end{proposition}
\begin{proof}
    Before the proof, note that the assumption $\chi(-1)=\eta(-1)$ is equivalent to $z_\psi(\pi)=\eta(-1)$, which is a condition of the theorems we are to prove.

    Let $\psi'$ be a nontrivial additive character such that $c(\psi')=0$ or $-1$ and we realize $\pi$ as $(\omega_{\psi'}^+, \mathcal{S}^+(F))$.
    Then by the proof of \cite[Theorem 4.4]{ish}, there exists $\varphi \in \mathcal{S}^+(F)$ with $\operatorname{supp}(\varphi) \subset \mathcal{O}$ such that
    \begin{itemize}
        \item $\varphi(1) \neq 0$;
        \item $\pi^{K_{c_\eta(\pi)}}_\eta = \C \varphi$.
    \end{itemize}
    Note that $\mathcal{R}_\varpi \varphi$ is given by
    \begin{equation*}
        [\mathcal{R}_\varpi \varphi] (y) = \int_{\mathcal{O}^\times}( u, \varpi )_{F,2} \psi'(\varpi^{-1} u y) \varphi(y) du.
    \end{equation*}
    If $\pi$ is $\psi$- or $\psi_\xi$-generic, then $c(\psi')=0$ and we have
    \begin{equation*}
        \mathcal{R}_\varpi \varphi (1)
        = g(\chi_\varpi, \psi'_{\varpi^{-1}}) \varphi(1)
        \neq 0.
    \end{equation*}
    This means that $\pi^{K_{c_\eta(\pi)}, \KUold}_\eta = \{0\}$ and
    \begin{equation*}
        \pi^{K_{c_\eta(\pi)}, \KUnew}_\eta = \C \varphi.
    \end{equation*}
    If $\pi$ is neither $\psi$- nor $\psi_\xi$-generic, then $c(\psi')=-1$ and we have
    \begin{equation*}
        \mathcal{R}_\varpi \varphi (y)
        = \int_{\mathcal{O}^\times} \chi_\varpi(u) du \varphi(y)
        = 0,
    \end{equation*}
    for any $y \in \mathcal{O}$.
    Hence we have $\varphi \in \operatorname{Ker}(\mathcal{R}_\varpi)$ and
    \begin{equation*}
        \pi^{K_{c_\eta(\pi)}, \KUnew}_\eta = \{0\}.
    \end{equation*}
    For any $m > c_\eta(\pi)$, by \cite[Theorem 4.4]{ish} we have $\pi^{K_m, \RSnew}_\eta = \{0\}$ and hence
    \begin{equation*}
        \pi^{K_m, \KUnew}_\eta = \{0\}.
    \end{equation*}
    This completes the proof.
\end{proof}

\begin{proposition}\label{prop:Steinberg}
    Let $\chi$ be a trivial or quadratic character of $F^\times$ such that $\chi(-1)=\eta(-1)$.
    Then Theorems \ref{thm:main-nongeneric} and \ref{thm:main-generic} hold for $\pi=\mathit{St}_{\psi, \chi}$ and $\eta$.
\end{proposition}
\begin{proof}
    Before the proof, note that the assumption $\chi(-1)=\eta(-1)$ is equivalent to $z_\psi(\pi)=\eta(-1)$, which is a condition of the theorems we are to prove.

    We shall divide the proof into the following four cases:
    \begin{enumerate}[(1)]
        \item $\chi$ is unramified and $\eta$ is trivial;
        \item $\chi$ is ramified and $\eta=\chi|_{\mathcal{O}^\times}$;
        \item $\chi$ is unramified and $\eta$ is not trivial;
        \item $\chi$ is ramified and $\eta \neq \chi|_{\mathcal{O}^\times}$;
    \end{enumerate}

    First we consider the case (1).
    In this case, by \cite[Theorem 4.7]{ish}, there are $v_1, v_2, v_3, v_4 \in \mathit{St}_{\psi, \chi}$ such that
    \begin{align*}
        (\mathit{St}_{\psi, \chi})^{K_1}_1 & = \C v_1,                                           \\
        (\mathit{St}_{\psi, \chi})^{K_2}_1 & = \C v_1 \oplus \C v_2,                             \\
        (\mathit{St}_{\psi, \chi})^{K_3}_1 & = \C v_1 \oplus \C v_2 \oplus \C v_3 \oplus \C v_4.
    \end{align*}
    Thanks to \eqref{eq:SESstd}, we may regard $\mathit{St}_{\psi, \chi}$ as a subrepresentation of $\pi_\psi(\mu)$, where $\mu=\chi \cdot |-|^{\frac{1}{2}}$.
    Since $\operatorname{Ker}(\mathcal{R}_\varpi|_{\pi_\psi(\mu)^{K_1}_1})$ is zero by the proof of Proposition \ref{prop:unram_case}, so is $\operatorname{Ker}(\mathcal{R}_\varpi|_{(\mathit{St}_{\psi, \chi})^{K_1}_1})$.
    Hence we have
    \begin{equation*}
        (\mathit{St}_{\psi, \chi})^{K_1, \KUnew}_1 = (\mathit{St}_{\psi, \chi})^{K_1}_1 = \C v_1.
    \end{equation*}
    On the other hand, again by the proof of Proposition \ref{prop:unram_case} we have
    \begin{equation*}
        v_2 \in (\mathit{St}_{\psi, \chi})^{K_2}_1 \subset \pi_\psi(\mu)^{K_2}_1 = \pi_\psi(\mu)^{K_1}_1 \oplus \operatorname{Ker}(\mathcal{R}_\varpi|_{\pi_\psi(\mu)^{K_2}_1}),
    \end{equation*}
    which implies that there are $v_2' \in \pi_\psi(\mu)^{K_1}_1$ and $v_2^\circ \in \operatorname{Ker}(\mathcal{R}_\varpi|_{\pi_\psi(\mu)^{K_2}_1})$ such that $v_2=v_2' + v_2^\circ$.
    Proposition \ref{prop:R-operator} (b) tells us that $\mathcal{R}_\varpi$ sends $(\mathit{St}_{\psi, \chi})^{K_2}_1$ to itself.
    In particular, we see that $\mathcal{R}_\varpi v_2 = v_2'$ is an element in $\pi_\psi(\mu)^{K_1}_1 \cap (\mathit{St}_{\psi, \chi})^{K_2}_1 = \C v_1$.
    Therefore, we have
    \begin{align*}
        (\mathit{St}_{\psi, \chi})^{K_2}_1
         & = \C v_1 \oplus \C v_2                                                                                                         \\
         & = \C v_1 \oplus \C v_2^\circ                                                                                                   \\
         & \subset (\mathit{St}_{\psi, \chi})^{K_1}_1 \oplus \operatorname{Ker}(\mathcal{R}_\varpi|_{(\mathit{St}_{\psi, \chi})^{K_2}_1})
        = (\mathit{St}_{\psi, \chi})^{K_2, \KUold}_1,
    \end{align*}
    and hence
    \begin{equation*}
        (\mathit{St}_{\psi, \chi})^{K_2, \KUnew}_1 = \{0\}.
    \end{equation*}
    Similarly, one can deduce that
    \begin{equation*}
        (\mathit{St}_{\psi, \chi})^{K_3, \KUnew}_1 = \{0\}.
    \end{equation*}
    Since $(\mathit{St}_{\psi, \chi})^{K_m, \RSnew}_1$ vanishes unless $1 \leq m \leq 3$ (\cite[Corollary 4.8]{ish}), this completes the proof of the first case.

    Second, we consider the case (2).
    Let $\mathcal{J}$ denote the quotient map in the short exact sequence \eqref{eq:SESdual}.
    By the proof of Proposition \ref{prop:quad-ps_case}, we can choose a basis $\{f_1, f_0\}$ of $\pi_\psi(\mu)^{K_1}_\eta$ such that $\mathcal{R}_\varpi f_1 \neq 0$ and $\mathcal{R}_\varpi f_0 = 0$, where $\mu=\chi\cdot|-|^{-\frac{1}{2}}$.
    Since $\mathcal{R}_\varpi ((\omega_{\psi, \chi}^+)^{K_1}_\eta) = 0$ by Proposition \ref{prop:evenWeil}, we deduce that $(\omega_{\psi, \chi}^+)^{K_1}_\eta = \C f_0$ and $\mathcal{J} f_1 \neq 0$.
    Hence we have
    \begin{equation*}
        \left( \mathit{St}_{\psi, \chi} \right)^{K_1}_1 = \C \mathcal{J} f_1.
    \end{equation*}
    Since $\mathcal{R}_\varpi f_1 \neq 0$, by Proposition \ref{prop:R-operator} (c), we have $\mathcal{R}_\varpi^2 f_1 \neq 0$, i.e., $\mathcal{R}_\varpi f_1 \notin \operatorname{Ker}(\mathcal{R}_\varpi)$.
    Noting that $\mathcal{R}_\varpi f_1 \in \pi_\psi(\mu)^{K_2}_\eta$ and $(\omega_{\psi, \chi}^+)^{K_2}_\eta = (\omega_{\psi, \chi}^+)^{K_1}_\eta$, we have $\mathcal{R}_\varpi f_1 \notin (\omega_{\psi, \chi}^+)^{K_1}_\eta$, which implies that
    \begin{equation*}
        \mathcal{R}_\varpi \mathcal{J} f_1 = \mathcal{J} \mathcal{R}_\varpi f_1 \neq 0.
    \end{equation*}
    Note that $\mathcal{R}_\varpi$ commutes with $\mathcal{J}$ by definition.
    Therefore, we have
    \begin{equation*}
        \left( \mathit{St}_{\psi, \chi} \right)^{K_1, \KUnew}_1 = \left( \mathit{St}_{\psi, \chi} \right)^{K_1}_1,
    \end{equation*}
    which is 1-dimensional.
    Now Proposition \ref{prop:quad-ps_case} implies that
    \begin{equation*}
        \left( \mathit{St}_{\psi, \chi} \right)^{K_2, \KUnew}_1 = 0.
    \end{equation*}
    Since $(\mathit{St}_{\psi, \chi})^{K_m, \RSnew}_1$ vanishes unless $1 \leq m \leq 2$ (\cite[Corollary 4.8]{ish}), this completes the second case.

    We next consider the case (3).
    Let $\mathcal{M}$ denotes the quotient map in the short exact sequence \eqref{eq:SESstd}, where $\mu=\chi \cdot |-|^{\frac{1}{2}}$.
    Put $M=2c(\eta)$, which is equal to $c_\eta(\mathit{St}_{\psi, \chi})$, $c_\eta(\pi_\psi(\mu))$, and $c_\eta(\omega_{\psi, \chi}^+)$.
    By the proof of Proposition \ref{prop:general-non-min-ps_case}, we have two elements $f_+$ and $f_-$ in $\pi_\psi(\mu)$ such that
    \begin{equation*}
        \mathcal{R}_\varpi f_\pm = \pm \gamma_F(\varpi, \psi) q^{-\frac{1}{2}} \operatorname{vol}(\mathcal{O}) f_\pm
    \end{equation*}
    and
    \begin{equation*}
        \pi_\psi(\mu)^{K_M}_\eta = \C f_+ \oplus \C f_-.
    \end{equation*}
    By Propositions \ref{prop:R-operator} and \ref{prop:evenWeil}, we also know that $(\omega_{\psi, \chi}^+)^{K_M}_\eta = (\omega_{\psi, \chi}^+)^{K_{M+1}}_\eta$ is a 1-dimensional vector space spanned by an eigenvector of $\mathcal{R}_\varpi$ with nonzero eigenvalue.
    Then, since $\mathcal{R}_\varpi$ and $\mathcal{M}$ commute, one of $\mathcal{M}f_+$ or $\mathcal{M}f_-$ is zero.
    In particular, $(\mathit{St}_{\psi, \chi})^{K_M}_\eta = \operatorname{Ker}(\mathcal{M}|_{\pi_\psi(\mu)^{K_M}_\eta})$ is spanned by exactly one of $f_+$ or $f_-$.
    Thus we have
    \begin{equation*}
        \left( \mathit{St}_{\psi, \chi} \right)^{K_M, \KUnew}_\eta = \left( \mathit{St}_{\psi, \chi} \right)^{K_M}_\eta,
    \end{equation*}
    which is 1-dimensional.
    Again by the proof of Proposition \ref{prop:general-non-min-ps_case}, we know that $\pi_\psi(\mu)^{K_{M+1}}_\eta$ is a direct sum of $\pi_\psi(\mu)^{K_M}_\eta$ and a two dimensional subspace on which $\mathcal{R}_\varpi$ is zero.
    Hence we have
    \begin{equation*}
        \left( \mathit{St}_{\psi, \chi} \right)^{K_{M+1}, \KUnew}_\eta = \{0\}.
    \end{equation*}
    The assertion in the third case now follows from the fact that $(\mathit{St}_{\psi, \chi})^{K_m, \RSnew}_\eta$ is zero unless $M \leq m \leq M+1$ (\cite[Corollary 4.8]{ish}).

    Now we come to the case (4).
    Put $M=c_\eta(\mathit{St}_{\psi, \chi})$.
    In this case, by Theorem \ref{thm:conductor}, $c_\eta(\pi_\psi(\mu))=M$ and $c_\eta(\omega_{\psi, \chi}^+)=M+1$.
    Thus we have
    \begin{equation*}
        \left( \mathit{St}_{\psi, \chi} \right)^{K_M}_\eta = \pi_\psi(\mu)^{K_M}_\eta,
    \end{equation*}
    and therefore by the proof of Proposition \ref{prop:general-non-min-ps_case}, there $v_+$ and $v_-$ in $\mathit{St}_{\psi, \chi}$ such that
    \begin{equation*}
        \mathcal{R}_\varpi v_\pm = \pm \gamma_F(\varpi, \psi) q^{-\frac{1}{2}} \operatorname{vol}(\mathcal{O}) v_\pm
    \end{equation*}
    and
    \begin{equation*}
        \left( \mathit{St}_{\psi, \chi} \right)^{K_M, \KUnew}_\eta = \left( \mathit{St}_{\psi, \chi} \right)^{K_M}_\eta = \C v_+ \oplus \C v_-.
    \end{equation*}
    One can see that $\left( \mathit{St}_{\psi, \chi} \right)^{K_m, \KUnew}_\eta$ is zero for $m \neq M$ in the similar way to the third case.
    This completes the proof.
\end{proof}

\section{Supercuspidal representations}\label{sec:sc}
In this section, we shall treat the supercuspidal representations.
Although they can be realized as compactly induced representations from some compact subgroups, the model is too abstract to calculate local newforms directly.
Instead, we will make a correspondence between irreducible (genuine) supercuspidal representations of $\SL_2(F)$ and $\widetilde{\SL_2}(F)$, and then transfer the local newforms for $\SL_2(F)$ to $\widetilde{\SL_2}(F)$.
As in the last section, let $\psi$ be a nontrivial additive character on $F$ of conductor 0.

\subsection{Constructions of supercuspidal representations}\label{subsec:sc-constr}
We begin by reviewing the constructions and the classifications of irreducible supercuspidal representations of $G=\SL_2(F)$ and $\widetilde{G}=\widetilde{\SL_2}(F)$ given by Manderscheid\cite{man1} and those of $G$ given by Kutzko--Sally\cite{ks}.
Put $J_0=\SL_2(\mathcal{O})$ and $J_1 = \beta J_0 \beta^{-1}$.
For $\delta\in\{0,1\}$ and $l\geq1$, let $J_{\delta, l}$ be an compact subgroup of $J_\delta$ defined by
\begin{align*}
    J_{\delta, l} = \Set{\left( \begin{array}{cc}a&b\\ c&d \end{array} \right) \in G | a, d\in 1+\mathfrak{p}^l,\ b\in\mathfrak{p}^{l-\delta},\ c\in\mathfrak{p}^{l+\delta} }.
\end{align*}
For any integer $j$, let $N_j$ (resp. $N^{\mathrm{op}}_j$) denote the subgroup of $G$ consisting of $n(b)$ (resp. $n^{\mathrm{op}}(b)$) where $b\in\mathfrak{p}^j$.

Let $\lambda$ be a nontrivial finite dimensional representation of $J_\delta$.
Then $\lambda$ is trivial on $J_{\delta, l}$ for some $l\geq1$.
We shall write $c(\lambda)$ for the minimum of such $l$, and call it the conductor of $\lambda$.
A representation $\lambda$ is said to be strongly cuspidal if
\begin{align*}
    \Hom_{N_{\lambda, \delta}} (\lambda|_{N_{\lambda, \delta}}, \mathbf{1}_{N_{\lambda, \delta}})=\{0\},
\end{align*}
where $N_{\lambda, \delta}$ denotes $N_{c(\lambda)-2\delta-1}$ unless $\delta=c(\lambda)=1$, in which case it denotes $N_{-1}$.
Moreover, a strongly cuspidal representation $\lambda$ of $J_1$ is said to have defect 1 if
\begin{align*}
    \Hom_{N_{c(\lambda)-2}} (\lambda|_{N_{c(\lambda)-2}}, \mathbf{1}_{N_{c(\lambda)-2}})\neq\{0\}.
\end{align*}
Otherwise, we say that $\lambda$ has defect 0.
Every strongly cuspidal representation of $J_0$ is also said to have defect 0.
In this paper, following \cite{man1}, we write $d(\lambda)$ for the defect of $\lambda$.

For two representations $\lambda_1$ and $\lambda_2$ of $J_{\delta_1}$ and $J_{\delta_2}$ respectively, we say that $\lambda_1$ and $\lambda_2$ are equivalent if $\delta_1=\delta_2$ and moreover $\lambda_1$ and $\lambda_2$ are equivalent as representations of $J_{\delta_1}=J_{\delta_2}$.
Manderscheid \cite{man1} gave the following constructions and classifications of irreducible supercuspidal representations of $\SL_2(F)$.
\begin{theorem}[\cite{man1}]\label{thm:mansl2}
    For any irreducible strongly cuspidal representation $\lambda$ of $J_\delta$, the compactly induced representation $\cInd^G_{J_\delta}(\lambda)$ is an irreducible supercuspidal representation of $G$, and every irreducible supercuspidal representation of $G$ can be obtained in this way.
    Moreover, inequivalent irreducible strongly cuspidal representations give inequivalent irreducible supercuspidal representations.
\end{theorem}
Moreover, he gave such a theory also for $\widetilde{\SL_2}(F)$.
Since $J_0=K_0$, the splitting over $K_0$ which we have recalled in \S\ref{subsec:groups} is also a splitting over $J_0$.
Let us write $\bm{s}_0=\bm{s}$.
Similarly, a mapping  $h \mapsto(h,\bm{s}_1(h))$ from $J_1$ to $\widetilde{G}$ given by
\begin{align*}
    \bm{s}_1(n(b))               & =1,                                        & b & \in\mathfrak{p}^{-1},  \\
    \bm{s}_1(t(a))               & =\gamma_F(a,\psi_\varpi) = \chi_\varpi(a), & a & \in\mathcal{O}^\times, \\
    \bm{s}_1(n^{\mathrm{op}}(c)) & =1,                                        & c & \in\mathfrak{p},       \\
    \bm{s}_1(w t(\varpi))        & = 1                                        &   &
\end{align*}
is a splitting over $J_1$, and we have a following lemma:
\begin{lemma}\label{lem:spl_lv1_J_1}
    For
    \begin{align*}
        h=\left( \begin{array}{cc}a&b\\ c&d \end{array} \right) \in J_1
    \end{align*}
    such that $d\in \mathcal{O}^\times$, we have
    \begin{align*}
        \bm{s}_1(h) = \begin{cases*}
                          (\varpi c,d)_{F,2}, & if $c\neq 0$, \\
                          (\varpi,d)_{F,2},   & if $c=0$.
                      \end{cases*}
    \end{align*}
\end{lemma}
\begin{proof}
    The proof is similar to that of Lemma \ref{lem:spl_lv1}.
\end{proof}
These splittings give bijective correspondences between the representations of $J_\delta$ and the genuine representations of $\widetilde{J}_\delta$ for $\delta=0,1$.
For an irreducible strongly cuspidal representation $\lambda$ of $J_\delta$, let $\widetilde{\lambda} = \lambda\bm{s}_\delta$ denote the genuine irreducible representation $(h,\epsilon) \mapsto \epsilon\bm{s}_\delta(h)\lambda(h)$.
Then, as follows, it is known by Manderscheid \cite{man1} that irreducible genuine supercuspidal representations of $\widetilde{G}$ are classified by irreducible strongly cuspidal representations of $J_0$ and $J_1$.
\begin{theorem}[\cite{man1}]\label{thm:manmp2}
    For any irreducible strongly cuspidal representation $\lambda$ of $J_\delta$, the compactly induced representation $\cInd^{\widetilde{G}}_{\widetilde{J}_\delta}(\widetilde{\lambda})$ is an irreducible genuine supercuspidal representation of $\widetilde{G}$, and every irreducible genuine supercuspidal representation of $\widetilde{G}$ can be obtained in this way.
    Moreover, inequivalent irreducible strongly cuspidal representations give inequivalent irreducible genuine supercuspidal representations.
\end{theorem}

Next, we shall review the result of Kutzko--Sally\cite{ks} after recalling Kutzko's construction(\cite{kut1, kut2}) of minimal supercuspidal representations of $\GL_2(F)$.
See also \cite{car}, \cite{bk}, or \cite[Chapter 4]{bh}.
Put
\begin{align*}
    I=I_0 & =\Set{\left( \begin{array}{cc}a&b\\ c&d \end{array} \right) \in \GL_2(F) | a,d \in \mathcal{O}^\times,\ b\in\mathcal{O},\ c\in\mathfrak{p}}, \\
    H=H_0 & =\GL_2(\mathcal{O}).
\end{align*}
For $l\geq1$, let $I_l$ and $H_l$ be subgroups of $I$ and $H$, respectively, defined by
\begin{align*}
    I_l & =\Set{\left( \begin{array}{cc}a&b\\ c&d \end{array} \right) \in I | a,d \in 1+\mathfrak{p}^l,\ b\in\mathfrak{p}^l,\ c\in\mathfrak{p}^{l+1}}, \\
    H_l & =\Set{\left( \begin{array}{cc}a&b\\ c&d \end{array} \right) \in H | a,d \in 1+\mathfrak{p}^l,\ b,c\in\mathfrak{p}^l}.
\end{align*}
Let $Z$ be the center of $\GL_2(F)$ and $Z'$ the subgroup of $\GL_2(F)$ generated by
\begin{align*}
    \left( \begin{array}{cc}&1\\ \varpi& \end{array} \right).
\end{align*}
Consider an irreducible finite dimensional representation $\rho$ of $\Gamma=ZH$ or $Z'I$.
It is said to be very cuspidal of level $l\geq1$ if the following two conditions hold:
\begin{enumerate}[(1)]
    \item $\rho$ is trivial on $\Gamma_l$, where $\Gamma_l=H_l$ if $\Gamma=ZH$ and $I_l$ if $Z'I$;
    \item $\Hom_{N_{l-1}} (\rho|_{N_{l-1}}, \mathbf{1}_{N_{l-1}})=\{0\}$.
\end{enumerate}
Note that such $\rho$ satisfies
\begin{equation*}
    \Hom_{N^{\mathrm{op}}_{l'}} (\rho|_{N^{\mathrm{op}}_{l'}}, \mathbf{1}_{N^{\mathrm{op}}_{l'}})=\{0\},
\end{equation*}
where $l'=l-1$ if $\Gamma=ZH$ and $l$ if $Z'I$.
In this paper, we shall write $l(\rho)$ for the level of $\rho$.

An irreducible admissible representation $\varPi$ of $\GL_2(F)$ is said to be minimal if
\begin{equation*}
    c(\varPi) \leq c(\varPi \otimes(\chi\circ\det))
\end{equation*}
for any character $\chi$ of $F^\times$, where $c(\varPi)$ denotes the conductor of $\varPi$ defined by Casselman \cite{cas}.
Kutzko \cite{kut1,kut2} gave the following classifications of irreducible minimal supercuspidal representations of $\GL_2(F)$.
\begin{proposition}\label{prop:Kut-clas}
    Let $\Gamma=ZH$ or $Z'I$.
    For any irreducible very cuspidal representation $\rho$ of $\Gamma$, $\cInd^{\GL_2(F)}_\Gamma(\rho)$ is an irreducible minimal supercuspidal representation of $\GL_2(F)$, and every irreducible minimal supercuspidal representation of $\GL_2(F)$ can be obtained in this way.
    Moreover, inequivalent irreducible very cuspidal representations compactly induce inequivalent irreducible minimal supercuspidal representations.
\end{proposition}
An irreducible minimal supercuspidal representation $\cInd^{\GL_2(F)}_\Gamma(\rho)$ is said to be unramified (resp. ramified) if $\Gamma=ZH$ (resp. $Z'I$).
\begin{proposition}\label{prop:Kut-cond}
    If $\rho$ is an irreducible very cuspidal representation of $\Gamma=ZH$ (resp. $Z'I$) of level $l$, then the conductor of $\cInd^{\GL_2(F)}_\Gamma(\rho)$ is $2l$ (resp. $2l+1$).
\end{proposition}

Let us now recall the result of Kutzko--Sally\cite{ks}.
We shall begin with the unramified case.
Let $\rho$ be an irreducible very cuspidal representation of $ZH$ of level $l$, and $\lambda$ its restriction to $ZH\cap G= J_0$.
Put $\varPi=\cInd^{\GL_2(F)}_{ZH}(\rho)$, which is an irreducible supercuspidal representation of $\GL_2(F)$.
Suppose first that $\lambda$ is irreducible.
Then two representations $\cInd^G_{J_0}(\lambda)$ and $\cInd^G_{J_1}(\prescript{\beta}{}{\lambda})$ are irreducible, and the restriction of $\varPi$ to $G$ is a direct sum of those two irreducible supercuspidal representations.
The set of those two is called an unramified supercuspidal $L$-packet of cardinality two.
It is also known that $\cInd^G_{J_0}(\lambda)$ (resp. $\cInd^G_{J_1}(\prescript{\beta}{}{\lambda})$) is $\psi_a$-generic if and only if $\operatorname{ord}_F(a)+l$ is even (resp. odd).
For the last assertion, see \cite[Proposition 3.3.5]{lr}, for example.
Suppose next that $\lambda$ is reducible.
In this case, we have $l=1$ and  $\lambda=\lambda_1 \oplus \lambda_2$, where $\lambda_1$ and $\lambda_2$ are irreducible representations of dimension $(q-1)/2$.
Then the restriction of $\varPi$ to $G$ is a direct sum of four irreducible supercuspidal representations $\cInd^G_{J_0}(\lambda_1)$, $\cInd^G_{J_1}(\prescript{\beta}{}{\lambda_1})$, $\cInd^G_{J_0}(\lambda_2)$, and $\cInd^G_{J_1}(\prescript{\beta}{}{\lambda_2})$.
The set of those four is called the unramified supercuspidal $L$-packet of cardinality four.
Also, exactly one of $\cInd^G_{J_0}(\lambda_1)$ and $\cInd^G_{J_0}(\lambda_2)$ is $\psi_\varpi$-generic, and the other is $\psi_{\xi \varpi}$-generic.
Moreover, for any nontrivial additive character $\Psi$, the representation $\cInd^G_{J_1}(\prescript{\beta}{}{\lambda_i})$ is $\Psi_\varpi$-generic if and only if $\cInd^G_{J_0}(\lambda_i)$ is $\Psi$-generic, for each $i=1,2$.
In addition, for any $\Psi$, exactly one of those four representations is $\Psi$-generic.
For the last, for example see \cite[Corollary 3.3.7]{lr} for example.

Now we come to the ramified case.
Let $\rho$ be an irreducible very cuspidal representation of $Z'I$ of level $l$, and $\lambda'$ its restriction to $Z'I\cap G= I\cap G$.
Put $\varPi=\cInd^{\GL_2(F)}_{Z'I}(\rho)$, which is an irreducible supercuspidal representation of $\GL_2(F)$.
Then we have $\lambda'=\lambda'_1 \oplus \lambda'_2$, where $\lambda'_1$ and $\lambda'_2$ are irreducible representations of $I\cap G$.
The restriction of $\varPi$ to $G$ is a direct sum of two irreducible supercuspidal representations $\cInd^G_{I\cap G}(\lambda'_1)$ and $\cInd^G_{I\cap G}(\lambda'_2)$.
The set of those two is called a ramified supercuspidal $L$-packet.
Exactly one of them is $\psi_a$-generic for $a \in \varpi^\Z \mathcal{O}^{\times 2}$, and the other is so for $a \in \xi \varpi^\Z \mathcal{O}^{\times 2}$.
For the last assertion, see \cite[Proposition 3.3.9]{lr}, for example.

Every element in those supercuspidal $L$-packets is irreducible supercuspidal representation of $G$.
Conversely, every irreducible supercuspidal representation of $G$ can be obtained in this way.

Now, we finish this subsection showing that the construction of Manderscheid is compatible with that of Kutzko--Sally.
\begin{lemma}\label{lem:compatib-cpt_Man-KS}
    Let $\rho$ be an irreducible very cuspidal representation of $ZH$ or $Z'I$ of level $l$.
    \begin{enumerate}
        \item Assume that $\rho$ is a representation of $ZH$, and let $\lambda$ be an irreducible component of its restriction to $ZH\cap G= J_0$.
              Then $\lambda$ and $\prescript{\beta}{}{\lambda}$ are strongly cuspidal representations of conductor $l$ and defect $0$, of $J_0$ and $J_1$, respectively.
        \item Assume that $\rho$ is a representation of $Z'I$, and let $\lambda'$ be one of the two irreducible components of its restriction to $Z'I\cap G= I\cap G$.
              Then $\lambda = \cInd^{J_1}_{I\cap G}(\lambda')$ is an irreducible strongly cuspidal representation of $J_1$ of conductor $l+1$ and defect $1$.
    \end{enumerate}
\end{lemma}
\begin{proof}
    We shall first show (a).
    Since $\rho$ is trivial on $H_l$, it is trivial that $\lambda$ is also trivial on $J_{0, l}$.
    The condition $\Hom_{N_{l-1}} (\rho, \mathbf{1})=\{0\}$ implies that $\Hom_{N_{l-1}} (\lambda, \mathbf{1})=\{0\}$.
    Thus, $\lambda$ is a strongly cuspidal representation of conductor $l$ and defect $0$, and hance so is $\prescript{\beta}{}{\lambda}$.

    To show (b), we realize the representation $\lambda$ as the left action provided by the right translation on the vector space of functions $\varphi \colon J_1 \to V_{\lambda'}$ such that $\varphi(xy) = \lambda'(x) \varphi(y)$ for any $x \in I \cap G$ and $y \in J_1$, where $V_{\lambda'}$ denotes a fixed representation space of $\lambda'$.
    From the result of Kutzko--Sally, we know that
    \begin{equation*}
        \cInd^G_{J_1}(\lambda) = \cInd^G_{I\cap G}(\lambda')
    \end{equation*}
    is irreducible.
    The representation $\lambda$ is therefore irreducible.
    Since $\rho$ is trivial on $I_l$, $\lambda'$ is also trivial on $J_{1, l+1} \subset I_l \cap G$.
    Combining this with the fact that $J_{1, l+1}$ is a normal subgroup of $J_1$, one can see that $\lambda = \cInd^{J_1}_{I\cap G}(\lambda')$ is also trivial on $J_{1, l+1}$.
    On the other hand, since $\rho$ satisfies the condition $\Hom_{N_{l-1}} (\rho, \mathbf{1})=\{0\}$, we have $\Hom_{N_{l-1}} (\lambda', \mathbf{1})=\{0\}$.
    Hence the action of $N_{l-1}$ on the elements in $\lambda$ supported on $I \cap G$ is not trivial.
    In particular, $\lambda$ is not trivial on $J_{1, l}$.
    We have now showed that $\lambda$ is irreducible representation of $J_1$ with conductor $l+1$, and it is enough to show that there exists a nonzero function $\varphi_0 \in \lambda$ such that $\lambda(n)\varphi_0 = \varphi_0$ for any $n \in N_{l-1}$ and that there does not exist a nonzero function $\varphi \in \lambda$ such that $\lambda(n)\varphi = \varphi$ for any $n \in N_{l-2}$.
    Let $\varphi_0$ be a nonzero function in $\lambda = \cInd^{J_1}_{I\cap G}(\lambda')$ supported on
    \begin{equation*}
        \left( I \cap G \right) \cdot \left( \begin{array}{cc}&\varpi^{-1}\\ -\varpi& \end{array} \right).
    \end{equation*}
    Then, for any $a \in \mathcal{O}$, we have
    \begin{align*}
        [\lambda(n(\varpi^{l-1}a)) \varphi_0] \left( \left( \begin{array}{cc}&\varpi^{-1}\\ -\varpi& \end{array} \right) \right)
         & = \varphi_0 \left( \left( \begin{array}{cc}&\varpi^{-1}\\ -\varpi& \end{array} \right) \left( \begin{array}{cc}1&\varpi^{l-1}a\\ &1 \end{array} \right) \right)  \\
         & = \varphi_0 \left( \left( \begin{array}{cc}1&\\ -\varpi^{l+1}a&1 \end{array} \right) \left( \begin{array}{cc}&\varpi^{-1}\\ -\varpi& \end{array} \right) \right) \\
         & = \lambda'(n^{\mathrm{op}}(-\varpi^{l+1}a)) \varphi_0\left( \left( \begin{array}{cc}&\varpi^{-1}\\ -\varpi& \end{array} \right) \right)                          \\
         & = \varphi_0 \left( \left( \begin{array}{cc}&\varpi^{-1}\\ -\varpi& \end{array} \right) \right),
    \end{align*}
    since $\lambda'$ is trivial on $J_{1, l+1}$.
    Since the natural right action of $J_1$ on $(I \cap G) \backslash J_1$ is transitive, it follows from the calculation above that the support of $\lambda(n(\varpi^{l-1}a)) \varphi_0$ is the same as that of $\varphi_0$.
    Now it also follows that $\lambda(n(\varpi^{l-1}a)) \varphi_0 = \varphi_0$.
    Finally, suppose that $\varphi \in \lambda$ satisfies $\lambda(n)\varphi = \varphi$ for any $n \in N_{l-2}$.
    We are going to show that $\varphi$ is identically zero.
    By the same calculation as above, we have
    \begin{equation*}
        [\lambda(n(\varpi^{l-2}a)) \varphi] \left( \left( \begin{array}{cc}&\varpi^{-1}\\ -\varpi& \end{array} \right) \right)
        = \lambda'(n^{\mathrm{op}}(-\varpi^l a)) \varphi\left( \left( \begin{array}{cc}&\varpi^{-1}\\ -\varpi& \end{array} \right) \right),
    \end{equation*}
    for any $a \in \mathcal{O}$, and thus
    \begin{equation*}
        \varphi\left( \left( \begin{array}{cc}&\varpi^{-1}\\ -\varpi& \end{array} \right) \right) \in \left( \lambda' \right)^{N^{\mathrm{op}}_l} \subset \rho^{N^{\mathrm{op}}_l} = \{0\}.
    \end{equation*}
    We also have
    \begin{equation*}
        [\lambda(n(\varpi^{l-2}a))\varphi](n(b)) = \lambda'(\varpi^{l-2}a) \varphi(n(b)),
    \end{equation*}
    for any $a \in \mathcal{O} \cap \mathfrak{p}^{2-l}$ and $b \in \mathfrak{p}^{-1}$.
    In particular,
    \begin{equation*}
        \varphi(n(b)) \in \left( \lambda' \right)^{N_{l-1}} \subset \rho^{N_{l-1}} = \{0\}.
    \end{equation*}
    Since a set
    \begin{align*}
        \Set{n(b) | b \in \mathfrak{p^{-1}}/\mathcal{O}} \cup \Set{\left( \begin{array}{cc}&\varpi^{-1}\\ -\varpi& \end{array} \right)}
    \end{align*}
    is a representative system for $(I \cap G) \backslash J_1$, this leads to $\varphi = 0$.
\end{proof}
Moreover, the indices $c(\lambda)$, $d(\lambda)$, and $l(\rho)$ are related as follows.
\begin{proposition}\label{prop:compatib-indices_Man-KS}
    Let $\pi$ be an irreducible supercuspidal representation of $G$.
    Assume that $\rho$ is an irreducible very cuspidal representation of $\Gamma\in\{ZH, Z'I\}$ such that $\pi$ is a subrepresentation of $\cInd^{\GL_2(F)}_\Gamma(\rho)$.
    In addition, assume that $\lambda$ is an irreducible strongly cuspidal representation of $J_\delta$ ($\delta=0,1$), such that $\pi=\cInd^G_{J_\delta}(\lambda)$.
    Then we have
    \begin{align*}
        l(\rho) & =c(\lambda)-d(\lambda),    \\
        \Gamma  & =\begin{cases*}
                       ZH,  & if $d(\lambda)=0$, \\
                       Z'I, & if $d(\lambda)=1$.
                   \end{cases*}
    \end{align*}
\end{proposition}
\begin{proof}
    The proposition follows from Lemma\ref{lem:compatib-cpt_Man-KS} and Theorem \ref{thm:mansl2}.
\end{proof}

\subsection{Explicit mapping of invariant subspaces}\label{subsec:preliminary (tentative)}
In the last subsection, we have reviewed that irreducible supercuspidal representations of $G=\SL_2(F)$ and $\widetilde{G}=\widetilde{\SL_2}(F)$ can be constructed in parallel.
For the proof of the main theorems, we need a more detailed correspondence.
In this subsection, we will construct an explicit mapping between the subspace $\pi^{K_m}_\eta$ of an irreducible supercuspidal representation $\pi$ of $\widetilde{G}$ and that of $G$.
\begin{lemma}\label{lem:coset_sc0}
    Set $g_i=t(\varpi^i)$ and $g_{i,j,z} = n(z \varpi^{-j}) t(\varpi^{i-j})$, for $i,j \in \Z$ and $z \in \mathcal{O}^\times$.
    For $m \geq 0$, put
    \begin{equation*}
        \mathscr{R}_{0,m} =
        \begin{dcases*}
            \Set{ g_i | i \geq 0},                                                                  & if $m=0$,      \\
            \Set{ g_i | i \in \Z},                                                                  & if $m=1$,      \\
            \Set{ g_i | i \in \Z} \cup \Set{ g_{i,j,z} | i \in \Z_{>0},\ z=1,\xi,\ j=1,\ldots,N-1}, & if $m \geq 2$,
        \end{dcases*}
    \end{equation*}
    where $N=\min(2i, m)$.
    Then $\mathscr{R}_{0,m}$ is a set of complete representatives of $K_m \backslash G / J_0$.
\end{lemma}
\begin{proof}
    Since $J_0=K_0=\SL_2(\mathcal{O})$, the assertion for $m=0$ is well-known.
    Suppose that $m>0$.
    We have
    \begin{align*}
        G & =\bigsqcup_{i \geq 0} K_0 t(\varpi^i) J_0                                                          \\
          & =K_m 1_2 J_0 \sqcup \bigsqcup_{i \geq 1} \bigcup_{h \in K_m \backslash K_0} K_m h t(\varpi^i) J_0.
    \end{align*}
    Here, whether the $(2,2)$-component of $h$ is in $\mathcal{O}^\times$ or $\mathfrak{p}$ does not depend on the choice of $h$, but of the coset $K_m h t(\varpi^i) K_0$.
    Indeed, let $\gamma \in K_m$, $h' \in J_0$, and $h'' \in K_0$ satisfy $\gamma h t(\varpi^i) h' = h'' t(\varpi^i)$, and we shall write
    \begin{align*}
         & \gamma = \left( \begin{array}{cc}*&*\\ u&v \end{array} \right),      &
         & h = \left( \begin{array}{cc}a&b\\ c&d \end{array} \right),           &
         & h' = \left( \begin{array}{cc}a'&b'\\ c'&d' \end{array} \right),      &
         & h'' = \left( \begin{array}{cc}a''&b''\\ c''&d'' \end{array} \right). &
    \end{align*}
    One can then deduce that $\varpi^{-2i} c' \in \mathcal{O}$.
    Since $i>0$ now, we have $c' \in \mathfrak{p}$ and hence $d' \in \mathcal{O}^\times$.
    We also have
    \begin{equation*}
        d'' = (ua+vc)\varpi^{2i}b' + ubd' + vdd'.
    \end{equation*}
    Noting that $i$ and $m$ are positive, it follows that $d'' \in d \mathcal{O}^\times + \mathfrak{p}$.
    Therefore, $d \in \mathcal{O}^\times$ if and only if $d'' \in \mathcal{O}^\times$.

    If $d \in \mathcal{O}^\times$, then
    \begin{equation*}
        \left( \begin{array}{cc}a&b\\ c&d \end{array} \right) = \left( \begin{array}{cc}d^{-1}&b\\ &d \end{array} \right) \left( \begin{array}{cc}1&\\ cd^{-1}&1 \end{array} \right) \in K_m n^{\mathrm{op}}(cd^{-1}).
    \end{equation*}
    If $d \in \mathfrak{p}$, then $c \in \mathcal{O}^\times$ and hence
    \begin{equation*}
        \left( \begin{array}{cc}a&b\\ c&d \end{array} \right) = \left( \begin{array}{cc}b&-a\\ d&-c \end{array} \right) \left( \begin{array}{cc}&1\\-1 & \end{array} \right) \in K_m n^{\mathrm{op}}(-c^{-1}d) w.
    \end{equation*}
    It thus follows that
    \begin{equation*}
        G  = K_m 1_2 J_0 \sqcup \bigsqcup_{i \geq 1} \left( \bigcup_{y \in \mathcal{O}/\mathfrak{p}^m} K_m n^{\mathrm{op}}(y) t(\varpi^i) J_0 \sqcup \bigcup_{y \in \mathfrak{p}/\mathfrak{p}^m} K_m n^{\mathrm{op}}(y) w t(\varpi^i) J_0 \right).
    \end{equation*}
    Since $n^{\mathrm{op}}(y) t(\varpi^i) = t(\varpi^i) n^{\mathrm{op}}(\varpi^{2i}y)$ and $n^{\mathrm{op}}(\varpi^{2i}y) \in J_0$ for any $i \geq 1$ and $y \in \mathcal{O}$, this equals
    \begin{equation*}
        \bigsqcup_{i \geq 0} K_m t(\varpi^i) J_0 \sqcup \bigsqcup_{i \geq 1} \bigcup_{y \in \mathfrak{p}/\mathfrak{p}^m} K_m n^{\mathrm{op}}(y) w t(\varpi^i) J_0.
    \end{equation*}
    Moreover, since $w \in J_0$, this equals
    \begin{align*}
         & \bigsqcup_{i \geq 0} K_m t(\varpi^i) J_0 \sqcup \bigsqcup_{i \geq 1} \bigcup_{y \in \mathfrak{p}/\mathfrak{p}^m} K_m n^{\mathrm{op}}(y) w t(\varpi^i) w J_0             \\
         & =\bigsqcup_{i \geq 0} K_m t(\varpi^i) J_0 \sqcup \bigsqcup_{i \geq 1} \bigcup_{y \in \mathfrak{p}/\mathfrak{p}^m} K_m t(\varpi^{-i}) n^{\mathrm{op}}(\varpi^{-2i}y) J_0 \\
         & =\bigsqcup_{i \geq 0} K_m t(\varpi^i) J_0 \sqcup \bigsqcup_{i \geq 1} \bigcup_{y \in \mathfrak{p}/\mathfrak{p}^N} K_m t(\varpi^{-i}) n^{\mathrm{op}}(\varpi^{-2i}y) J_0 \\
         & =\bigsqcup_{i \geq 0} K_m t(\varpi^i) J_0 \sqcup \bigsqcup_{i \geq 1} \bigcup_{y \in \mathfrak{p}/\mathfrak{p}^N} K_m n^{\mathrm{op}}(y) t(\varpi^{-i}) J_0,
    \end{align*}
    where $N=\min(2i, m)$.
    Now the assertion for $m=1$ follows since $\mathfrak{p}/\mathfrak{p}^N=\{0\}$ in this case.

    Assume that $m \geq 2$.
    Fix $i \geq 1$ and consider $\cup_{y \in \mathfrak{p}/\mathfrak{p}^N} K_m n^{\mathrm{op}}(y) t(\varpi^{-i}) J_0$.
    Since $t(a)$ belongs to both $K_m$ and $J_0$ for any $a\in \mathcal{O}^\times$, we have
    \begin{equation*}
        \bigcup_{y \in \mathfrak{p}/\mathfrak{p}^N} K_m n^{\mathrm{op}}(y) t(\varpi^{-i}) J_0
        = K_m t(\varpi^{-i}) J_0 \cup \bigcup_{j=1}^{N-1} \bigcup_{z\in\{1,\xi\}} K_m n^{\mathrm{op}}(z\varpi^j) t(\varpi^{-i}) J_0.
    \end{equation*}
    This union is disjoint.
    Indeed, suppose that $y$ and $y' \in \mathfrak{p}$ with $1 \leq \operatorname{ord}(y), \operatorname{ord}(y') \leq N-1$ satisfy
    \begin{equation*}
        \gamma n^{\mathrm{op}}(y) t(\varpi^{-i}) h = n^{\mathrm{op}}(y') t(\varpi^{-i}),
    \end{equation*}
    for some $\gamma \in K_m$ and $h \in J_0$.
    If we write
    \begin{align*}
         & \gamma = \left( \begin{array}{cc}*&*\\ u&v \end{array} \right), &
         & h = \left( \begin{array}{cc}a&b\\ c&d \end{array} \right),      &
    \end{align*}
    then a direct calculation shows that
    \begin{equation}\label{eq:lemcoset1}
        au + yav + \varpi^{2i}cv =y'
    \end{equation}
    and
    \begin{equation}\label{eq:lemcoset2}
        \varpi^{-2i}bu + \varpi^{-2i}bvy + vd = 1.
    \end{equation}
    Moreover, since
    \begin{equation*}
        t(\varpi^{-i}) h t(\varpi^{-i})^{-1} = n^{\mathrm{op}}(y)^{-1} \gamma^{-1} n^{\mathrm{op}}(y') \in K_0,
    \end{equation*}
    we have $b \in \mathfrak{p}^{2i}$, and hence $ad = 1-bc \in 1+\mathfrak{p}$.
    It follows from \eqref{eq:lemcoset2} that $vd \in 1+\mathfrak{p}$.
    Thus $av$ is an element in $1 + \mathfrak{p}$ and in particular $\mathcal{O}^{\times 2}$.
    Then it follows from \eqref{eq:lemcoset1} that
    \begin{equation*}
        y' \in y \mathcal{O}^{\times 2}.
    \end{equation*}
    Therefore, we have
    \begin{equation*}
        G = \bigsqcup_{i \in \Z} K_m t(\varpi^i) J_0 \sqcup \bigsqcup_{i \geq 1} \bigsqcup_{j=1}^{N-1} \bigsqcup_{z\in\{1,\xi\}} K_m n^{\mathrm{op}}(z\varpi^j) t(\varpi^{-i}) J_0.
    \end{equation*}
    Since
    \begin{equation*}
        n^{\mathrm{op}}(z\varpi^j) t(\varpi^{-i}) J_0 = n(z^{-1} \varpi^{-j}) t(\varpi^{-j+i}) J_0,
    \end{equation*}
    this completes the proof.
\end{proof}

\begin{lemma}\label{lem:coset_sc1}
    For $m \geq 0$, put
    \begin{equation*}
        \mathscr{R}_{1,m} =
        \begin{dcases*}
            \Set{ g_i | i \geq 1},                                                                  & if $m=0$,      \\
            \Set{ g_i | i \in \Z},                                                                  & if $m=1$,      \\
            \Set{ g_i | i \in \Z} \cup \Set{ g_{i,j,z} | i \in \Z_{>1},\ z=1,\xi,\ j=1,\ldots,N-1}, & if $m \geq 2$,
        \end{dcases*}
    \end{equation*}
    where $N=\min(2i-1, m)$.
    Then $\mathscr{R}_{1,m}$ is a set of complete representatives of $K_m \backslash G / J_1$.
\end{lemma}
\begin{proof}
    Notice that $w t(\varpi)$ is in $J_1$, instead of $w$.
    Then the proof is similar to that of Lemma \ref{lem:coset_sc0}.
\end{proof}
For $\delta =0, 1$, we shall write $\widetilde{\mathscr{R}}_{\delta, m}$ for a complete system of representatives of $K_m \backslash \widetilde{G} / \widetilde{J}_\delta$ consisting of $(g,1) \in \widetilde{G}$ where $g \in \mathscr{R}_{\delta, m}$.
By abuse of notation, we sometimes identify $\widetilde{\mathscr{R}}_{\delta, m}$ with $\mathscr{R}_{\delta, m}$ via $(g,1) \leftrightarrow g$.
Let $\pi=\cInd^{\widetilde{G}}_{\widetilde{J}_\delta}(\widetilde{\lambda})$ be an irreducible supercuspidal representation of $\widetilde{G}$, and $\eta$ a character of $\mathcal{O}^\times$.
Notice that $\pi^{K_m}_\eta$ can be identified with $\Hom_{K_m}(\eta, \pi)$ via $f(1) \leftrightarrow f$.
Also notice that $( \prescript{g}{}{\widetilde{\lambda}} )^{K_m \cap \prescript{g}{}{(\widetilde{J}_\delta)}}_\eta$ can be identified with $\Hom_{K_m \cap \prescript{g}{}{(\widetilde{J}_\delta)}}(\eta, \prescript{g}{}{\widetilde{\lambda}})$ for any $g \in G$, in a similar manner.
By the argument for \cite[Corollary 2.7 (2)]{yam}, for any $m\geq0$, we have a linear isomorphism
\begin{alignat*}{3}
    \mathcal{G}_m \colon & \pi^{K_m}_\eta &  & \overset{\simeq}{\longrightarrow} &  & \bigoplus_{g \in \widetilde{\mathscr{R}}_{\delta, m}} \left( \prescript{g}{}{\widetilde{\lambda}} \right)^{K_m \cap \prescript{g}{}{(\widetilde{J}_\delta)}}_\eta, \\
                         & \phi           &  & \mapsto                           &  & \mathcal{G}_m \phi = \left( (\mathcal{G}_m \phi)_g \right)_{g \in \widetilde{\mathscr{R}}_{\delta, m}},
\end{alignat*}
given by $(\mathcal{G}_m \phi)_g= \phi(g^{-1})$.
Let $\tau=\cInd^G_{J_\delta}(\lambda)$.
We define $\eta'$ to be $\eta$ itself if $\delta = 0$ and $\eta \cdot \chi_\varpi$ if $\delta=1$.
Namely, $\eta'=\eta \cdot \chi_\varpi^\delta$.
Then similarly, we have a linear isomorphism from $\tau^{K_m}_{\eta'}$ to $\bigoplus_{g \in \mathscr{R}_{\delta, m}} (\prescript{g}{}{\lambda})^{K_m \cap \prescript{g}{}{(J_\delta)}}_{\eta'}$, for which we shall write $\mathcal{G}'_m$.
If $m=0$, both $\pi^{K_0}_\eta$ and $\tau^{K_0}_{\eta'}$ are zero, in particular isomorphic.
If $m \geq 1$, thanks to Lemmas \ref{lem:spl_lv1} and \ref{lem:coset_sc0} (resp. \ref{lem:spl_lv1_J_1} and \ref{lem:coset_sc1}), we can calculate $\bm{s}_0(\gamma^g)$ (resp. $\bm{s}_1(\gamma^g)$) to obtain $\bm{s}_0(\gamma^g) \bm{s}(\gamma) = 1$ (resp. $\bm{s}_1(\gamma^g)\bm{s}(\gamma) = \chi_\varpi(\gamma)$), for any $g \in \mathscr{R}_{0,m}$ (resp. $\mathscr{R}_{1,m}$) and $\gamma \in K_m \cap \prescript{g}{}{(\widetilde{J}_0)}$ (resp. $K_m \cap \prescript{g}{}{(\widetilde{J}_1)}$).
Hence for any $g \in \mathscr{R}_{\delta, m}$, we have
\begin{equation*}
    \left( \prescript{g}{}{\widetilde{\lambda}} \right)^{K_m \cap \prescript{g}{}{(\widetilde{J}_\delta)}}_\eta
    = \left( \prescript{g}{}{\lambda} \right)^{K_m \cap \prescript{g}{}{(J_\delta)}}_{\eta'}.
\end{equation*}
Indeed, if $\upsilon$ is an element in the left hand side, for any $\gamma = (\gamma, \bm{s}(\gamma)) \in K_m\cap \prescript{g}{}{(\widetilde{J}_\delta)}$, we have
\begin{align*}
    \prescript{g}{}{\lambda}(\gamma) \upsilon
     & =\bm{s}_\delta(\gamma^g) \bm{s}(\gamma) \cdot \prescript{g}{}{\widetilde{\lambda}}(\gamma) \upsilon \\
     & =\bm{s}_\delta(\gamma^g) \bm{s}(\gamma) \cdot \eta(\gamma) \upsilon                                 \\
     & =\eta'(\gamma) \upsilon.
\end{align*}
At the first equality, we use an equation $(\gamma, \bm{s}(\gamma))^g = (\gamma^g, \bm{s}(\gamma))$ which follows from \cite[(2.15)]{szpthesis}.
Then we have a linear isomorphism $(\mathcal{G}'_m)^{-1} \circ \mathcal{G}_m \colon \pi^{K_m}_\eta \to \tau^{K_m}_{\eta'}$.
Let $\mathscr{G}_m$ denote it.
Now we can see that the conductors of supercuspidal representations of $\widetilde{G}$ are greater than 1:
\begin{proposition}\label{prop:sc_<2}
    For any irreducible supercuspidal representation $\tau$ of $G$ and any character $\mu$ of $\mathcal{O}^\times$, the subspace $\tau^{K_m}_\mu$ is zero if $m=0,1$.
    Moreover, for any irreducible genuine supercuspidal representation $\pi$ of $\widetilde{G}$ and any character $\eta$ of $\mathcal{O}^\times$, the subspace $\pi^{K_m}_\eta$ is zero if $m=0,1$.
\end{proposition}
\begin{proof}
    The first assertion is a corollary of \cite[Propositions 3.3.4, 3.3.6, and 3.3.8]{lr}.
    The second follows from the first via the isomorphism $\mathscr{G}_m$.
\end{proof}

For an irreducible representation $(\tau, V_\tau)$ of $G$ and a character $\mu$ of $\mathcal{O}^\times$, let $\mathcal{R}_\varpi'$ be a linear operator on $\tau^{K_\infty}_\mu$ defined by
\begin{equation*}
    \mathcal{R}_\varpi' v     = \int_{\mathcal{O}^\times} (u,\varpi)_{F,2} \tau \left( \left( \begin{array}{cc}1&\varpi^{-1} u\\ 0&1 \end{array} \right) \right) v du,
\end{equation*}
for $v \in V_\tau$.
\begin{lemma}\label{lem:comm_R-operator_sc}
    For any $m \geq 0$, we have
    \begin{equation*}
        \mathscr{G}_m \circ \mathcal{R}_\varpi = \mathcal{R}_\varpi' \circ \mathscr{G}_m,
    \end{equation*}
    i.e., the following diagram commutes:
    \[\begin{tikzcd}
            \pi^{K_m}_\eta \arrow[d, "\mathcal{R}_\varpi"] \arrow[r, "\mathscr{G}_m"] &  \tau^{K_m}_{\eta'} \arrow[d, "\mathcal{R}_\varpi'"]\\
            \pi^{K_m}_\eta \arrow[r, "\mathscr{G}_m"] & \tau^{K_m}_{\eta'}.
        \end{tikzcd}\]
\end{lemma}
\begin{proof}
    We know from Proposition \ref{prop:sc_<2} that $\pi^{K_m}_\eta$ and $\tau^{K_m}_{\eta'}$ are zero and hence the lemma is trivial unless $m \geq 2$.
    Suppose that $m \geq 2$.
    We are going to calculate $\mathcal{G}_m \circ \mathcal{R}_\varpi$ explicitly, and show that there exists a linear endomorphism $\mathcal{R}_{\varpi,m}$ of
    \begin{equation*}
        \bigoplus_{g \in \widetilde{\mathscr{R}}_{\delta, m}} \left( \prescript{g}{}{\widetilde{\lambda}} \right)^{K_m \cap \prescript{g}{}{(\widetilde{J}_\delta)}}_\eta = \bigoplus_{g \in \mathscr{R}_{\delta, m}} (\prescript{g}{}{\lambda})^{K_m \cap \prescript{g}{}{(J_\delta)}}_{\eta'},
    \end{equation*}
    such that the following diagram commutes:
    \[\begin{tikzcd}
            \pi^{K_m}_\eta \arrow[d, "\mathcal{R}_\varpi"] \arrow[r, "\mathcal{G}_m"] & \bigoplus_{g \in \widetilde{\mathscr{R}}_{\delta, m}} \left( \prescript{g}{}{\widetilde{\lambda}} \right)^{K_m \cap \prescript{g}{}{(\widetilde{J}_\delta)}}_\eta \arrow[d, "\mathcal{R}_{\varpi,m}"] \arrow[r, equal] & \bigoplus_{g \in \mathscr{R}_{\delta, m}} (\prescript{g}{}{\lambda})^{K_m \cap \prescript{g}{}{(J_\delta)}}_{\eta'} \arrow[d, "\mathcal{R}_{\varpi,m}"] & \tau^{K_m}_{\eta'} \arrow[d, "\mathcal{R}_\varpi'"] \arrow[l, "\mathcal{G}'_m"]\\
            \pi^{K_m}_\eta \arrow[r, "\mathscr{G}_m"] & \bigoplus_{g \in \widetilde{\mathscr{R}}_{\delta, m}} \left( \prescript{g}{}{\widetilde{\lambda}} \right)^{K_m \cap \prescript{g}{}{(\widetilde{J}_\delta)}}_\eta \arrow[r, equal] & \bigoplus_{g \in \mathscr{R}_{\delta, m}} (\prescript{g}{}{\lambda})^{K_m \cap \prescript{g}{}{(J_\delta)}}_{\eta'} & \tau^{K_m}_{\eta'} \arrow[l, "\mathcal{G}'_m"].
        \end{tikzcd}\]
    Recall that $\mathscr{R}_{\delta, m}$ is explicitly given in Lemmas \ref{lem:coset_sc0} and \ref{lem:coset_sc1}.
    Let $\phi \in \pi^{K_m}_\eta$.
    We have
    \begin{equation*}
        \left( \mathcal{G}_m \mathcal{R}_\varpi \phi \right)_g = \int_{\mathcal{O}^\times} (-u,\varpi)_{F,2} \phi \left( \left( ( n(\varpi^{-1} u), 1 ) (g,1) \right)^{-1} \right) du.
    \end{equation*}
    Straightforward calculations show that $( n(\varpi^{-1} u), 1 ) (g,1)$ for $u \in \mathcal{O}^\times$ and $g \in \mathscr{R}_{\delta, m}$ are given as follows.
    \begin{itemize}
        \item If $g=g_i$ with $i \geq \delta$, then
              \begin{equation*}
                  ( n(\varpi^{-1} u), 1 ) (g_i,1) = (t(x), 1) (g_{i+1, 1, z}, 1) (t(x^{-1}), 1),
              \end{equation*}
              where $z\in\{1,\xi\}$ and $x \in \mathcal{O}^\times$ such that $u=zx^2$.
        \item If $g=g_i$ with $i < \delta$, then
              \begin{equation*}
                  ( n(\varpi^{-1} u), 1 ) (g_i,1) = (g_i, 1) (n(\varpi^{-2i-1}u), 1).
              \end{equation*}
        \item If $g=g_{i,1,z}$ with $i \geq 1$ and $u+z \in \mathfrak{p}$, then
              \begin{equation*}
                  ( n(\varpi^{-1} u), 1 ) (g_{i,1,z},1) = (n(\varpi^{-1}(u+z)), 1) (g_{i-1}, 1).
              \end{equation*}
        \item If $g=g_{i,1,z}$ with $i \geq 1$ and $u+z \in \mathcal{O}^\times$, then
              \begin{equation*}
                  ( n(\varpi^{-1} u), 1 ) (g_{i,1,z},1) = (t(x), 1) (g_{i,1,z'}, 1) (t(x^{-1}), 1),
              \end{equation*}
              where $z'\in\{1, \xi\}$ and $x \in \mathcal{O}^\times$ such that $u+z = z' x^2$.
        \item If $g=g_{i,j,z}$ with $j \geq 2$, $i \geq 1$, and $z \in \{1,\xi\}$, then
              \begin{equation*}
                  ( n(\varpi^{-1} u), 1 ) (g_{i,j,z},1) = (t(x), 1) ( g_{i,j,z} , 1) (t(x^{-1}), 1),
              \end{equation*}
              where $x \in \mathcal{O}^\times$ such that $z + \varpi^{j-1} u = z x^2$.
    \end{itemize}
    In particular, there exist $h_g(u) \in J_\delta$, $\gamma_g(u) \in K_m$, and $\breve{g}(u) \in \mathscr{R}_{\delta, m}$ such that
    \begin{equation}\label{eq:pf(comm_R)}
        \left( ( n(\varpi^{-1} u), 1 ) \ (g,1) \right)^{-1}
        = (h_g(u), \bm{s}_\delta(h_g(u))) \ (\breve{g}(u),1)^{-1} \ (\gamma_g(u), \bm{s}(\gamma_g(u)) \chi_\varpi^\delta(\gamma_g(u))).
    \end{equation}
    Notice also that $h_g(u)$, $\gamma_g(u)$, and $\breve{g}(u)$ are all continuous in $(g, u)$ if we equip $\mathscr{R}_{\delta, m}$ with the discrete topology.
    Therefore, we have $\mathcal{G}_m \circ \mathcal{R}_\varpi = \mathcal{R}_{\varpi, m} \circ \mathcal{G}_m$ if we define a linear endomorphism
    \begin{alignat*}{3}
        \mathcal{R}_{\varpi, m} \colon & \bigoplus_{g \in \widetilde{\mathscr{R}}_{\delta, m}} \left( \prescript{g}{}{\widetilde{\lambda}} \right)^{K_m \cap \prescript{g}{}{(\widetilde{J}_\delta)}}_\eta &  & {}\longrightarrow{} &  & \bigoplus_{g \in \widetilde{\mathscr{R}}_{\delta, m}} \left( \prescript{g}{}{\widetilde{\lambda}} \right)^{K_m \cap \prescript{g}{}{(\widetilde{J}_\delta)}}_\eta, \\
                                       & \bm{\upsilon}=(\upsilon_g)_{g \in \widetilde{\mathscr{R}}_{\delta, m}}                                                                                            &  & {}\mapsto{}         &  & \left( \mathcal{R}_g \bm{\upsilon} \right)_{g \in \widetilde{\mathscr{R}}_{\delta, m}},
    \end{alignat*}
    by
    \begin{equation*}
        \mathcal{R}_g \bm{\upsilon} = \int_{\mathcal{O}^\times} ( u, \varpi )_{F,2} \eta'(\gamma_g(u)) \lambda(h_g(u))\upsilon_{\breve{g}(u)} du.
    \end{equation*}

    From \eqref{eq:pf(comm_R)}, we have
    \begin{equation*}
        ( n(\varpi^{-1} u)g)^{-1}
        = h_g(u) \ \breve{g}(u)^{-1} \ \gamma_g(u).
    \end{equation*}
    Hence we also have $\mathcal{G}'_m \circ \mathcal{R}_\varpi' = \mathcal{R}_{\varpi, m} \circ \mathcal{G}'_m$.
    Then we have two commutative diagrams
    \[\begin{tikzcd}
            \pi^{K_m}_\eta \arrow[d, "\mathcal{R}_\varpi"] \arrow[r, "\mathcal{G}_m"] & \bigoplus_{g \in \widetilde{\mathscr{R}}_{\delta, m}} \left( \prescript{g}{}{\widetilde{\lambda}} \right)^{K_m \cap \prescript{g}{}{(\widetilde{J}_\delta)}}_\eta \arrow[d, "\mathcal{R}_{\varpi,m}"]  & \bigoplus_{g \in \mathscr{R}_{\delta, m}} (\prescript{g}{}{\lambda})^{K_m \cap \prescript{g}{}{(J_\delta)}}_{\eta'} \arrow[d, "\mathcal{R}_{\varpi,m}"] & \tau^{K_m}_{\eta'} \arrow[d, "\mathcal{R}_\varpi'"] \arrow[l, "\mathcal{G}'_m"]\\
            \pi^{K_m}_\eta \arrow[r, "\mathscr{G}_m"] & \bigoplus_{g \in \widetilde{\mathscr{R}}_{\delta, m}} \left( \prescript{g}{}{\widetilde{\lambda}} \right)^{K_m \cap \prescript{g}{}{(\widetilde{J}_\delta)}}_\eta, & \bigoplus_{g \in \mathscr{R}_{\delta, m}} (\prescript{g}{}{\lambda})^{K_m \cap \prescript{g}{}{(J_\delta)}}_{\eta'} & \tau^{K_m}_{\eta'} \arrow[l, "\mathcal{G}'_m"].
        \end{tikzcd}\]
    This completes the proof.
\end{proof}
Similarly, we also have the next lemma.
For an irreducible representation $(\pi, V_\pi)$ of $\widetilde{G}$, we introduce a linear operator $\mathcal{Q}_\varpi$ defined by
\begin{equation*}
    \mathcal{Q}_\varpi v     = \int_{\mathcal{O}^\times} (u,\varpi)_{F,2} \pi \left( \left( \left( \begin{array}{cc}1&\varpi^{-2} u\\ 0&1 \end{array} \right), 1 \right) \right) v du,
\end{equation*}
on $V_\pi$.
Then for any integer $m$, we have
\begin{equation}\label{eq:prop_Q-operator}
    \beta_2(\pi^{K_m}_\eta) \subset \operatorname{Ker}(\mathcal{Q}_\varpi),
\end{equation}
which can be shown in the same way as (a) of Proposition \ref{prop:R-operator}.
For an irreducible representation $(\tau, V_\tau)$ of $G$, we introduce a linear operator $\mathcal{Q}_\varpi'$ defined by
\begin{equation*}
    \mathcal{Q}_\varpi' v     = \int_{\mathcal{O}^\times} (u,\varpi)_{F,2} \tau \left( \left( \begin{array}{cc}1&\varpi^{-2} u\\ 0&1 \end{array} \right) \right) v du,
\end{equation*}
on $V_\tau$.
\begin{lemma}\label{lem:comm_Q-operator_sc}
    For any $m \geq 0$, we have
    \begin{equation*}
        \mathscr{G}_m \circ \mathcal{Q}_\varpi = \mathcal{Q}_\varpi' \circ \mathscr{G}_m,
    \end{equation*}
    i.e., the following diagram commutes:
    \[\begin{tikzcd}
            \pi^{K_m}_\eta \arrow[d, "\mathcal{Q}_\varpi"] \arrow[r, "\mathscr{G}_m"] &  \tau^{K_m}_{\eta'} \arrow[d, "\mathcal{Q}_\varpi'"]\\
            \pi^{K_m}_\eta \arrow[r, "\mathscr{G}_m"] & \tau^{K_m}_{\eta'}.
        \end{tikzcd}\]
\end{lemma}
\begin{proof}
    The idea of the proof is as same as Lemma \ref{lem:comm_R-operator_sc}.
    Since the assertion is trivial unless $m \geq 2$, we suppose it.
    For $u \in \mathcal{O}^\times$ and $g \in \mathscr{R}_{\delta, m}$, we can calculate $( n(\varpi^{-2} u), 1 ) (g,1)$.
    \begin{itemize}
        \item If $g=g_i$ with $i \geq 0$, then
              \begin{equation*}
                  ( n(\varpi^{-2} u), 1 ) (g_i,1) = (t(x), 1) (g_{i+2, 2, z}, 1) (t(x^{-1}), 1),
              \end{equation*}
              where $z\in\{1,\xi\}$ and $x \in \mathcal{O}^\times$ such that $u=zx^2$.
        \item If $g=g_i$ with $i < 0$, then
              \begin{equation*}
                  ( n(\varpi^{-2} u), 1 ) (g_i,1) = (g_i, 1) (n(\varpi^{-2i-2}u), 1).
              \end{equation*}
        \item If $g=g_{i,1,z}$ with $i \geq 1$, then
              \begin{equation*}
                  ( n(\varpi^{-2} u), 1 ) (g_{i,1,z},1) = (t(x), 1) (g_{i+1,2,z'}, 1) (t(x^{-1}), 1),
              \end{equation*}
              where $z'\in\{1,\xi\}$ and $x \in \mathcal{O}^\times$ such that $z\varpi+u=z'x^2$.
        \item If $g=g_{i,2,z}$ with $i \geq 1$ and $u+z \in \mathfrak{p}^2$, then
              \begin{equation*}
                  ( n(\varpi^{-2} u), 1 ) (g_{i,2,z},1) = (n(\varpi^{-2}(u+z)), 1) (g_{i-2}, 1).
              \end{equation*}
        \item If $g=g_{i,2,z}$ with $i \leq 1+\delta$ and $u+z \in \mathfrak{p} \setminus \mathfrak{p}^2$, then
              \begin{equation*}
                  ( n(\varpi^{-2} u), 1 ) (g_{i,2,z},1) = (g_{i-2}, 1) (n(\varpi^{2-2i}(u+z)), 1).
              \end{equation*}
        \item If $g=g_{i,2,z}$ with $i \geq 2+\delta$ and $u+z \in \mathfrak{p} \setminus \mathfrak{p}^2$, then
              \begin{equation*}
                  ( n(\varpi^{-2} u), 1 ) (g_{i,2,z},1) = (t(x), 1) (g_{i-1,1,z'}, 1) (t(x^{-1}), 1),
              \end{equation*}
              where $z'\in\{1, \xi\}$ and $x \in \mathcal{O}^\times$ such that $u+z = \varpi z' x^2$.
        \item If $g=g_{i,2,z}$ with $i \geq 1$ and $u+z \in \mathcal{O}^\times$, then
              \begin{equation*}
                  ( n(\varpi^{-2} u), 1 ) (g_{i,2,z},1) = (t(x), 1) (g_{i,2,z'}, 1) (t(x^{-1}), 1),
              \end{equation*}
              where $z'\in\{1, \xi\}$ and $x \in \mathcal{O}^\times$ such that $u+z = z' x^2$.
        \item If $g=g_{i,j,z}$ with $j \geq 3$, $i \geq 1$, and $z \in \{1,\xi\}$, then
              \begin{equation*}
                  ( n(\varpi^{-2} u), 1 ) (g_{i,j,z},1) = (t(x), 1) ( g_{i,j,z} , 1) (t(x^{-1}), 1),
              \end{equation*}
              where $x \in \mathcal{O}^\times$ such that $z + \varpi^{j-2} u = z x^2$.
    \end{itemize}
    In particular, there exist $h_{g,\mathcal{Q}}(u) \in J_\delta$, $\gamma_{g,\mathcal{Q}}(u) \in K_m$, and $\breve{g}_{\mathcal{Q}}(u) \in \mathscr{R}_{\delta, m}$, which are all continuous in $(g,u)$, such that
    \begin{equation*}
        \left( ( n(\varpi^{-2} u), 1 ) \ (g,1) \right)^{-1}
        = (h_{g,\mathcal{Q}}(u), \bm{s}_\delta(h_{g,\mathcal{Q}}(u))) \ (\breve{g}_{\mathcal{Q}}(u),1)^{-1} \ (\gamma_{g,\mathcal{Q}}(u), \bm{s}(\gamma_{g,\mathcal{Q}}(u)) \chi_\varpi^\delta(\gamma_{g,\mathcal{Q}}(u))).
    \end{equation*}

    Therefore, we have $\mathcal{G}_m \circ \mathcal{Q}_\varpi = \mathcal{Q}_{\varpi, m} \circ \mathcal{G}_m$ if we define a linear endomorphism
    \begin{alignat*}{3}
        \mathcal{Q}_{\varpi, m} \colon & \bigoplus_{g \in \widetilde{\mathscr{R}}_{\delta, m}} \left( \prescript{g}{}{\widetilde{\lambda}} \right)^{K_m \cap \prescript{g}{}{(\widetilde{J}_\delta)}}_\eta &  & {}\longrightarrow{} &  & \bigoplus_{g \in \widetilde{\mathscr{R}}_{\delta, m}} \left( \prescript{g}{}{\widetilde{\lambda}} \right)^{K_m \cap \prescript{g}{}{(\widetilde{J}_\delta)}}_\eta, \\
                                       & \bm{\upsilon}=(\upsilon_g)_{g \in \widetilde{\mathscr{R}}_{\delta, m}}                                                                                            &  & {}\mapsto{}         &  & \left( \mathcal{Q}_g \bm{\upsilon} \right)_{g \in \widetilde{\mathscr{R}}_{\delta, m}},
    \end{alignat*}
    by
    \begin{equation*}
        \mathcal{Q}_g \bm{\upsilon} = \int_{\mathcal{O}^\times} ( u, \varpi )_{F,2} \eta'(\gamma_{g,\mathcal{Q}}(u)) \lambda(h_{g,\mathcal{Q}}(u))\upsilon_{\breve{g}_\mathcal{Q}(u)} du.
    \end{equation*}
    We also have $\mathcal{G}'_m \circ \mathcal{Q}_\varpi' = \mathcal{Q}_{\varpi, m} \circ \mathcal{G}'_m$, and the assertion follows.
\end{proof}

\subsection{Test vectors for $\SL_2(F)$}\label{subsec:sc-newvec_sl2}
In this subsection, we fix some test vectors in irreducible supercuspidal representations of $G$.
These vectors will be shown to correspond to local newforms in the corresponding supercuspidal representations of $\widetilde{G}$ in the sense of last subsection.
Moreover, it will be shown that they are themselves the local newforms for $G$, after the proof of the main theorem.

We shall start by reviewing the Kirillov model for $\GL_2(F)$.
For an irreducible infinite dimensional representation $\varPi$ of $\GL_2(F)$, let $\omega_\varPi$ denote its central character and $K(\varPi)$ its Kirillov model with respect to the fixed additive character $\psi$.
Recall that $\varPi$ is supercuspidal if and only if $K(\varPi)$ is the space of locally constant compactly supported $\C$-valued functions on $F^\times$.
Let $\varPi$ be an irreducible supercuspidal representation of $\GL_2(F)$.
The action of the upper triangular matrices is given by
\begin{equation}\label{eq:Kirillov_action_Borel}
    \left[ \varPi\left( \left( \begin{array}{cc}a&b\\ 0&d \end{array} \right) \right) f \right] (x) = \omega_\varPi(d) \psi(d^{-1} b x) f(d^{-1} a x),
\end{equation}
for any $a,d,x \in F^\times$, $b \in F$, and $f \in K(\varPi)$.
It follows from this that the evaluation at $a \in F^\times$ is a nonzero $\psi_a$-Whittaker functional.
We shall write $\operatorname{ev}_{a}$ for the evaluation map $f \mapsto f(a)$ at $a \in F^\times$.
For any character $\mu$ of $\mathcal{O}^\times$ with $\mu(-1)=\omega_\varPi(-1)$, we define $f_\mu \in K(\varPi)$ by
\begin{equation*}
    f_\mu (x) = \begin{dcases*}
        \chi(x), & if $x\in \mathcal{O}^{\times 2}$,     \\
        0,       & if $x \notin \mathcal{O}^{\times 2}$,
    \end{dcases*}
\end{equation*}
where $\chi$ is a character of $\mathcal{O}^{\times 2}$ such that $\chi(a^2) = \mu(a^{-1}) \omega_\varPi(a)$ for any $a \in \mathcal{O}^\times$.

Put
\begin{equation*}
    \alpha = \left( \begin{array}{cc}1&\\ &\xi \end{array} \right).
\end{equation*}
\begin{lemma}\label{lem:testvector_unram2_sl2}
    Let $\lambda_0$ be an irreducible strongly cuspidal representation of $J_0$ such that $\{\tau_0, \tau_1\}$ is an unramified supercuspidal $L$-packet for $G$ of cardinality two, where we set $\lambda_1 = \prescript{\beta}{}{(\lambda_0)}$ and $\tau_\delta = \cInd^G_{J_\delta}(\lambda_\delta)$ for $\delta=0,1$.
    Put $l = c(\lambda_0)$, which is equal to $c(\lambda_1)$.
    Let $\rho$ be an irreducible very cuspidal representation of $ZH$ whose restriction to $J_0$ is $\lambda_0$.
    Put $\varPi = \cInd^{\GL_2(F)}_{ZH}(\rho)$ and we realize it as $K(\varPi)$.
    Let $\mu$ be a character of $\mathcal{O}^\times$ such that $\mu(-1)=\omega_\varPi(-1)$, and put $M=2 \max(l,c(\mu))$.
    Then we have:
    \begin{enumerate}
        \item $f_\mu, \varPi(\alpha) f_\mu \in (\tau_{\delta_l})^{K_M}_\mu \setminus (\tau_{\delta_l})^{K_{M-1}}_\mu$, where $\delta_l$ denotes the one of $0$ or $1$ with the same parity as $l$;
        \item $\varPi(\beta)f_\mu, \varPi(\alpha\beta)f_\mu \in (\tau_{1-\delta_l})^{K_{M+1}}_\mu \setminus (\tau_{1-\delta_l})^{K_M}_\mu$;
        \item $\mathcal{R}_\varpi' f_\mu = g(\chi_\varpi, \psi_{\varpi^{-1}}) f_\mu$, $\mathcal{R}_\varpi' \varPi(\alpha) f_\mu = -g(\chi_\varpi, \psi_{\varpi^{-1}}) \varPi(\alpha) f_\mu$, and $\mathcal{R}_\varpi' \varPi(\beta) f_\mu = \mathcal{R}_\varpi' \varPi(\alpha\beta) f_\mu = 0$;
        \item $\mathcal{Q}_\varpi' \varPi(\beta) f_\mu = g(\chi_\varpi, \psi_{\varpi^{-1}}) \varPi(\beta) f_\mu$ and $\mathcal{Q}_\varpi' \varPi(\alpha\beta) f_\mu = -g(\chi_\varpi, \psi_{\varpi^{-1}}) \varPi(\alpha\beta) f_\mu$.
    \end{enumerate}
\end{lemma}
\begin{proof}
    We know from Propositions \ref{prop:Kut-cond}, \ref{prop:compatib-indices_Man-KS} and the arguments around them that $l(\rho) = l$ and $\varPi$ is a minimal irreducible supercuspidal representation of $\GL_2(F)$ of conductor $2l$ whose restriction to $G$ is a direct sum of $\tau_0$ and $\tau_1$.
    Let
    \begin{equation}\label{eq:decomp-Kirillov}
        K(\varPi) = K(\tau_0, \varPi) \oplus K(\tau_1, \varPi)
    \end{equation}
    be the decomposition parallel to
    \begin{equation*}
        \varPi|_G = \tau_0 \oplus \tau_1.
    \end{equation*}
    Recall that $\tau_{\delta}$ is not $\psi_a$-generic if $\operatorname{ord}_F(a) + l + \delta$ is odd.
    Since the evaluation map $\operatorname{ev}_{a}$ gives a $\psi_a$-Whittaker functional, this implies that $\operatorname{ev}_a$ vanishes on $K(\tau_\delta, \varPi)$ if $\operatorname{ord}_F(a) + l + \delta$ is odd.
    In view of the direct decomposition \eqref{eq:decomp-Kirillov}, we can see that $K(\tau_\delta, \varPi)$ is the space of locally constant compactly supported $\C$-valued functions on $\varpi^{l+\delta+2\Z}\mathcal{O}^\times$.
    In particular, both $f_\mu$ and $\varPi(\alpha) f_\mu$ belong to $K(\tau_{\delta_l}, \varPi)$, while both $\varPi(\beta) f_\mu$ and $\varPi(\alpha\beta) f_\mu$ belong to $K(\tau_{1-\delta_l}, \varPi)$.

    We are now ready to start proving (a).
    The assertion (a) for $\varPi(\alpha)f_\mu$ follows from that for $f_\mu$, since $\prescript{\alpha}{}{K_m}=K_m$ for any $m$.
    Thus we shall show only for $f_\mu$.
    The idea of the proof is similar to that of ``the absolutely cuspidal case'' of \cite[Theorem 1]{cas}.
    It follows from \eqref{eq:Kirillov_action_Borel} that for any $a \in \mathcal{O}^\times$ and $b \in \mathcal{O}$, we have
    \begin{equation*}
        \tau_{\delta_l}\left( t(a)n(b) \right) f_\mu = \mu(a^{-1}) f_\mu.
    \end{equation*}
    It now remains to show that $\tau_{\delta_l}(n^{\mathrm{op}}(c)) f_\mu = f_\mu$ for any $c \in \mathfrak{p}^k$ if and only if $k \geq M$.
    Fix a positive integer $k$ and put
    \begin{equation*}
        f_\mu' = \varPi\left( \left( \begin{array}{cc}\varpi^k&\\ &1 \end{array} \right) \right) f_\mu,
    \end{equation*}
    i.e., $f_\mu'(x) = f_\mu(\varpi^k x)$.
    Since
    \begin{equation*}
        \left( \begin{array}{cc}&1\\ -1& \end{array} \right) \left( \begin{array}{cc}\varpi^k&\\ &1 \end{array} \right) \left( \begin{array}{cc}1&\\ c&1 \end{array} \right)
        = \left( \begin{array}{cc}1&-\varpi^{-k}c\\ &1 \end{array} \right) \left( \begin{array}{cc}&1\\ -1& \end{array} \right) \left( \begin{array}{cc}\varpi^k&\\ &1 \end{array} \right),
    \end{equation*}
    we see that $\tau_{\delta_l}(n^{\mathrm{op}}(c)) f_\mu = f_\mu$ for any $c \in \mathfrak{p}^k$ if and only if $\varPi(n(b)) \varPi(w)f_\mu' = \varPi(w)f_\mu'$ for any $b \in \mathcal{O}$.
    By the formula \eqref{eq:Kirillov_action_Borel}, this is equivalent to $\operatorname{supp}(\varPi(w)f_\mu') \subset \mathcal{O}$.

    The action of $w$ on the Kirillov model was studied by Jacquet--Langlands \cite{jl} and is characterized in the following way.
    See Proposition 2.10 and Theorem 2.18 in loc. cit. and the proofs of them for details.
    For any $f \in K(\varPi)$ and any character $\nu$ of $\mathcal{O}^\times$, define
    \begin{equation*}
        \widehat{f}_n(\nu) = \operatorname{vol}(\mathcal{O}^\times)^{-1} \int_{u \in \mathcal{O}^\times} f(\varpi^n u) \nu(u) du,
        \qquad n \in \Z.
    \end{equation*}
    Define further a formal power series
    \begin{equation*}
        \widehat{f}(\nu, t) = \sum_{n \in \Z} \widehat{f}_n(\nu) t^n,
    \end{equation*}
    which is actually a Laurent polynomial in $t$, since $f$ is compactly supported.
    For each $\nu$, there are a complex number $C_0(\nu)=C_0(\nu, \varPi)$ and an integer $n_\nu=n_\nu(\varPi)$ such that
    \begin{equation*}
        \widehat{\varPi(w)f} (\nu, t) = C_0(\nu) t^{n_\nu} \widehat{f}(\nu^{-1}\omega_\varPi^{-1}, \omega_\varPi(\varpi)^{-1} t^{-1}),
    \end{equation*}
    for any $f \in K(\varPi)$.
    From the equation (2.18.1) of \cite{jl} and Remark in pp.306-307 of \cite{cas}, we have
    \begin{equation*}
        n_\nu(\varPi) = -c(\varPi \otimes (\nu_F^{-1} \circ \det)),
    \end{equation*}
    where $\nu_F$ is any character of $F^\times$ such that $\nu_F|_{\mathcal{O}^\times} = \nu$.
    Note that this theory holds for not only our $\varPi$ but also any irreducible supercuspidal representation of $\GL_2(F)$.
    We now return to our proof of Lemma \ref{lem:testvector_unram2_sl2} (a).
    The condition $\operatorname{supp}(\varPi(w)f_\mu') \subset \mathcal{O}$ is satisfied if and only if $\widehat{\varPi(w)f_\mu'} (\nu, t)$ has no negative degree terms for any $\nu$.
    For any character $\nu$, a direct calculation shows that
    \begin{align*}
        \widehat{f_\mu'} (\nu, t)
        = \begin{dcases*}
              \frac{1}{2} t^{-k}, & if $\nu^2 = \mu \omega_\varPi^{-1}$ on $\mathcal{O}^\times$, \\
              0,                  & otherwise.
          \end{dcases*}
    \end{align*}
    Hence we get
    \begin{align*}
        \widehat{\varPi(w)f_\mu'} (\nu, t)
        =\begin{dcases*}
             C_1(\nu) t^{n_\nu + k}, & if $\nu^2 = \mu^{-1} \omega_\varPi^{-1}$ on $\mathcal{O}^\times$, \\
             0,                      & otherwise,
         \end{dcases*}
    \end{align*}
    where $C_1(\nu) = C_1(\nu, \varPi)$ is a nonzero constant.
    Therefore, $\widehat{\varPi(w)f_\mu'} (\nu, t)$ has no negative degree terms for any $\nu$ if and only if
    \begin{equation*}
        k \geq c(\varPi \otimes (\nu_F^{-1} \circ \det))
    \end{equation*}
    for any character $\nu$ such that $\nu^2 = \mu^{-1} \omega_\varPi^{-1}$ on $\mathcal{O}^\times$.
    Since the residual character is odd, the conductor of such $\nu$ equals that of $\mu \omega_\varPi$.
    Noting that $\varPi$ is minimal supercuspidal, by \cite[Theorem 3.4]{tun}, we have
    \begin{align*}
        c(\varPi \otimes (\nu_F^{-1} \circ \det))
         & =\max(c(\varPi), 2c(\mu \omega_\varPi)) \\
         & =\max(c(\varPi), 2c(\mu))               \\
         & =\max(2l, 2c(\mu)).
    \end{align*}
    This completes the proof of (a).
    The assertion (b) follows from (a), since $\beta n^{\mathrm{op}}(c) \beta^{-1} = n^{\mathrm{op}}(\varpi c)$.

    We now come to the proof of (c).
    By \eqref{eq:Kirillov_action_Borel}, we have
    \begin{align*}
        [\mathcal{R}_\varpi' f_\mu] (x)
         & =\int_{\mathcal{O}^\times} ( u, \varpi )_{F,2} \psi(\varpi^{-1} u x) f_\mu(x) du     \\
         & =\int_{\mathcal{O}^\times} ( ux^{-1}, \varpi )_{F,2} \psi(\varpi^{-1} u) f_\mu(x) du \\
         & = g(\chi_\varpi, \psi_{\varpi^{-1}}) f_\mu(x).
    \end{align*}
    Here the last equality holds since $f_\mu$ is supported on $\mathcal{O}^{\times 2}$.
    Similarly, we have
    \begin{align*}
        \mathcal{R}_\varpi' \varPi(\alpha) f_\mu      & = -g(\chi_\varpi, \psi_{\varpi^{-1}}) \varPi(\alpha) f_\mu, \\
        \mathcal{R}_\varpi' \varPi(\beta) f_\mu       & = g(\chi_\varpi, \psi_{\varpi^{-2}}) \varPi(\alpha) f_\mu,  \\
        \mathcal{R}_\varpi' \varPi(\alpha\beta) f_\mu & = -g(\chi_\varpi, \psi_{\varpi^{-2}}) \varPi(\alpha) f_\mu.
    \end{align*}
    The assertion (c) now follows from Lemma \ref{lem:gausssum1}.
    The proof of (d) is similar to that of (c).
\end{proof}

\begin{lemma}\label{lem:testvector_unram4_sl2}
    Let $\rho$ be an irreducible very cuspidal representation of $ZH$ whose restriction to $J_0$ is a direct sum of two irreducible strongly cuspidal representations.
    Let $\lambda_{0,1}$ be the one of them such that $\cInd^G_{J_0}(\lambda_{0,1})$ is $\psi_\varpi$-generic, and $\lambda_{0,2}$ the other one.
    Put $\tau_{0,i} = \cInd^G_{J_0}(\lambda_{0,i})$ and $\tau_{1,i} = \cInd^G_{J_1}(\prescript{\beta}{}{\lambda_{0,i}})$ for $i=1,2$.
    Put $\varPi = \cInd^{\GL_2(F)}_{ZH}(\rho)$ and we realize it as $K(\varPi)$.
    Let $\mu$ be a character of $\mathcal{O}^\times$ such that $\mu(-1)=\omega_\varPi(-1)$, and put $M=2 \max(1,c(\mu))$.
    Then we have:
    \begin{enumerate}
        \item $f_\mu \in (\tau_{1,1})^{K_M}_\mu \setminus (\tau_{1,1})^{K_{M-1}}_\mu$, $\varPi(\alpha) f_\mu \in (\tau_{1,2})^{K_M}_\mu \setminus (\tau_{1,2})^{K_{M-1}}_\mu$, $\varPi(\beta) f_\mu \in (\tau_{0,1})^{K_{M+1}}_\mu \setminus (\tau_{0,1})^{K_M}_\mu$, and $\varPi(\alpha\beta) f_\mu \in (\tau_{0,2})^{K_{M+1}}_\mu \setminus (\tau_{0,2})^{K_M}_\mu$;
        \item $\mathcal{R}_\varpi' f_\mu = g(\chi_\varpi, \psi_{\varpi^{-1}}) f_\mu$, $\mathcal{R}_\varpi' \varPi(\alpha) f_\mu = -g(\chi_\varpi, \psi_{\varpi^{-1}}) \varPi(\alpha) f_\mu$, and $\mathcal{R}_\varpi' \varPi(\beta) f_\mu = \mathcal{R}_\varpi' \varPi(\alpha\beta) f_\mu = 0$;
        \item $\mathcal{Q}_\varpi' \varPi(\beta) f_\mu = g(\chi_\varpi, \psi_{\varpi^{-1}}) \varPi(\beta) f_\mu$ and $\mathcal{Q}_\varpi' \varPi(\alpha\beta) f_\mu = -g(\chi_\varpi, \psi_{\varpi^{-1}}) \varPi(\alpha\beta) f_\mu$.
    \end{enumerate}
\end{lemma}
\begin{proof}
    Notice that $\{\tau_{0,1}, \tau_{0,2}, \tau_{1,1}, \tau_{1,2}\}$ is an unramified supercuspidal $L$-packet for $G$ of cardinality four, and that $c(\lambda_{\delta, i})=l(\rho)=1$.
    Let
    \begin{equation*}
        K(\varPi) = K(\tau_{0,1}, \varPi) \oplus K(\tau_{0,2}, \varPi) \oplus K(\tau_{1,1}, \varPi) \oplus K(\tau_{1,2}, \varPi)
    \end{equation*}
    be the decomposition parallel to
    \begin{equation*}
        \varPi|_G = \tau_{0,1} \oplus \tau_{0,2} \oplus \tau_{1,1} \oplus \tau_{1,2}.
    \end{equation*}
    Then $K(\tau_{\delta, i}, \varPi)$ is the space of locally constant compactly supported $\C$-valued functions on $\xi^{1+i+2\Z} \varpi^{1+\delta+2\Z} \mathcal{O}^\times$, since $\tau_{\delta, i}$ is $\psi_{\xi^j \varpi^\epsilon}$-generic if and only if both $j+i$ and $\epsilon+\delta$ are odd.
    In particular, $f_\mu$ (resp. $\varPi(\alpha) f_\mu$, resp. $\varPi(\beta)f_\mu$, resp. $\varPi(\alpha\beta)f_\mu$) belongs to $K(\tau_{1,1}, \varPi)$ (resp. $K(\tau_{1,2}, \varPi)$, resp. $K(\tau_{0,1}, \varPi)$, resp. $K(\tau_{0,2}, \varPi)$).
    Then the rest of the proof is same as that of Lemma \ref{lem:testvector_unram2_sl2}.
\end{proof}

\begin{lemma}\label{lem:testvector_ram_sl2}
    Let $\rho$ be an irreducible very cuspidal representation of $Z'I$ of level $l$, and let $\lambda'_1$ and $\lambda'_2$ be the two irreducible components of $\Res^{Z'I}_{I\cap G}(\rho)$.
    Put $\lambda_i = \cInd^{J_1}_{I\cap G}(\lambda'_i)$ and $\tau_i=\cInd^G_{J_1}(\lambda_i)$, for $i=1,2$.
    Suppose that $\tau_1$ is $\psi$-generic.
    Put $\varPi = \cInd^{\GL_2(F)}_{Z'I}(\rho)$ and we realize it as $K(\varPi)$.
    Let $\mu$ be a character of $\mathcal{O}^\times$ such that $\mu(-1)=\omega_\varPi(-1)$, and put $M=\max(2l+1, 2c(\mu))$.
    Then we have:
    \begin{enumerate}
        \item $f_\mu \in (\tau_1)^{K_M}_\mu \setminus (\tau_1)^{K_{M-1}}_\mu$ and $\varPi(\alpha)f_\mu \in (\tau_2)^{K_M}_\mu \setminus (\tau_2)^{K_{M-1}}_\mu$;
        \item $\varPi(\beta) f_\mu \in (\tau_1)^{K_{M+1}}_\mu \setminus (\tau_1)^{K_M}_\mu$ and $\varPi(\alpha\beta) f_\mu \in (\tau_2)^{K_{M+1}}_\mu \setminus (\tau_2)^{K_M}_\mu$;
        \item $\mathcal{R}_\varpi' f_\mu = g(\chi_\varpi, \psi_{\varpi^{-1}}) f_\mu$, $\mathcal{R}_\varpi' \varPi(\alpha) f_\mu = -g(\chi_\varpi, \psi_{\varpi^{-1}}) \varPi(\alpha) f_\mu$, and $\mathcal{R}_\varpi' \varPi(\beta) f_\mu = \mathcal{R}_\varpi' \varPi(\alpha\beta) f_\mu = 0$;
        \item $\mathcal{Q}_\varpi' \varPi(\beta) f_\mu = g(\chi_\varpi, \psi_{\varpi^{-1}}) \varPi(\beta) f_\mu$ and $\mathcal{Q}_\varpi' \varPi(\alpha\beta) f_\mu = -g(\chi_\varpi, \psi_{\varpi^{-1}}) \varPi(\alpha\beta) f_\mu$.
    \end{enumerate}
\end{lemma}
\begin{proof}
    We know from Lemma \ref{lem:compatib-cpt_Man-KS} and Proposition \ref{prop:compatib-indices_Man-KS} that $\lambda_i$ is an irreducible strongly cuspidal representation of $J_1$ of conductor $l+1$ and defect $1$.
    Moreover, we know that $\varPi$ is a minimal supercuspidal representation of $\GL_2(F)$ of conductor $2l+1$ whose restriction to $G$ is a direct sum of $\tau_1$ and $\tau_2$.

    Let
    \begin{equation*}
        K(\varPi) = K(\tau_1, \varPi) \oplus K(\tau_2, \varPi)
    \end{equation*}
    be the decomposition parallel to
    \begin{equation*}
        \varPi|_G = \tau_1 \oplus \tau_2.
    \end{equation*}
    Then $K(\tau_i, \varPi)$ is the space of locally constant compactly supported $\C$-valued functions on $\xi^{i-1} \varpi^\Z \mathcal{O}^{\times 2}$.
    In particular, both $f_\mu$ and $\varPi(\beta) f_\mu$ belong to $K(\tau_1, \varPi)$, while both $\varPi(\alpha) f_\mu$ and $\varPi(\alpha\beta) f_\mu$ belong to $K(\tau_2, \varPi)$.

    Since $\prescript{\alpha}{}{\tau_1}=\tau_2$, the assertion (a) for $\varPi(\alpha) f_\mu$ follows from that for $f_\mu$.
    Since $\prescript{\beta}{}{\tau_i}=\tau_i$, the assertion (b) follows from the assertion (a).
    The proofs of (a) for $f_\mu$, (c), and (d) are same as those for Lemma \ref{lem:testvector_unram2_sl2} except that $c(\varPi)=2l+1$ now.
\end{proof}

\subsection{The final step}\label{subsec:sc-final}
Now, we start the proof of Theorems \ref{thm:conductor}, \ref{thm:main-nongeneric}, and \ref{thm:main-generic} for supercuspidal representations.
Let $\pi$ be an irreducible genuine supercuspidal representation of $\widetilde{G}$ and $\eta$ a character of $\mathcal{O}^\times$ with $\eta(-1)=z_\psi(\pi)$.
Take a compact subgroup $J=J_0$ or $J_1$ of $G$ and an irreducible strongly cuspidal representation $\lambda$ of $J$ such that $\pi=\cInd^{\widetilde{G}}_{\widetilde{J}}(\widetilde{\lambda})$, and let $\tau=\cInd^G_J(\lambda)$.
Let $\varPi$ be an irreducible minimal supercuspidal representation of $\GL_2(F)$ as in Lemma \ref{lem:testvector_unram2_sl2}, \ref{lem:testvector_unram4_sl2}, or \ref{lem:testvector_ram_sl2}.
In addition, let $\eta'$ be the corresponding character in the sense of the subsection \ref{subsec:preliminary (tentative)}.

The proof is based on the comparison of $\pi^{K_m, \RSnew}_\eta$ and $\pi^{K_m, \KUnew}_\eta$.
However, if $\pi$ is not $\psi_a$-generic for any $a \in \mathcal{O}^\times$, the latter will be zero and the comparison does not work well.
To address the problem, we introduce another subspace
\begin{equation*}
    \pi^{K_m, \Qold}_\eta = \pi^{K_{m-1}}_\eta + \operatorname{Ker}(\mathcal{Q}|_{\pi^{K_m}_\eta}),
\end{equation*}
and define $\pi^{K_m, \Qnew}_\eta$ to be the quotient of $\pi^{K_m}_\eta$ by $\pi^{K_m, \Qold}_\eta$.
It follows from \eqref{eq:prop_Q-operator} that
\begin{equation*}
    \pi^{K_m, \Qold}_\eta \supset \pi^{K_m, \RSold}_\eta,
\end{equation*}
and hence
\begin{equation*}
    \dim_\C \pi^{K_m, \Qnew}_\eta \leq \dim_\C \pi^{K_M, \RSnew}_\eta.
\end{equation*}

\begin{lemma}\label{lem:sc-generic}
    For any nontrivial additive character $\Psi$ of $F$, $\pi$ is $\Psi$-generic if and only if $\tau$ is $\Psi$-generic.
\end{lemma}
\begin{proof}
    Using Frobenius reciprocity and Mackey theory (\cite[Corollary 2.7 (1)]{yam}), we have
    \begin{align*}
        \Hom_N(\pi, \Psi)
         & =\Hom_{\widetilde{G}}(\cInd^{\widetilde{G}}_{\widetilde{J}} (\widetilde{\lambda}), \Ind^{\widetilde{G}}_N (\Psi))                                        \\
         & \cong \prod_{g\in \widetilde{J} \backslash \widetilde{G} / N} \Hom_{\widetilde{J} \cap \prescript{g}{}{N}} (\widetilde{\lambda}, \prescript{g}{}{\Psi}).
    \end{align*}
    Similarly, we have
    \begin{equation*}
        \Hom_N(\tau, \Psi)
        \cong \prod_{g\in J \backslash G / N} \Hom_{J \cap \prescript{g}{}{N}} (\lambda, \prescript{g}{}{\Psi}).
    \end{equation*}
    By the Iwasawa decomposition and the Levi decomposition, we know that $G=JTN$, and thus we may choose a representative system $R(J \backslash G / N)$ consists of elements in $T$.
    Then we have $\prescript{g}{}{N}=N$ for every $g \in R(J \backslash G / N)$.
    Of course $\widetilde{J} \backslash \widetilde{G} / N$ can be identified with $J \backslash G / N$, and we now have
    \begin{equation*}
        \Hom_N(\tau, \Psi)
        \cong \prod_{g\in R(J \backslash G / N)} \Hom_{J \cap N} (\lambda, \prescript{g}{}{\Psi})
    \end{equation*}
    and
    \begin{equation*}
        \Hom_N(\pi, \Psi)
        \cong \prod_{g\in R(J \backslash G / N)} \Hom_{\widetilde{J} \cap N} (\widetilde{\lambda}, \prescript{g}{}{\Psi}).
    \end{equation*}
    Since $\bm{s}_\delta$ is trivial on $J \cap N$, we have
    \begin{equation*}
        \Hom_{\widetilde{J} \cap N} (\widetilde{\lambda}, \prescript{g}{}{\Psi})
        = \Hom_{J \cap N} (\lambda, \prescript{g}{}{\Psi}).
    \end{equation*}
    Consequently, we have
    \begin{align*}
        \Hom_N(\pi, \Psi) \cong \Hom_N(\tau, \Psi).
    \end{align*}
    This completes the proof.
\end{proof}

\begin{proposition}\label{prop:sc-unram2}
    Assume that $\tau$ is an element of unramified supercuspidal $L$-packet of cardinality two.
    Then Theorems \ref{thm:conductor}, \ref{thm:main-nongeneric}, and \ref{thm:main-generic} hold for $\pi$ and $\eta$.
\end{proposition}
\begin{proof}
    Suppose first that $\tau$ is $\psi$-generic.
    Then by Lemmas \ref{lem:testvector_unram2_sl2} and \ref{lem:comm_R-operator_sc} we have
    \begin{equation*}
        \mathcal{R}_\varpi v_\pm = \pm g(\chi_\varpi, \psi_{\varpi^{-1}}) v_\pm \in \pi^{K_M}_\eta \setminus \pi^{K_{M-1}}_\eta,
    \end{equation*}
    where $v_+=\mathscr{G}_M^{-1}(f_{\eta'})$, $v_-= \mathscr{G}_M^{-1}(\varPi(\alpha)f_{\eta'})$, and $M=2 \max(c(\lambda), c(\eta'))$.
    We note that $M$ is equal to $2 \max(c(\lambda), c(\eta))$, since $c(\lambda) \geq 1$.
    In particular we have $M \geq 2$.
    By Propositions \ref{prop:U-operator}, \ref{prop:R-operator}, and \ref{prop:sc_<2}, we have
    \begin{equation*}
        \pi^{K_M, \KUold}_\eta = \pi^{K_{M-1}}_\eta + \operatorname{Ker}(\mathcal{R}_\varpi|_{\pi^{K_M}_\eta}).
    \end{equation*}
    We are now going to show that any nonzero element of $\C v_+ \oplus \C v_-$ does not lie in $\pi^{K_M, \KUold}_\eta$.
    Suppose that $v_0$ is an element in the intersection of $\C v_+ \oplus \C v_-$ and $\pi^{K_M, \KUold}_\eta$.
    Then both of $\mathcal{R}_\varpi v_0$ and $\mathcal{R}_\varpi^2 v_0$ lie in the intersection of $\C v_+ \oplus \C v_-$ and $\pi^{K_{M-1}}_\eta$.
    Since $g(\chi_\varpi, \psi_{\varpi^{-1}}) \neq 0$, this means that $\pi^{K_{M-1}}_\eta$ contains at least one of $v_+$ or $v_-$ unless $v_0 = 0$.
    Thus $v_0$ must be zero.
    Consequently, $v_+$ and $v_-$ are linearly independent in $\pi^{K_M, \KUnew}_\eta$.
    In particular, we have
    \begin{equation*}
        \dim_\C \pi^{K_M, \KUnew}_\eta \geq 2.
    \end{equation*}
    Recalling that
    \begin{equation*}
        \dim_\C \pi^{K_M, \KUnew}_\eta \leq \dim_\C \pi^{K_M, \RSnew}_\eta \leq \sum_{m \geq 0} \dim_\C \pi^{K_m, \RSnew}_\eta = 2,
    \end{equation*}
    we get
    \begin{align*}
        \dim_\C \pi^{K_M, \KUnew}_\eta & =2,              \\
        \dim_\C \pi^{K_m, \RSnew}_\eta & =\begin{cases*}
                                              2, & if $m=M$,  \\
                                              0, & otherwise,
                                          \end{cases*}
    \end{align*}
    and hence
    \begin{align*}
        \dim_\C \pi^{K_m, \KUnew}_\eta & =\begin{cases*}
                                              2, & if $m=M$,  \\
                                              0, & otherwise.
                                          \end{cases*}
    \end{align*}
    Since $\pi^{K_{c_\eta(\pi)}, \RSnew}_\eta$ is necessarily nonzero and equal to $\pi^{K_{c_\eta(\pi)}}_\eta$, we have
    \begin{align*}
        c_\eta(\pi)            & = M = 2 \max(c(\lambda), c(\eta)),       \\
        \pi^{K_M, \KUnew}_\eta & = \pi^{K_M}_\eta = \C v_+ \oplus \C v_-.
    \end{align*}
    Note that in this case $\pi$ is $\psi_a$-generic for every $a \in \mathcal{O}^\times$, as implied by the result of Kutzko--Sally and Lemma \ref{lem:sc-generic}.

    Suppose next that $\tau$ is not $\psi$-generic.
    Then by Lemmas \ref{lem:testvector_unram2_sl2} and \ref{lem:comm_Q-operator_sc} we have
    \begin{equation*}
        \mathcal{Q}_\varpi v_\pm = \pm g(\chi_\varpi, \psi_{\varpi^{-1}}) v_\pm \in \pi^{K_{M+1}}_\eta \setminus \pi^{K_M}_\eta,
    \end{equation*}
    where $v_+=\mathscr{G}_{M+1}^{-1}(\varPi(\beta)f_{\eta'})$, $v_-= \mathscr{G}_{M+1}^{-1}(\varPi(\alpha\beta)f_{\eta'})$, and $M=2 \max(c(\lambda), c(\eta'))$.
    By arguing as above and appealing to $\mathcal{Q}_\varpi$, $\pi^{K_m, \Qold}_\eta$, and $\pi^{K_m, \Qnew}_\eta$ instead of $\mathcal{R}_\varpi$, $\pi^{K_m, \KUold}_\eta$, and $\pi^{K_m, \KUnew}_\eta$ respectively, we deduce that $v_+$ and $v_-$ are linearly independent in $\pi^{K_m, \Qnew}_\eta$ and obtain
    \begin{align*}
        c_\eta(\pi)               & = M+1 = 2 \max(c(\lambda), c(\eta)) + 1,                                  \\
        \pi^{K_{M+1}, \Qnew}_\eta & = \pi^{K_{M+1}, \RSnew}_\eta = \pi^{K_{M+1}}_\eta = \C v_+ \oplus \C v_-,
    \end{align*}
    and
    \begin{equation}\label{eq:sc-unram4-nongen}
        \pi^{K_m, \RSnew}_\eta    = \{0\},
    \end{equation}
    for $m \neq M+1$.
    Lemma \ref{lem:comm_R-operator_sc} and (c) of Lemma \ref{lem:testvector_unram2_sl2} imply that $\mathcal{R}_\varpi$ vanishes both of $v_\pm$.
    In particular, $\pi^{K_{M+1}}_\eta$, which is spanned by $v_+$ and $v_-$, is contained in $\pi^{K_{M+1}, \KUold}_\eta$.
    Combined this with \eqref{eq:sc-unram4-nongen} above, we conclude that
    \begin{equation*}
        \pi^{K_m, \KUnew}_\eta = \{0\},
    \end{equation*}
    for every integer $m$.
    Note that in this case $\pi$ is not $\psi_a$-generic for any $a \in \mathcal{O}^\times$, as implied by the result of Kutzko--Sally and Lemma \ref{lem:sc-generic}.
    This completes the proof.
\end{proof}

\begin{proposition}\label{prop:sc-unram4}
    Assume that $\tau$ is an element of an unramified supercuspidal $L$-packet of cardinality four.
    Then Theorems \ref{thm:conductor}, \ref{thm:main-nongeneric}, and \ref{thm:main-generic} hold for $\pi$ and $\eta$.
\end{proposition}
\begin{proof}
    We will only sketch the proof since its method is essentially same as that of Proposition \ref{prop:sc-unram2} above.

    Suppose first that $\tau$ is $\psi_a$-generic for some $a \in \mathcal{O}^\times$, and let $f_\tau \in \tau^{K_M}_{\eta'} \setminus \tau^{K_{M-1}}_{\eta'}$ be either $f_{\eta'}$ or $\varPi(\alpha)f_{\eta'}$, where $M=2 \max(1, c(\eta'))$
    Recall from Lemma \ref{lem:testvector_unram4_sl2} that exactly one of them is an element in $\tau^{K_M}_{\eta'} \setminus \tau^{K_{M-1}}_{\eta'}$.
    By Lemma \ref{lem:comm_R-operator_sc} and (b) of Lemma \ref{lem:testvector_unram4_sl2}, we have
    \begin{equation*}
        \mathcal{R}_\varpi v_0 = \pm g(\chi_\varpi, \psi_{\varpi^{-1}}) v_0 \in \pi^{K_M}_\eta \setminus \pi^{K_{M-1}}_\eta,
    \end{equation*}
    where $v_0=\mathscr{G}_M^{-1}(f_\tau)$.
    This implies that $v_0 \notin \pi^{K_M, \KUold}_\eta$ and
    \begin{equation*}
        1 \leq \dim_\C \pi^{K_M, \KUnew}_\eta \leq \dim_\C \pi^{K_M, \RSnew}_\eta \leq \sum_{m \geq 0} \dim_\C \pi^{K_m, \RSnew}_\eta = 1.
    \end{equation*}
    Consequently, we have
    \begin{align*}
        c_\eta(\pi)                    & = M = 2 \max(1, c(\eta)),                        \\
        \dim_\C \pi^{K_m, \KUnew}_\eta & =\dim_\C \pi^{K_m, \RSnew}_\eta =\begin{cases*}
                                                                              1, & if $m=M$,  \\
                                                                              0, & otherwise.
                                                                          \end{cases*}
    \end{align*}
    Note that in this case $\pi$ is $\psi_a$-generic for some $a \in \mathcal{O}^\times$.

    Suppose next that $\tau$ is $\psi_{a\varpi}$-generic for some $a \in \mathcal{O}^\times$.
    By Lemma \ref{lem:testvector_unram4_sl2}, there exists $f_\tau \in \tau^{K_{M+1}}_{\eta'} \setminus \tau^{K_M}_{\eta'}$ and $a_\tau \in \C^\times$ such that $\mathcal{R}_\varpi' f_\tau = 0$ and $\mathcal{Q}_\varpi' f_\tau = a_\tau f_\tau$.
    Then we see from Lemma \ref{lem:comm_Q-operator_sc} that $v_1=\mathscr{G}_M^{-1}(f_\tau)$ lies in $\pi^{K_{M+1}}_\eta \setminus \pi^{K_M}_\eta$, and satisfies $\mathcal{R}_\varpi v_1 = 0$ and $\mathcal{Q}_\varpi v_1 = a_\tau v_1$.
    Arguing as above, we obtain
    \begin{align*}
        c_\eta(\pi)                   & = M+1 = 2 \max(1, c(\eta)) + 1,                   \\
        \dim_\C \pi^{K_m, \Qnew}_\eta & =\dim_\C \pi^{K_m, \RSnew}_\eta =\begin{cases*}
                                                                             1, & if $m=M+1$, \\
                                                                             0, & otherwise,
                                                                         \end{cases*}
    \end{align*}
    and
    \begin{equation*}
        \pi^{K_m, \KUnew}_\eta = \{0\},
    \end{equation*}
    for any integer $m$.
    Note that in this case $\pi$ is not $\psi_a$-generic for any $a \in \mathcal{O}^\times$.
    This completes the proof.
\end{proof}
\begin{remark}
    If $\tau$ is an element of an unramified supercuspidal $L$-packet of cardinality four, then $\pi$ is an odd Weil representation.
\end{remark}

\begin{proposition}\label{prop:sc-ram}
    Assume that $\tau$ is an element of a ramified supercuspidal $L$-packet.
    Then Theorems \ref{thm:conductor} and \ref{thm:main-generic} hold for $\pi$ and $\eta$.
\end{proposition}
\begin{proof}
    The method of proof is essentially same as that of Proposition \ref{prop:sc-unram2} above.

    By Lemma \ref{lem:testvector_ram_sl2}, there exist $f_\tau \in \tau^{K_M}_{\eta'} \setminus \tau^{K_{M-1}}_{\eta'}$ and $a_\tau \in \C^\times$ such that $\varPi(\beta)f_\tau \in \tau^{K_{M+1}}_{\eta'} \setminus \tau^{K_M}_{\eta'}$, $\mathcal{R}_\varpi' f_\tau = a_\tau f_\tau$, $\mathcal{R}_\varpi' \varPi(\beta)f_\tau = 0$, and $\mathcal{Q}_\varpi' \varPi(\beta)f_\tau= a_\tau \varPi(\beta)f_\tau$, where $M=\max(2c(\lambda)-1, 2c(\eta'))$.
    Then $v_0=\mathscr{G}_M^{-1}(f_\tau)$ and $v_1=\mathscr{G}_{M+1}^{-1}(\varPi(\beta)f_\tau)$ lie in $\pi^{K_M}_\eta \setminus \pi^{K_{M+1}}_\eta$ and $\pi^{K_{M+1}}_\eta \setminus \pi^{K_M}_\eta$, respectively.
    Moreover, we have $\mathcal{R}_\varpi v_0 = a_\tau v_0$, $\mathcal{R}_\varpi v_1 = 0$, and $\mathcal{Q}_\varpi v_1 = a_\tau v_1$, thanks to Lemmas \ref{lem:comm_R-operator_sc} and \ref{lem:comm_Q-operator_sc}.
    These imply that $v_0 \notin \pi^{K_M, \KUold}_\eta$ and $v_1 \notin \pi^{K_{M+1}, \Qold}_\eta$.
    In particular, we have
    \begin{align*}
        \dim_\C \pi^{K_M, \KUnew}_\eta    & \geq 1, \\
        \dim_\C \pi^{K_{M+1}, \Qnew}_\eta & \geq 1,
    \end{align*}
    and
    \begin{align*}
        2 & \leq \dim_\C \pi^{K_M, \KUnew}_\eta + \dim_\C \pi^{K_{M+1}, \Qnew}_\eta                                                           \\
          & \leq \dim_\C \pi^{K_M, \RSnew}_\eta + \dim_\C \pi^{K_{M+1}, \RSnew}_\eta \leq \sum_{m \geq 0} \dim_\C \pi^{K_m, \RSnew}_\eta = 2.
    \end{align*}
    Therefore, we have
    \begin{align*}
        c_\eta(\pi)                    & = M = \max(2c(\lambda)-1, 2c(\eta)), \\
        \pi^{K_M}_\eta                 & = \C v_0,                            \\
        \pi^{K_{M+1}}_\eta             & = \C v_0 \oplus \C v_1,              \\
        \dim_\C \pi^{K_m, \RSnew}_\eta & =\begin{cases*}
                                              1, & if $m=M, M+1$, \\
                                              0, & otherwise,
                                          \end{cases*}
    \end{align*}
    and
    \begin{align*}
        \dim_\C \pi^{K_m, \KUnew}_\eta  =\begin{cases*}
                                             1, & if $m=M$,  \\
                                             0, & otherwise.
                                         \end{cases*}
    \end{align*}
    Note that $\pi$ is either $\psi$-generic or $\psi_\xi$-generic.
    This completes the proof.
\end{proof}
Now, we have finished the proof of our main theorems.
\begin{corollary}\label{cor:local-newform_sc_sl2}
    For any character $\mu$ of $\mathcal{O}^\times$ such that $\tau(-1_2)$ is a scalar multiplication by $\mu(-1)$, we have
    \begin{align*}
        c_\mu(\tau) = \begin{cases*}
                          \max(c(\varPi), 2c(\mu)),   & if $\tau$ is $\psi_a$-generic for some $a \in \mathcal{O}^\times$,    \\
                          1+\max(c(\varPi), 2c(\mu)), & if $\tau$ is not $\psi_a$-generic for any $a \in \mathcal{O}^\times$.
                      \end{cases*}
    \end{align*}
    Moreover, $\tau^{K_{c_\mu(\tau)}}_\mu$ is spanned by the vectors given in (a) of Lemmas \ref{lem:testvector_unram2_sl2}, \ref{lem:testvector_unram4_sl2}, and \ref{lem:testvector_ram_sl2}.
    In particular, an appropriate Whittaker functional does not vanish on $\tau^{K_{c_\mu(\tau)}}_\mu$.
\end{corollary}
\begin{proof}
    The corollary follows from Propositions \ref{prop:sc-unram2}, \ref{prop:sc-unram4}, \ref{prop:sc-ram}, and the family of isomorphisms $\{\mathscr{G}_m\}_{m \geq 0}$.
\end{proof}

\section{Variation of compact subgroups}\label{sec:var-subgroup}
We continue to take $F$ to be a $p$-adic field with odd residual characteristic.
The group $G=\SL_2(F)$ has exactly two maximal open compact subgroups up to conjugacy.
One is $K_0=\SL_2(\mathcal{O})$.
The other is $\prescript{\beta}{}{K_0}$, for which we shall write $L_0$.
Recall from Subsection \ref{subsec:sc-constr} that $L_0=J_1$ can be regarded as a subgroup of $\widetilde{G}$ via $\gamma \mapsto (\gamma, \bm{s}_1(\gamma))$.
In addition, put $L_m=\prescript{\beta}{}{K_m}$ for every $m \geq 1$.
In this little section, we state the theory of local newforms for $\widetilde{G}$ relative to $\{L_m\}_{m \geq 0}$.

As before, for any character $\eta$ of $\mathcal{O}^\times$, we also write $\eta$ for the map
\begin{equation*}
    L_m \ni \left( \begin{array}{cc}a&b\\ c&d \end{array} \right) \mapsto \eta(d) \in \C^\times,
\end{equation*}
which is a character if $m \geq c(\eta)$.
Then we put
\begin{equation*}
    \pi^{L_m}_\eta =\begin{cases*}
        \Set{v\in V | \pi(\gamma)v=\eta(\gamma)v,\ \text{for all $\gamma\in L_m$}}, & if $m \geq c(\eta)$, \\
        \{0\},                                                                      & if $m<c(\eta)$,
    \end{cases*}
\end{equation*}
and define a linear operator $\mathcal{V}_{\varpi^2}$ by
\begin{equation*}
    \mathcal{V}_{\varpi^2} v  = ( -1, \varpi )_{F,2}\int_{\mathcal{O}} \pi\left( \left( \left( \begin{array}{cc}\varpi&\varpi^{-2} u\\ 0&\varpi^{-1} \end{array} \right), 1 \right) \right) v du.
\end{equation*}
We finally define the space of local oldforms and newforms (in the sense of Kohnen and Ueda) relative to $L_m$ by
\begin{align*}
    \pi^{L_m, \KUold}_\eta & = \pi^{L_{m-1}}_\eta + \mathcal{V}_{\varpi^2}\left( \pi^{L_{m-1}}_\eta \right) + \left( \mathcal{Q}_\varpi\left( \pi^{L_{m-1}}_\eta \right) \cap \pi^{L_m}_\eta \right) + \operatorname{Ker}(\mathcal{Q}_\varpi|_{\pi^{L_m}_\eta}), \\
    \pi^{L_m, \KUnew}_\eta & = \pi^{L_m}_\eta / \pi^{L_m, \KUold}_\eta.
\end{align*}
\begin{corollary}\label{cor:main-thm-L_m}
    Let $\psi \colon F \to \C^\times$ be a nontrivial additive character of conductor 1. 
    Let $(\pi, V)$ be an irreducible genuine representation of $\widetilde{G}$.
    Let $\eta$ be a character of $\mathcal{O}^\times$.
    Then the subspace $\pi^{L_m}_\eta$ is zero for every $m$ if and only if $\eta(-1) \neq z_\psi(\pi)$.

    Assume that $\eta(-1) = z_\psi(\pi)$.
    If $\pi$ is neither $\psi$-generic nor $\psi_\xi$-generic, then we have
    \begin{equation*}
        \pi^{L_m, \KUnew}_\eta = \{0\}
    \end{equation*}
    for all $m$.
    If $\pi$ is $\psi_a$-generic for some $a \in \mathcal{O}^\times$, then $\pi^{L_m, \KUnew}_\eta$ is zero unless $m$ is the smallest number such that $\pi^{L_m}_\eta$ is not zero, in which case it has dimension 1 or 2.
    Moreover, in the case of dimension 2, we have
    \begin{equation*}
        \pi^{K_m, \KUnew}_\eta = \pi^{K_m}_\eta,
    \end{equation*}
    for such $m$, and it has a basis $\{v_+, v_-\}$ such that
    \begin{equation*}
        \mathcal{Q}_\varpi v_\pm = \pm \gamma_F(\varpi, \psi) q^{-\frac{1}{2}} \operatorname{vol}(\mathcal{O}) v_\pm.
    \end{equation*}
\end{corollary}
\begin{proof}
    In general, any element $h$ in $\GL_2(F)$ gives an automorphism $x \mapsto \prescript{h}{}{x} = hxh^{-1}$ of $G$, and this lifts to $\widetilde{G}$.
    We shall write $g \mapsto \prescript{h}{}{g}$ for this action on $\widetilde{G}$, which can be regarded as the conjugation in the metaplectic cover of $\GL_2(F)$.
    This gives us another representation
    \begin{equation*}
        [\pi^h](g) = \pi(\prescript{h}{}{g})
    \end{equation*}
    of $\widetilde{G}$ on the underlying space $V$ of $\pi$.
    It also gives a map $\prescript{h}{}{\eta}$ from $\prescript{h}{}{K_m}$ to $\C^\times$ by $\prescript{h}{}{\gamma} \mapsto \eta(\gamma)$.
    Then we define $\pi^{\prescript{h}{}{K_m}}_{\prescript{h}{}{\eta}}$ in the obvious manner, i.e.,
    \begin{equation*}
        \pi^{\prescript{h}{}{K_m}}_{\prescript{h}{}{\eta}}
        = \Set{v \in V | \pi(\gamma) v = \prescript{h}{}{\eta}(\gamma) v, \quad \text{for all }\gamma \in \prescript{h}{}{K_m} }.
    \end{equation*}
    It easy to see that $\pi^{\prescript{h}{}{K_m}}_{\prescript{h}{}{\eta}}$ is equal to $(\pi^h)^{K_m}_\eta$.
    It is also easy to check that $\pi$ is $(\prescript{h}{}{N}, \prescript{h}{}{\Psi})$-generic if and only if $\pi^h$ is $(N, \Psi)$-generic, for any character $\Psi$ of $N$, where $\prescript{h}{}{\Psi}$ denotes the character of $\prescript{h}{}{N}$ that sends $\prescript{h}{}{n}$ to $\Psi(n)$.

    Consider the special case $h=\beta$.
    Let us write $\mathcal{R}_\varpi^\pi$ (resp. $\mathcal{U}_{\varpi^2}^\pi$, resp. $\mathcal{Q}_\varpi^\pi$, resp. $\mathcal{V}_{\varpi^2}^\pi$) for $\mathcal{R}_\varpi$ (resp. $\mathcal{U}_{\varpi^2}$, resp. $\mathcal{Q}_\varpi$, resp. $\mathcal{V}_{\varpi^2}$) in order to clarify that it depends not only on $V$ but also $\pi$.
    Since we know from \cite{kub} that
    \begin{equation*}
        \prescript{\beta}{}{\left( \left( \begin{array}{cc}a&b\\ &a^{-1} \end{array} \right), 1 \right)} = \left( \left( \begin{array}{cc}a&\varpi^{-1}b\\ &a^{-1} \end{array} \right), ( a, \varpi )_{F,2} \right),
    \end{equation*}
    by a straightforward calculation, we have
    \begin{align*}
        \mathcal{R}_\varpi^{\pi^\beta}     & = \mathcal{Q}_\varpi^\pi,     \\
        \mathcal{U}_{\varpi^2}^{\pi^\beta} & = \mathcal{V}_{\varpi^2}^\pi,
    \end{align*}
    which imply
    \begin{align*}
        \pi^{L_m, \KUold}_{\prescript{\beta}{}{\eta}} & = \left( \pi^\beta \right)^{K_m, \KUold}_\eta, \\
        \pi^{L_m, \KUnew}_{\prescript{\beta}{}{\eta}} & = \left( \pi^\beta \right)^{K_m, \KUnew}_\eta.
    \end{align*}
    One can also check that $z_\psi(\pi) = z_{\psi_\varpi}(\pi^\beta)$, and that $\prescript{\beta}{}{\eta} = \eta$ as functions on $L_0$.
    The corollary is now an immediate consequence of Theorems \ref{thm:conductor}, \ref{thm:main-nongeneric}, and \ref{thm:main-generic}.
\end{proof}


\end{document}